\documentclass[reqno,11pt,draft]{amsart}
\usepackage[top=2.0cm,bottom=2.0cm,left=3cm,right=3cm]{geometry}
\usepackage{amsthm,amsmath,amssymb,dsfont}
\usepackage{mathrsfs,amsfonts,functan,extarrows,mathtools}

\usepackage[colorlinks,
linkcolor=black,
citecolor=black
]{hyperref}

\usepackage{xcolor}
\usepackage{cite}
\usepackage{indentfirst, latexsym, amssymb, enumerate,amsmath,graphicx}
\usepackage{float}
\usepackage{relsize}
\usepackage{marginnote}
\usepackage{stmaryrd}
\usepackage{esint}
\usepackage{graphicx}
\usepackage{bm}
\usepackage{caption}
\usepackage{subfigure}
\usepackage{cite}
\usepackage{color}
\usepackage{graphicx}
\usepackage{subfigure}
\usepackage{float}
\usepackage{paralist}
\usepackage{indentfirst}
\usepackage{cite}
\allowdisplaybreaks 
\theoremstyle{plain}
\newtheorem{thm}{Theorem}[section]

\newtheorem{lem}[thm]{Lemma}
\newtheorem{prop}[thm]{Proposition}
\newtheorem{rem}{Remark}[section]

\newtheorem{defn}{Definition}[section]

\numberwithin{equation}{section}

\DeclareMathOperator{\dive}{div}

\usepackage{appendix}
\usepackage{xcolor}

\begin{document}

\title[Fujita-Kato solution of the  inhomogeneous MHD system]
{Global Fujita-Kato solutions of the incompressible inhomogeneous magnetohydrodynamic equations}

\author[f.-C. Li]{fucai Li}
\address{Department of Mathematics, Nanjing University, Nanjing
 210093, P. R. China}
\email{fli@nju.edu.cn}

\author[J.-K. Ni]{Jinkai Ni$^*$} \thanks{$^*$\! Corresponding author}
\address{Department of Mathematics, Nanjing University, Nanjing
 210093, P. R. China}
\email{602023210006@smail.nju.edu.cn}

\author[L.-Y. Shou]{Ling-Yun Shou} 
\address{Department of Mathematics, Nanjing Normal University, Nanjing
 210023, P. R. China}
\email{shoulingyun11@gmail.com}

\begin{abstract}
We investigate the incompressible inhomogeneous magnetohydrodynamic  equations in $\mathbb{R}^3$, under the assumptions that the initial density $\rho_0$ is only bounded, and the initial velocity $u_0$ and magnetic field $B_0$ exhibit critical regularities. In particular, the density is allowed to be piecewise constant with jumps. First, we establish the global-in-time well-posedness and large-time behavior of solutions to the Cauchy problem in the case that $\rho_0$ has small variations, and $u_0$ and
$B_0$ are sufficiently small in the critical Besov space $\dot{B}^{3/p-1}_{p,1}$ with $1<p<3$. Moreover, the small variation assumption on $\rho_0$ is no longer required in the case $p=2$. Then, we construct a unique global Fujita-Kato solution under the weaker condition that $u_0$ and $B_0$ are small in $\dot{B}^{1/2}_{2,\infty}$ but may be large in $\dot{H}^{1/2}$. Additionally, we show a general uniqueness result with only bounded and nonnegative density, without assuming the $L^1(0,T;L^{\infty})$ regularity of the velocity. Our study systematically addresses the global solvability of the inhomogeneous  magnetohydrodynamic  equations with rough density in the critical regularity setting.


\end{abstract}



\maketitle

\setcounter{equation}{0}
 \indent \allowdisplaybreaks

\section{Introduction }\label{Sec:intro-resul}

The magnetohydrodynamic (MHD) equations are widely used to describe the evolution of electrically conducting fluids such as plasmas, liquid metals and electrolytes, and have played an important role in understanding various phenomena in geophysics, cosmology and engineering. In these applications, Maxwell’s equations of electromagnetism are coupled with the Navier-Stokes equations with Lorentz force generated by the magnetic field, where the displacement current is typically neglected (see, e.g. \cite{biskamp1,davidson1,laudau1}). Although they share similarities with the Navier-Stokes equations, the MHD equations constitute a coupled and highly interactive system with more complex structures, making the study of their dynamics challenging.

In this paper, we address the global well-posedness and stability of the incompressible inhomogeneous MHD equations  in $\mathbb{R}^3$:
\begin{equation}\label{I-1}
\left\{
\begin{aligned}
& \partial_t \rho+u\cdot\nabla\rho =0, \\
& \rho (\partial_t u+u\cdot \nabla u)-\mu\Delta u+\nabla P=B\cdot\nabla B, \\
& \partial_t B+u\cdot\nabla B-\nu\Delta B=B\cdot\nabla u,\\
& \dive u=\dive B=0,
\end{aligned}\right.
\end{equation}
where $\rho=\rho(t,x)\geq 0$, $u=(u_1,u_2,u_3)(t,x)$, $B=(B_1,B_2,B_3)(t,x)$, and $P=P(x,t)$, stand for
the fluid density, the fluid velocity field, the magnetic field, and the fluid pressure,  respectively, with $(t,x)\in \mathbb{R}_+\times\mathbb{R}^3$.

We consider the Cauchy problem of the MHD system
\eqref{I-1} supplemented with the initial data
\begin{align}
 (\rho,u,B)(0,x)=(\rho_0,u_0,B_0)(x),\quad\quad x\in\mathbb{R}^3.\label{d}
\end{align}

\vspace{3mm}

So far, significant progresses have been made regarding the global existence, regularity and stability of the MHD system \eqref{I-1}. When $\rho$ is a constant, indicating that the fluid is homogeneous, the first pioneering work was studied by Duraut and Lions \cite{DL-ARMA-1972}. They established the local existence and uniqueness of strong solutions in the Sobolev spaces $H^s(\mathbb{R}^d)$ with $s \geq d$ and proved the global existence for small initial data. Subsequently, Sermange and Temam \cite{sermange1} investigated the large-time properties of solutions and established the global well-posedness for large data in two dimensions.
Recent advancements on regularity and singularity criteria can be found in  \cite{wu1,wu2,chenmiaozhang1,hexin1,Zhou0}. 
Concerning the global regularity and large-time asymptotics of solutions on the MHD equations with reduced dissipation or close to non-zero equilibrium states, 
interested readers can
refer to \cite{DT-CPDE-2021,LZ-CPAM-2014,CW-Adv-2011,LXZ-JDE-2015,chemin1,he1} and the references therein.

When the variable density is considered, i.e. the fluid is inhomogeneous, the system becomes more complicated due to the influence of the density. Gerbeau and Le Bris \cite{Gerbeau1} proved the global-in-time existence of finite energy weak solutions  with density-dependent viscosity coefficients and general initial data in a bounded domain (see also the work of Desjardins and Le Bris \cite{Desjardins1}). For the ideal (i.e. inviscid and non-resistive) case, the local well-posedness of classical or strong solutions has been established by many authors, cf., for example,
Schmidt \cite{Schmidt1}, Secchi \cite{Secchi1}, and Zhou and Fan \cite{Zhou1}. In absence of vacuum, Li and Wang \cite{LiWang1} obtained the existence and uniqueness of local strong solutions to the inhomogeneous viscous MHD system with general 
initial data in Sobolev spaces and proved the global existence as well as global strong solutions with small
initial data. They also gave a weak-strong uniqueness result. 
When the vacuum appears, the local well-posedness of strong solutions was investigated by Chen, Tan and Wang \cite{ChenTan1} under some compatibility conditions, and the global existence of strong solutions in the 2D case was obtained by Huang and Wang \cite{HuangWang1} and L\"u, Xu and Zhong \cite{LuXuZhong1}. In the case where the density is only bounded from above and below but without the assumption of small variation, Gui \cite{Ggl-JFA-2014} constructed a unique global solution for the 2D MHD system with variable electrical conductivity under the assumption $u_0, B_0\in H^s$ with $s>0$ (see also the 3D case by Chen-Guo-Zhai \cite{CGZ-KRM-2019} with $u_0, B_0\in H^s$, $s>1/2$).

There are many studies on the problem of solvability for the inhomogeneous MHD system \eqref{I-1} in $\mathbb{R}^d$ ($d\geq2$) in the so-called “critical regularity" setting. The key ingredient, originating  from the work of Fujita and
Kato \cite{fujita1} for the incompressible Navier-Stokes equations with constant density, is that the “optimal” functional spaces for
well-posedness of \eqref{I-1} obey the scaling invariances:
\begin{equation*}
\begin{aligned}
(\rho_{\lambda}, u_{\lambda},B_{\lambda},  P_{\lambda})(t,x)\leadsto (\rho, \lambda^2 u,  \lambda^2 B,  \lambda P)(\lambda^2 t, \lambda x),\label{scaling1}
\end{aligned}
\end{equation*}
and the corresponding scaling invariances for the initial data: 
\begin{equation*}
\begin{aligned}
(\rho_{0,\lambda}, u_{0,\lambda},B_{0,\lambda})(x)\leadsto (\rho_0, \lambda^2 u_0, \lambda^2 B_0)(\lambda x).\label{scaling2}
\end{aligned}
\end{equation*}
There are several possible choices for critical spaces of initial data, for example, $L^{\infty}(\mathbb{R}^d)\times \dot{H}^{-1+d/2}\times \dot{H}^{-1+d/2}$ (Fujita-Kato space) or Besov spaces $\dot{B}^{d/p}_{p,r}\times \dot{B}_{p,r}^{-1+d/p}\times \dot{B}^{-1+d/p}_{p,r}$. 
The global existence of solutions in critical spaces for the  system \eqref{I-1} was firstly established by Abidi  and Hmidi \cite{AH-AMBP-2007} (see also Abidi and Paicu \cite{AP-PRSESA-2008}). They assumed  that the initial density is close to some constant state in $\dot{B}^{d/p}_{p,1}$ and that the
initial velocity and magnetic fields are suitably small in $\dot{B}^{-1+d/p}_{p,1}$. For more recent advances on the well-posedness of solutions to \eqref{I-1} in critical spaces,  interested readers can see \cite{CLX-DCDS-2016,Ggl-JFA-2014,ZY-JDE-2017,Tan1} and the references therein.

When the effect of the magnetic field is neglected, i.e. $B = 0$, the system \eqref{I-1} reduces to the well-known incompressible inhomogeneous  Navier-Stokes system:
\begin{equation}\label{I-2}
\left\{
\begin{aligned}
& \partial_t \rho + u \cdot \nabla \rho = 0, \\
& \rho (\partial_t u + u \cdot \nabla u) - \mu \Delta u + \nabla P = 0, \\
& \dive u = 0,
\end{aligned}
\right.
\end{equation}
which arises in the theory of geophysical flows where the fluid is incompressible but has variable density, as in oceans or rivers. The system \eqref{I-2} has been studied extensively with many significant results. The existence of smooth solutions away from vacuum was established in, e.g. \cite{LS-ZMSLOMIS-1975,Dr-2004-adv}, whereas the theory of global weak solutions with finite energy was developed in the book by Lions \cite{Lions-book-1996} (see also earlier works \cite{Kazhihov,simon}). When the vacuum is excluded, and the perturbation of the initial density $\rho_0$ and the initial velocity $u_0$ belong to the critical spaces $\dot{B}^{d/p}_{p,1}$ and $\dot{B}^{d/p-1}_{p,1}$, respectively,  the well-posedness issue in $\mathbb{R}^d$ ($d\geq2$) has been studied in many works \cite{Ah-RMI-2007,DM-CPAM-2012, AP-AIF-2007, Dr-PRSES-2003, HPZ-ARMA-2013,AG-ARMA-2021,Xh-JDE-2022}. When the vacuum appears, the corresponding global well-posedness of strong solutions was recently established by \cite{craig1,he1,heli1,LuXuZhong1}. For rough and discontinuous density that only has lower and upper bounds, Zhang \cite{Zp-Adv-2020} established the global existence of Fujita-Kato type solutions to the system \eqref{I-2} under the condition that the initial velocity $u_0$ is small in $\dot{B}_{2,1}^{1/2}$. Subsequently, Danchin and Wang \cite{DW-CMMP-2023} provided the uniqueness of solutions constructed in \cite{Zp-Adv-2020} and showed a new global existence result for $u_0$ in $\dot{B}^{-1+d/p}_{p,1}$ and $\rho_0$ close to some positive constant. Very recently, Hao et al. \cite{HSWZ-arXiv-2024-08} established the global existence and uniqueness of weak solutions for the  system \eqref{I-2} in $\mathbb{R}^2$ with bounded initial density away from vacuum and initial velocity in the energy space $L^2$. In particular, they solved Lions’ open problem \cite{Lions-book-1996}: the uniqueness of weak solutions for \eqref{I-2} in two dimensions. Moreover, the authors in \cite{HSWZ-arXiv-2024-08} constructed a unique global solution to \eqref{I-2} in $\mathbb{R}^3$ by assuming only nonnegative and bounded density $\rho_0$ and small initial velocity $u_0$ in $\dot{H}^{1/2}$. Abidi, Gui, and Zhang \cite{AGZ-arXiv-2024} further improved this result, allowing the velocity $u_0$ to be large in $\dot{H}^{1/2}$ and  small in $\dot{B}^{1/2}_{2,\infty}$.

\vspace{2mm}

It is interesting to investigate whether the incompressible inhomogeneous MHD system \eqref{I-1} is well-posed for discontinuous density. In this case, the piecewise constant density with jumps for fluid mixtures is allowed. However, to the best of our knowledge, there is no result concerning the global solvability issue of the inhomogeneous MHD system \eqref{I-1} for discontinuous and large density.

In this paper, our aim is to study the global dynamics of Fujita-Kato type solutions to the Cauchy problem of the 3D inhomogeneous MHD system \eqref{I-1} with rough density in critical Besov spaces. 
Our main results can be summarized as follows:

\begin{itemize}
\item The global existence and large-time stability of solutions under the assumption conditions that $\|\rho_0-1\|_{L^{\infty}}$ and $\mu^{-1}\|(u_0,B_0)\|_{\dot{B}^{-1+3/p}_{p,1}}$ with $1<p<3$ are sufficiently small. 

\item A global well-posedness result under the assumption that $\rho_0$ is merely bounded and may have a large variation, and $\mu^{-1}\|(u_0,B_0)\|_{\dot{B}^{1/2}_{2,1}}$ is suitable small. 

\item  The existence of a unique global Fujita-Kato solution under the case that $\rho_0$ and $\|(u_0,B_0)\|_{\dot{H}^{1/2}_{2,1}}$ are only bounded, with a refined smallness of $\mu^{-1}\|(u_0,B_0)\|_{\dot{B}^{1/2}_{2,\infty}}$.  

\item A general uniqueness result  within the critical regularity framework, allowing for bounded and non-negative density, without requiring $L^1(0,T;L^{\infty})$-regularity of the fluid velocity. 
\end{itemize}

\vspace{1mm}

More precisely, our first result states the global existence of the Cauchy problem \eqref{I-1}--\eqref{d} in the $L^p$ framework with small density variations, which reads

\begin{thm}\label{T2.3}
Let $1<p<3$ and $1<q
<\infty$ such that $3/p+2/q=3$. Assume that $(u_0,B_0)$ fulfills $\dive u_0=\dive B_0=0$ and $ u_0,B_0\in\dot{B}_{p,1}^{-1+3/p}$, and
there exists a positive constant $\varepsilon_1>0$ depending only on $p$ and 
the ratio ${\nu}/{\mu}$ such that 
\begin{align}\label{A1.13}
\|\rho_0-1\|_{L^\infty}\leq\varepsilon_1,\quad\|(u_0,B_0)\|_{\dot{B}_{p,1}^{-1+3/p}}\leq\varepsilon_1\mu.  
\end{align}
Then the Cauchy problem \eqref{I-1}--\eqref{d} has a global-in-time solution $(\rho,u,B,\nabla P)$ enjoying the following properties:
\begin{itemize}
\item $\rho$ satisfies
\begin{align}
\|(\rho-1)(t)\|_{L^{\infty}}=\|\rho_{0}-1\|_{L^{\infty}},\quad\quad t>0;\label{1.7}
\end{align}
\item $u,B\in C(\mathbb{R}_{+};\dot{B}^{-1+3/p}_{p,1})${\rm;}
\item For any $3<m<\infty$ and $q<s<\infty$ satisfy $3/m+2/s=1$, the following uniform bound holds for all $t>0${\rm:}
\begin{align}
&\|(u,B)\|_{L^{\infty}(\mathbb{R_{+}};\dot{B}^{-1+3/p}_{p,1})}+\|(\nabla^2 u,\nabla^2 B,\nabla P)\|_{ L^{q,1}(\mathbb{R_{+}};L^{p})}+\|(\nabla u,\nabla B)\|_{ L^{1}(\mathbb{R_{+}};L^{\infty})}\nonumber\\
&\quad+
\|(u,B)\|_{ L^{2}(\mathbb{R_{+}};L^{\infty})}+\|(u,B)\|_{L^{s,1}(\mathbb{R_{+}};{L^m})}+\|(t\nabla P, \partial_t (t u), \partial_t (t B)\|_{ L^{s,1}(\mathbb{R_{+}};L^{m})}\nonumber\\
&\quad+\|t(D_t u, tD_t{B})\|_{ L^{s,1}(\mathbb{R_{+}};L^{p})}+\|t(\nabla^2 D_t u, t\nabla^2 D_t{B})\|_{ L^{s,1}(\mathbb{R_{+}};L^{p})}\nonumber\\
&\quad\quad\quad\leq C\|(u_{0},B_0)\|_{\dot{B}_{p,1}^{-1+3/p}}\nonumber,
\end{align}
furthermore, we have the following decay rates,
\begin{equation}\label{1.9}
\left\{
\begin{aligned}
&\|(\nabla u,\nabla B)(t)\|_{\dot{B}^{-1+3/m}_{m,1}}\leq C\|(u_{0},B_0)\|_{\dot{B}_{p,1}^{-1+3/p}} t^{-1/2} ,\\
&\|(\nabla^2 u,\nabla^2 B)(t)\|_{\dot{B}^{-1+3/m}_{m,1}}\leq C\|(u_{0},B_0)\|_{\dot{B}_{p,1}^{-1+3/p}} t^{-1},\\
&\|(D_t u, D_t B)(t)\|_{\dot{B}^{-1+3/p}_{p,1}}\leq C\|(u_{0},B_0)\|_{\dot{B}_{p,1}^{-1+3/p}} t^{-1}.
\end{aligned}
\right.
\end{equation}
\end{itemize}
Here   $C>0$ is a constant depending on $\rho_0, \mu$ and $\nu$.   
In the case $1<p\leq 2$, the solution is unique.
\end{thm}

Next, we prove a global existence and uniqueness result in the $L^2$ framework, assuming that the density is only bounded from above and below.

\begin{thm}\label{T2.1} 
Assume that the initial data $(\rho_0, u_0, B_0)$ satisfies $\dive u_0=\dive B_0=0$,
\begin{align}\label{A1.2}
0<\inf\limits_{x\in\mathbb{R}^3}\rho_0(x)\leq\rho_0(x)\leq \sup\limits_{x\in\mathbb{R}^3}\rho_0(x)<\infty\quad\text{and}\quad u_0, B_0\in \dot{B}^{1/2}_{2,1}.
\end{align}
Then there exists a constant $\varepsilon_{2}>0$ depending only on $\rho_0$ and ${\nu}/{\mu}$ such that if
\begin{gather}\label{A1.3}
\|(u_0,B_0)\|_{\dot{B}^{1/2}_{2,1}}\leq\varepsilon_2\mu,    
\end{gather}
the system \eqref{I-1} has a global unique weak solution $(\rho,u,B,\nabla P)$  satisfying the following properties:
\begin{itemize}
\item $\rho$ fulfills
\begin{align}\label{A1.4}
0<\inf\limits_{x\in\mathbb{R}^3}\rho_0(x)\leq\rho(t,x)\leq \sup\limits_{x\in\mathbb{R}^3}\rho_0(x)<\infty,\quad (t,x)\in\mathbb{R}_+\times\mathbb{R}^3;  
\end{align}

\item $u, B\in$ $C(\mathbb{R}_+,\dot{B}^{1/2}_{2,1})${\rm;}

\item It holds that
\begin{align}
&\|(u,B)\|_{\widetilde{L}^{\infty}(\mathbb{R}_{+};\dot{B}^{1/2}_{2,1})} +\|(u,B)\|_{\widetilde{L}^{2}(\mathbb{R}_{+};\dot{B}^{3/2}_{2,1})}\nonumber\\
&\quad+\|{t^{1/2}}(\partial_t u,\partial_t B)\|_{\widetilde{L}^{2}(\mathbb{R}_{+};\dot{B}^{1/2}_{2,1})}+\|t(D_{t}u,D_{t}B)\|_{\widetilde{L}^{2}(\mathbb{R}_{+};\dot{B}^{3/2}_{2,1})}\nonumber\\
&\quad+{\|(\nabla u,\nabla B)\|_{L^4(\mathbb{R}_{+};L^2)}}+{{\|(\nabla u,\nabla B)\|_{L^1(\mathbb{R}_+;L^{\infty})}}+\|t^{1/2}(\nabla u,\nabla B)\|_{L^2(\mathbb{R}_+;L^{\infty})}}\nonumber\\
&\quad+\|t^{1/4} (\nabla^2 u, \nabla^2 B, \nabla P, \partial_t u,\partial_t B)\|_{L^{2}(\mathbb{R}_+;L^2)} +{\|t^{3/4}(\nabla^{2}u,\nabla^2 B, \nabla P)\|_{L^{2}(\mathbb{R}_+;L^{6})} }\nonumber\\
&\quad+{{\| t^{3/4}(\nabla\partial_{t}u,\nabla\partial_{t}B,\nabla D_t u,\nabla  D_t B)\|_{L^2(\mathbb{R}_{+};L^2)}}} \nonumber\\
&\quad\quad\quad\leq C\|(u_{0},B_0)\|_{\dot{B}^{1/2}_{2,1}}, \label{T2.1:e1}   
\end{align}

\item For all $t>0$, the following decay estimates hold{\rm:}
\begin{equation}\label{T2.1:e2}   
\left\{
    \begin{aligned}
         &\|(\nabla u,\nabla B)(t)\|_{\dot{B}^{0}_{2,1}}\leq C\|(u_{0},B_0)\|_{\dot{B}^{1/2}_{2,1}} t^{-\frac{1}{4}},\\
         &\|(\nabla u,\nabla B)(t)\|_{\dot{B}^{1/2}_{2,1}}\leq C\|(u_{0},B_0)\|_{\dot{B}^{1/2}_{2,1}} t^{-\frac{1}{2}},\\
         &\|(\nabla^2 u, \nabla^2 B,\partial_t u,\partial_t B)(t)\|_{L^2}\leq  {C \|(u_{0},B_0)\|_{\dot{B}^{1/2}_{2,1}} t^{-3/4}},\\
         &\|(\partial_t u,\partial_t B)(t)\|_{\dot{B}^{1/2}_{2,1}}\leq C \|(u_{0},B_0)\|_{\dot{B}^{1/2}_{2,1}} t^{-1 }.
    \end{aligned}
    \right.
\end{equation}
\end{itemize}
Here $C>0$ is a constant depending on $\rho_0, \mu$ and $\nu$. 

\end{thm}

Moreover, we state
a result on Fujita-Kato type solutions to the system \eqref{I-1}.

\begin{thm}\label{T2.2}
Assume that the initial data $(\rho_0, u_0, B_0)$ fulfills $\dive u_0=\dive B_0=0$, 
\begin{gather}\label{A1.6}
0<\inf\limits_{x\in\mathbb{R}^3}\rho_0(x)\leq\rho_0(x)\leq \sup\limits_{x\in\mathbb{R}^3}\rho_0(x)<\infty\quad\text{and}\quad  u_0, B_0\in \dot{H}^{1/2}.
\end{gather}
Then there exists a constant $\varepsilon_{3}>0$ depending only on $\rho_0$ 
and  ${\nu}/{\mu}$  such that if
\begin{gather}\label{A1.7}
\|(u_0,B_0)\|_{\dot{B}^{1/2}_{2,\infty}}\leq\varepsilon_3\mu,    
\end{gather}
then the system \eqref{I-1} has a unique global solution $(\rho,u,B,\nabla P)$. In addition, the following properties hold{\rm:}
\begin{itemize}

\item $\rho$ satisfies
\begin{align}\label{A1.8}
0<\inf\limits_{x\in\mathbb{R}^3}\rho_0(x)\leq\rho(t,x)\leq \sup\limits_{x\in\mathbb{R}^3}\rho_0(x)<\infty,\quad  (t,x)\in\mathbb{R}_+\times\mathbb{R}^3; 
\end{align}

\item $u,B\in C(\mathbb{R}_+,\dot{H}^{1/2})${\rm;}

\item It holds that
\begin{align}
&\|(u,B)\|_{L^2(\mathbb{R}_{+};L^{\infty})}+\|(\nabla u,\nabla B)\|_{L^{4}(\mathbb{R}_+;L^{2})}+\|t^{1/2}(\nabla u,\nabla B)\|_{L^{2}(\mathbb{R}_+;L^{\infty})} \nonumber\\
&\quad+\|t^{1/4} (\nabla^2 u, \nabla^2 B, \nabla P, \partial_t u,\partial_t B)\|_{L^{2}(\mathbb{R}_+;L^2)}+\|t^{3/4}(\nabla^{2}u,\nabla^2 B, \nabla P)\|_{L^{2}(\mathbb{R}_+;L^{6})}\nonumber\\
&\quad+\|t^{3/4}(\nabla\partial_{t}u,\nabla\partial_{t}B,\nabla D_{t}u, \nabla D_{t}B)\|_{L^{2}(\mathbb{R}_+;L^{2})}\nonumber\\
&\quad\quad\quad\leq C\|(u_{0},B_0)\|_{\dot{H}^{1/2}};\label{1.13}
\end{align} 

\item For any $2\leq r\leq \infty,$ $(u,B)$ satisfies
\begin{align}
&\|(u,B)\|_{\widetilde{L}^{\infty}(\mathbb{R}_+;\dot{B}_{2,r}^{1/2})}+\|(u,B)\|_{\widetilde{L}^{2}(\mathbb{R}_+;\dot{B}_{2,r}^{3/2})}\nonumber\\
&\quad +\|t^{1/2}(u,B)\|_{\widetilde{L}^{\infty}(\mathbb{R}_+;\dot{B}_{2,r}^{3/2})}+\|t^{1/2}(\partial_{t}u, \partial_{t}B)\|_{\widetilde{L}^{2}(\mathbb{R}_+;\dot{B}_{2,r}^{1/2})}\leq C\|(u_{0},B_0)\|_{\dot{B}_{2,r}^{1/2}};\label{1.14}
\end{align}

\item 

The following decay estimates hold{\rm:}
\begin{equation}\label{1.22}
    \left\{
    \begin{aligned}
    &\|(\nabla u,\nabla B)(t)\|_{L^2}\leq C \|(u_{0},B_0)\|_{\dot{H}^{1/2}}t^{-\frac{1}{4}},\\
    &\|(\nabla u,\nabla B)(t)\|_{\dot{H}^{1/2}}\leq C \|(u_{0},B_0)\|_{\dot{H}^{1/2}}t^{-\frac{1}{2}},\\
    &\|(\partial_{t}u,\partial_t B,\nabla^{2}u,\nabla^2 B,D_{t}u)(t)\|_{L^{2}}\leq C \|(u_{0},B_0)\|_{\dot{H}^{1/2}} t^{-3/4}.
    \end{aligned}
    \right.
\end{equation}

\end{itemize}
Here, $C>0$ is a constant depending on $\rho_0, \mu$ and $\nu$.
\end{thm}

Finally, our uniqueness result is stated as follows.

\begin{thm}\label{thmunique}
Let $(\rho_{1},u_{1},B_{1},\nabla P_{1})$ and $(\rho_{2},u_{2},B_{2},\nabla P_{2})$ be two solutions to the system \eqref{I-1} on $[0,T]\times\mathbb{R}^3 $ with $0<T<\infty$ supplemented with the same initial data. In addition, we assume that
\begin{equation}\label{1.23}
\left\{
\begin{aligned}
&0\leq \rho_1,\rho_2\in L^{\infty}(0,T;L^{\infty}),\\
&\nabla u_2,\, \nabla B_2,\, \in L^4(0,T;L^2),\\
&t^{1/2}\nabla u_2 \, \in L^2(0,T;L^{\infty}),\\
&t^{3/4} D_t u_2\,\in L^2(0,T;L^2),\\
&t^{-3/4}(\rho_1-\rho_2)\in L^{\infty}(0,T;\dot{W}^{-1,3}),\\
&\sqrt{\rho_1}(u_1-u_2),\, B_1-B_2\,\in L^{\infty}(0,T;L^2),\\
&\nabla (u_1-u_2),\, \nabla (B_1-B_2)\in  L^{2}(0,T;L^2).
\end{aligned}
\right.
\end{equation}
Then, $(\rho_{1},u_{1},B_{1},\nabla P_{1})(t,x)=(\rho_{2},u_{2},B_{2},\nabla P_{2})(t,x)$ holds on $[0,T]\times\mathbb{R}^3$.
\end{thm}

\begin{rem}

A few comments to our results are listed below.

\begin{itemize}

\item For global solutions given in Theorems \ref{T2.3}--\ref{T2.2}, if we assume  additionally that $u_0, B_0\in L^2$, then one has the basic energy equality:
\begin{align*}
&\frac{1}{2}\|\sqrt{\rho} u(t)\|_{L^2}^2+\frac{1}{2}\|B(t)\|_{L^2}^2+\int_0^t\big( \mu \|\nabla u(t')\|_{L^2}^2+\nu\|\nabla B(t')\|_{L^2}^2\big)\,{\rm d}t'\\
&\quad =\frac{1}{2}\|\sqrt{\rho_0} u_0\|_{L^2}^2+\frac{1}{2}\|B_0\|_{L^2}^2.
\end{align*}
Note that this additional condition will ensure the uniqueness of the solutions stated in Theorem \ref{T2.3} in the case $2<p<3$.

\item When $\nu\sim \mu$, we can relax the smallness of the corresponding norms for $(u_0,B_0)$ if $\mu$ is sufficiently large.

\item In Theorem \ref{thmunique},  the regularities on $(\rho_2,u_2,B_2,\nabla P_2)$ are assumed to be critical, but the information of $u_1$ and $B_1$ is not necessary to be the same. Under the assumption that the basic energy inequality holds, Theorem \ref{thmunique} in fact provides a “weak-strong" uniqueness result{\rm:} If $(\rho_1,u_1,B_1,\nabla P_1)$ is a finite-energy weak solution to \eqref{I-1}, and  $(\rho_2,u_2,B_2,\nabla P_2)$ is a solution with higher regularity required in \eqref{1.23}, subject to the same initial data, then the two solutions coincide with each other. In fact, we shall prove a more general stability result, see Proposition \ref{propunique} below. 

\end{itemize}

\end{rem}

In what follows, we will give some comments on the proofs of Theorems \ref{T2.3}--\ref{thmunique}. 

Motivated by \cite{DW-CMMP-2023},  our first result is to establish the global existence of solutions to the system \eqref{I-1} for $\rho_0$ with small variation and $u_0, B_0$ in the critical $L^p$ framework.  To achieve this, we first decompose \eqref{I-1} into two auxiliary systems \eqref{A5.5} and \eqref{A5.7},
which play a key role in deriving uniform {\emph{a priori}} estimates. In Propositions \ref{P5.1}--\ref{P5.3}, we extend the maximal regularity estimates developed in \cite{DMT-JEE-2021,DW-CMMP-2023} to the incompressible inhomogeneous  MHD equations setting. The main challenge in Proposition
\ref{P5.3} lies in estimating the nonlinear term $B\cdot \nabla B$ caused by the interaction of the magnetic and fluid  (for example, see the estimates of $\|D_t(B\cdot\nabla B)\|_{{L^{q,1}_T(L^{p})}}$ in \eqref{A5.50}).  Compared with \cite[Proposition 4.3]{DW-CMMP-2023}, where only 
bilinear terms are involved, we need to handle the additional trilinear term $u\cdot\nabla(B\cdot\nabla B)$. 
 To overcome this difficulty, we make full use of the elaborate estimates based on interpolation and embedding inequalities in Besov and Lorentz spaces. 
Building on uniform bounds and a compactness argument, we obtain the existence of a solution $(\rho,u,B,\nabla P)$ to the problem \eqref{I-1}--\eqref{d}.

To prove Theorem \ref{T2.1} concerning the global well-posedness for large density, inspired by  \cite{Zp-Adv-2020}, we decompose $(u,B,\nabla P)=\sum\limits_{j\in\mathbb{Z}}(u_j,B_j,\nabla P_j)$ with $(u_j,B_j,\nabla P_j)$ solving the following localized quasi-linear system with convection:
\begin{equation*}
\left\{
\begin{aligned}
& \rho(\partial_tu_j+u\cdot\nabla u_j)-\Delta u_j+\nabla P_j=B\cdot\nabla B_j, \\
& \partial_t B_j+u\cdot \nabla B_j-\Delta B_j=B\cdot\nabla u_j, \\
& \dive u_j=\dive B_j=0,\\
& u_j|_{t=0}=\dot{\Delta}_j u_0, \quad B_j|_{t=0}=\dot{\Delta}_j B_0.
\end{aligned}\right.
\end{equation*}
By performing functional inequalities for $(u_j,B_j,\nabla P_j)$ and their derivatives with suitable time weights for initial layer analysis, we establish the localized energy estimates of $(u_j,B_j,\nabla P_j)$ (see Propositions \ref{P3.1}--\ref{P3.6}), which lead to the desired uniform bounds of $(u,B,\nabla P)$ by carrying out an interpolation argument. Using these bounds, one can extend the local approximate solutions to the global ones and show the convergence to a desired global solution.

Furthermore, to show Theorem \ref{T2.2}, by using the refined embeddings and estimates of products in Lorentz spaces in the same spirit of \cite{Ggl-JFA-2014}, we succeed in deriving {\emph{a priori}} estimates under the refined smallness of $\|(u_0,B_0)\|_{\dot{B}^{1/2}_{2,\infty}}$ (cf. Propositions \ref{P4.1}--\ref{P4.3}).

Finally, Theorem \ref{thmunique} is based on a more general stability result, i.e. Proposition \ref{propunique}. The main challenge arises in the lack of $L_t^1L^{\infty}_x$--regularity for $\nabla u$. Here we adapt the method of \cite{HSWZ-arXiv-2024-08} on the backward transport equation. On the other hand, we also handle some difficult nonlinear terms coming from the interactions from velocity and magnetic fields, which rely on the additional $L_t^{4}L^2_x$ bound of $\nabla B_2$ and a cancellation property. With these observations, we establish the stability estimates of the incompressible inhomogeneous MHD system \eqref{I-1}.

\vspace{2mm}

The rest of this paper is organized as follows. 
In Section 2, we offer some preliminaries
with regard to a series of notations and definitions.
In Section 3, we establish some uniform estimates in Propositions
\ref{P5.1}--\ref{P5.3} and give the proof of Theorem \ref{T2.3}.
In Section 4, we present the proof of Theorem \ref{T2.1} based on
Propositions
\ref{P3.1}--\ref{P3.6}.
In Section 5, we show the proof of Theorem \ref{T2.2}.
Section 6 is devoted to a stability result, which in particular implies Theorem \ref{thmunique}. 
Finally, we show some practical lemmas 
in the Appendix.

\section{Preliminaries}
In this section, we introduce  notations, definitions and some basic facts, which shall be used
frequently  throughout this paper. 

  The letter $C$ 
is a generic positive constant, which may change from line to line.
And $a\lesssim b$ means $a\leq Cb$. We denote $(d_j)_{j\in\mathbb{Z}}$ $({\rm resp}.\, (c_{j,r})_{j\in\mathbb{Z}}$) to be a generic element of $\ell^1(\mathbb{Z})$ $({\rm resp}.\,\ell^r(\mathbb{Z})$). 
For two operators $\mathcal{T}_1,\mathcal{T}_2$, we set $[\mathcal{T}_1,\mathcal{T}_2]:=\mathcal{T}_1\mathcal{T}_2-\mathcal{T}_2\mathcal{T}_1$ as the commutator of $\mathcal{T}_1$ and $\mathcal{T}_2$.
We denote $D_{t}:=\partial_{t}+u\cdot\nabla$  the standard material derivative.

 For a Banach space $X$ and an interval $I$ of $\mathbb{R}$, we define by $C(I;X)$ 
the set of continuous functions on $I$ with values in $X$, and $L^q(I;X)$ stands for the set of measurable functions on $I$ with values in $X$. If $I=(0,T)$, we denote its norm as $\|\cdot\|_{L^q_T(X)}$ or $\|\cdot\|_{L^q(0,T;X)}$. 
Besides, we set $\|(g,h)\|_{X}:=\|g\|_{X}+\|h\|_{X}$ for the Banach
space $X$, where $g=g(x)$ and $h=h(x)$ belong to $X$.

Let us recall some notations with regard to
Besov spaces on $\mathbb{R}^3$ (see \cite{BCD-Book-2011}). 
Define
\begin{align*}
\dot{\Delta}_{j}u:{=}\mathcal{F}^{-1}(\varphi(2^{-j}|\xi|)\widehat{u}),\quad S_{j}u:{=}\mathcal{F}^{-1}(\chi(2^{-j}|\xi|)\widehat{u}),    
\end{align*}
where $\mathcal{F}u=\widehat{u}$ represents the Fourier transform of the
tempered distribution $u$, and  $\varphi(\xi)$ and $\chi(\xi)$ are smooth functions satisfying
\begin{gather*}
 \mathrm{Supp} \, \varphi\subset\Big\{\xi\in\mathbb{R}: \frac34\leq|\xi|\leq\frac83\Big\}\quad\mathrm{for}\quad\forall\xi>0,\quad \sum_{j\in\mathbb{Z}}\varphi(2^{-j}\xi)=1,\\
 \mathrm{Supp}\, \chi\subset\Big\{\xi \in\mathbb{R}: |\xi|\leq \frac43\Big\}\quad\mathrm{for}\quad\forall\xi\in\mathbb{R}, \quad \chi(\xi)+\sum_{j\geq0}\varphi(2^{-j}\xi)=1.    
\end{gather*}
The homogeneous Littwood-Paley decomposition of $u$ reads
$$
u=\sum_{j\in\mathbb{Z}}\dot{\Delta}_j u.
$$
The above equality holds true for any tempered distribution $u$ in the set
$$
\mathcal{S}_{h}^{\prime}(\mathbb{R}^{3}):=\Big\{u\in \mathcal{S}^{\prime}(\mathbb{R}^{3})\,:\, \lim\limits_{j \to -\infty} \|S_j u\|_{L^{\infty}}=0 \Big\}.
$$

\vspace{2mm}

Below, we give some definitions of Besov spaces and Chemin-Lernel type spaces. For more details on Besov spaces and related properties, interested readers can refer to \cite{BCD-Book-2011}.

\begin{defn}{\rm(\!\!\cite{BCD-Book-2011})}
Let $1\leq p,r\leq \infty$, and $s\in\mathbb{R}$.  We define  the homogeneous Besov norm $\|\cdot\|_{\dot{B}^{s}_{p,r}}$ on $\mathbb{R}^3$ as 
\begin{align*}
\|u\|_{\dot{B}_{p,r}^{s}}:=\Big\|\Big(2^{js}\|\dot{\Delta}_{j}u\|_{L^{p}}\Big)_{j\in\mathbb{Z}}\Big\|_{\ell^{r}(\mathbb{Z})}.   
\end{align*}
 Moreover,

\begin{itemize}
\item  If $s<3/p$ {\rm(}or $s
=3/p$ and $r=1${\rm)}, we denote $\dot{B}_{p,r}^{s}:=\Big\{u\in\mathcal{S}_{h}^{\prime}(\mathbb{R}^{3}):\|u\|_{\dot{B}_{p,r}^{s}}<\infty\Big\}$.

\item  If $k\in \mathbb{N}$ and if $s<3/p+k+1$ {\rm(}or $s=3/p+k+1$ and $r=1${\rm)}, we denote $\dot{B}_{p,r}^{s}$ as the subset
of $u\in\mathcal{S}_{h}^{\prime}(\mathbb{R}^3)$ such that $\partial^{\beta} u$ belongs to $\dot{B}_{p,r}^{s-k} $ while $|\beta|=k$.

\item  In particular, the Besov space $\dot{B}_{2,2}^{s}$ is consistent with the homogeneou Sobolev spaces $\dot{H}^{s}$.

\end{itemize}

\end{defn}


\begin{defn}{\rm(\!\!\cite{BCD-Book-2011})}
Let $s\in\mathbb{R}$, $1\leq p,q,r\leq\infty$,
and $0<T\leq \infty$. We denote $\widetilde{L}_{T}^{q}(\dot{B}_{p,r}^{s})$
with $C([0,T],\mathcal{S}_h(\mathbb{R}^3))$ through the norm
\begin{align*}
\|f\|_{\widetilde{L}_{T}^{q}(\dot{B}_{p,r}^{s})}:=\Big\|\Big(2^{js}\|\dot{\Delta}_{j}u\|_{L^q_T(L^{p})}
\Big)_{j\in\mathbb{Z}}\Big\|_{\ell^{r}(\mathbb{Z})}<\infty.    
\end{align*} 
Due to Minkowski's inequality, we have
$$
\|f\|_{\widetilde{L}_{T}^{\infty}(\dot{B}_{p,1}^{s})}\geq \|f\|_{L_{T}^{\infty}(\dot{B}_{p,1}^{s})}\quad\text{and}\quad \|f\|_{\widetilde{L}_{T}^{1}(\dot{B}_{p,1}^{s})}= \|f\|_{L_{T}^{1}(\dot{B}_{p,1}^{s})}.
$$
\end{defn}

We also recall the definition of Lorentz spaces (see \cite{Or-DM-1963,LR-Book-2002,Grafakos1}).

\begin{defn}{\rm(\!\! Lorentz spaces)} 
Let $1\leq p,q\leq\infty$, and assume that $f$ is a measurable function on  a measure space $(X,\mu^*)$. We define its non-increasing rearrangement by $f^*$ as
\begin{align*}
f^*(\tau):=\inf\big\{s\geq0:|\{x\in X, d_f(s)\leq\tau         \big\}.  
\end{align*}
Here, $d_f(s)$ is given by
$$
d_f(\tau)=\mu^*\big(\{x\in X: |f(x)|>\tau\}\big).
$$
Then $f$ belongs to the Lorentz space $L^{p,q}(X)$ only if
\begin{align*}
\|f\|_{L^{p,q}}:=\begin{cases}\Big(\int_0^\infty\big(t^{\frac{1}{p}}f^*(t)\big)^q\frac{{\rm d}t}{t}\Big)^{\frac{1}{q}}<\infty,&q<\infty,\\
\sup\limits_{t>0}\bigl(t^{\frac{1}{p}}f^*(t)\bigr)<\infty,&q=\infty.\end{cases}    
\end{align*}  
Alternatively, the Lorentz space $L^{p,q}(X)$ can be defined by the real interpolation between the Lebesgue spaces:
\begin{align*}
L^{p,q}(X)=\big(L^{p_1}(X),L^{p_2}(X)\big)_{\theta,q},
\end{align*}
with $1\leq p_1<p<p_2\leq\infty, 0<\theta<1, 1\leq q\leq \infty$ such that $1/p=(1-\theta)/p_1+\theta/p_2$.
\end{defn}

\section{Proof of Theorem \ref{T2.3}}
This section is devoted to the proof of Theorem \ref{T2.3} concerning the global existence in the critical $L^p$ framework. Before showing Theorem \ref{T2.3},
we always use the fundamental fact: the solution $(\rho,u,B,\nabla P)$ solves the MHD system \eqref{I-1} with the coefficient $\mu$ and $\nu$ if and only if the
rescaled solution
\begin{equation}\label{2.1}
\begin{aligned}
    (\widetilde{\rho},\widetilde{u},\widetilde{B},\nabla\widetilde{P})(t,x):=\Big(\rho,\frac{1}{\mu}u,\frac{1}{\mu}B,\frac{1}{\mu^2}\nabla P\Big)\Big(\frac{1}{\mu}t,x\Big)
\end{aligned}
\end{equation}
fulfills  \eqref{I-1} with the coefficients $1$ and ${\nu}/{\mu}$. Therefore, one can perform the change of unknowns \eqref{2.1} and the change of initial data \begin{align}
  (\widetilde{\rho}_0,\widetilde{u}_0,\widetilde{B}_0)(x):=\Big(\rho_0,\frac{1}{\mu}u_0,\frac{1}{\mu}B_0\Big)(x)  \label{2.2}
\end{align}
to reduce the proof to the case $\mu=1$ due to the scaling invariances of the homogeneous Besov norms. More precisely, once we obtain the {\emph{a priori}} bounds under the assumption \eqref{A1.2}--\eqref{A1.3} for $(\widetilde{\rho}_0,\widetilde{u}_0,\widetilde{B}_0)$ with $\mu=1$, one can recover the dependence of $\mu$ in \eqref{A1.3} for the original data $(\rho_0,u_0,B_0)$ (see a similar setting in \cite{DT-CPDE-2021}). To simplify the notation, below we take the magnetic resistivity $\nu=\mu=1$.

 The proof of Theorem \ref{T2.3} relies on the {\emph{a priori}} estimates which will be stated in Propositions \ref{P5.1}--\ref{P5.3}.  
\begin{prop}\label{P5.1}
Let $(\rho,u,B,\nabla P)$ be a smooth solution of \eqref{I-1} in $[0,T^*)\times\mathbb{R}^{3}$ with some  $T^*>0$. There exists a generic constant $\varepsilon_1>0$ such that if
\begin{align}\label{A5.1}
\sup\limits_{t\in[0,T^*)} \|(\rho-1)(t)\|_{L^\infty}+\|(u_0,B_0)\|_{\dot{B}_{p,1}^{-1+3/p}}<\varepsilon_1, 
\end{align}
then for any  $1<m,p,q,s<\infty$ satisfying $p<m<\infty$, $q<s<\infty$,
and
\begin{gather*}
3/p+2/q=3,\quad 3/m+2/s=1,     
\end{gather*}
it holds that, for any $T\in(0,T^*)$, there exists a universal constant $C$ such that
\begin{align}\label{A5.3}
&\|(u,B)\|_{L^{\infty}_T(\dot{B}_{p,1}^{-1+3/p})}+\|(u,B)\|_{L^{s,1}_T(L^{m})}+\|(u,B)\|_{L^{2}_T(L^{\infty})}\nonumber\\
&\quad +\|( D_t u,D_t B,\partial_t u,\partial_t B)\|_{L^{q,1}_T(L^{p})}+\|(\nabla^{2}u,\nabla^2 B,\nabla P)\|_{L^{q,1}_T(L^{p})}\leq C\|(u_{0},B_0)\|_{\dot{B}_{p,1}^{-1+3/p}}.
\end{align}
\end{prop}
\begin{proof}
First, employing the maximal regularity estimates in Lemma \ref{L2.8} to the following system:
\begin{equation}\label{A5.5}
\left\{
\begin{aligned}
& \partial_t u-\Delta u+\nabla P=-(\rho-1)\partial_t u-\rho u\cdot\nabla u+B\cdot\nabla B, \\
& \dive u=0,\\
& u|_{t=0}=u_0
\end{aligned}\right.
\end{equation}
for $t\in(0,T)$, we have
\begin{align}\label{A5.6}
&\|u\|_{L^{\infty}_t(\dot{B}_{p,1}^{-1+3/p})}+\|(\partial_t u, \nabla^{2}u, \nabla P)\|_{L^{q,1}_t(L^{p})}+\|u\|_{L^{\infty}_t(L^{m})} \nonumber\\
&\quad\quad\leq\, C\|u_{0}\|_{\dot{B}_{p,1}^{-1+3/p}}+C\|(\rho-1)\partial_t u+\rho u\cdot\nabla u+B\cdot\nabla B\|_{L^{q,1}_t(L^{p})}. 
\end{align}
A direct application of H\"{o}lder’s inequality gives
\begin{align}\label{A5.611}
&\|(\rho-1)\partial_t u+\rho u\cdot\nabla u+B\cdot\nabla B\|_{L^{q,1}_t(L^{p})}\nonumber\\
&\quad\quad\leq C\|\rho-1\|_{L^{\infty}(0,t)\times\mathbb{R}^{3})}\|\partial_t u\|_{L^{q,1}_t(L^{p})}\nonumber\\
&\,\quad\quad\quad\quad+C\|\rho\|_{L^{\infty}(0,t)\times\mathbb{R}^{3})}\|u\cdot\nabla u\|_{L^{q,1}_t(L^{p})}+C\|B\cdot\nabla B\|_{L^{q,1}_t(L^{p})}.
\end{align}

On the other hand, we recall that $B$ satisfies
\begin{equation}\label{A5.7}
\left\{
\begin{aligned}
& \partial_t B-\Delta B=B\cdot\nabla u-u\cdot\nabla B, \\
& \dive B=0,\\
& B|_{t=0}=B_0.
\end{aligned}\right.
\end{equation}
Using Lemma \ref{L2.9} for \eqref{A5.7}, we arrive at
\begin{align}\label{A5.8}
&\|B\|_{L^{\infty}_t(\dot{B}_{p,1}^{-1+3/p})}+\|(\partial_t B, \nabla^{2}B)\|_{L^{q,1}_t(L^{p})}+\|B\|_{L^{\infty}_t(L^{m})} \nonumber\\
\leq\,& C\|B_{0}\|_{\dot{B}_{p,1}^{-1+3/p}}+C\| B\cdot\nabla u-u\cdot\nabla B\|_{L^{q,1}_t(L^{p})} \nonumber\\
\leq\,&C\|B_{0}\|_{\dot{B}_{p,1}^{-1+3/p}}+C\|B\cdot\nabla u\|_{L^{q,1}_t(L^{p})}+C\|u\cdot\nabla B\|_{L^{q,1}_t(L^{p})}.
\end{align}
Combining \eqref{A5.1} with $\varepsilon_1$ suitably small, \eqref{A5.6}, \eqref{A5.611} and \eqref{A5.8}, we end up with
\begin{align}\label{A5.9}
&\|(u,B)\|_{L^{\infty}_t(\dot{B}_{p,1}^{-1+3/p})}+\|(\partial_t u,\partial_t B, \nabla^{2}u,\nabla^2 B, \nabla P)\|_{L^{q,1}_t(L^{p})}+\|(u,B)\|_{L^{\infty}_t(L^{m})} \nonumber\\
&\quad\leq C\|(u_{0},B_0)\|_{\dot{B}_{p,1}^{-1+3/p}}+C\|u\cdot\nabla u\|_{L^{q,1}_t(L^{p})}+C\|B\cdot\nabla B\|_{L^{q,1}_t(L^{p})}\nonumber\\
&\quad\quad+C\|B\cdot\nabla u\|_{L^{q,1}_t(L^{p})}+C\|u \cdot\nabla B\|_{L^{q,1}_t(L^{p})}.
\end{align}

By utilizing Lemma \ref{lemmabesov} and the Sobolev embedding
\begin{align}\label{A5.10}
\dot{W}_p^1 \hookrightarrow L^{p^*} \quad\text{with}\quad  1/p^*=1/p-1/3,   
\end{align}
we infer
\begin{align}\label{A5.11}
&\|u\cdot\nabla u\|_{L^{q,1}_t(L^{p})}+\|B\cdot\nabla B\|_{L^{q,1}_t(L^{p})}+\|B\cdot\nabla u\|_{L^{q,1}_t(L^{p})}+\|u \cdot\nabla B\|_{L^{q,1}_t(L^{p})}\nonumber\\
 \leq\,&C\|(u,B)\|_{L^{\infty}_t(L^{3})}\|(\nabla u,\nabla B)\|_{L^{q,1}_t(L^{p^{*}})}  \nonumber\\
\leq\,&C\|(u,B)\|_{L^{\infty}_t(\dot{B}_{p,1}^{-1+3/p})}\|(\nabla^2 u,\nabla^2 B)\|_{L^{q,1}_t(L^{p})},
\end{align}
which, together with \eqref{A5.9}, implies that
\begin{align}\label{A5.12}
&\|(u,B)\|_{L^{\infty}_t(\dot{B}_{p,1}^{-1+3/p})}+\|(\partial_t u,\partial_t B, \nabla^{2}u,\nabla^2 B, \nabla P)\|_{L^{q,1}_t(L^{p})}+\|(u,B)\|_{L^{\infty}_t(L^{m})} \nonumber\\
&\quad\leq C\|(u_{0},B_0)\|_{\dot{B}_{p,1}^{-1+3/p}}+C\|(u,B)\|_{L^{\infty}_t(\dot{B}_{p,1}^{-1+3/p})}\|(\nabla^2 u,\nabla^2 B)\|_{L^{q,1}_t(L^{p})}.
\end{align}

Let us denote 
\begin{align*}
\mathcal{E}_1(t):=\|(u,B)(t)\|_{\dot{B}_{p,1}^{-1+3/p}}+\|(\partial_t u,\partial_t B,\nabla^{2}u,\nabla^2B,\nabla P)\|_{L^{q,1}_t(L^{p})}+\|(u,B)\|_{L^{\infty}_t(L^{m})},   
\end{align*}
and
\begin{align}
T_1^*:=\sup\bigl\{t\in(0,T^*)~:~\mathcal{E}_1(t)\leq  \beta_1\bigl\},\label{A5.141} 
\end{align}
where $\beta_1$ will be determined later.
Then \eqref{A5.12} implies that there exists a constant $C_1^*>0$ such that
\begin{align}\label{A5.14}
\sup_{s\in[0,T_1^*)}\mathcal{E}_1(s)\leq C_1^*\|(u_{0},B_0)\|_{\dot{B}_{p,1}^{-1+3/p}}+C_1^*\Big(\sup_{s\in[0,T_1^*)}\mathcal{E}_1(s)\Big)^2.   
\end{align}
By \eqref{A5.141} and \eqref{A5.14}, one can choose $\beta_1=\frac{1}{3C_1^*}$ and let $\|(u_{0},B_0)\|_{\dot{B}_{p,1}^{-1+3/p}}<\frac{2\beta_1}{3C_1^*}$ to deduce that
\begin{align*}
\sup_{s\in (0,T_1^*)}\mathcal{E}_1(t)\leq \frac{3}{2}C_1^*\|(u_{0},B_0)\|_{\dot{B}_{p,1}^{-1+3/p}}<\beta_1.   
\end{align*}
This implies that if $T_1^*<T^*$, then thanks to the continuity of $\mathcal{E}_1(t)$, one can take some time $t^*\in (T_1^*, T^*)$ such that $\mathcal{E}_1(t^*)\leq \beta_1$ holds, which contradicts with the maximality of $T_1^*$. Therefore, we conclude $T_1^*=T^*$, and  we infer that, for any $T\in(0,T^*)$,
\begin{align}\label{A5.12111}
&\|(u,B)\|_{L^{\infty}_T(\dot{B}_{p,1}^{-1+3/p})}+\|(u,B)\|_{L^{s,1}_T(L^{m})}\nonumber\\
&\quad +\|(\partial_t u,\partial_t B, \nabla^{2}u,\nabla^2 B, \nabla P)\|_{L^{q,1}_T(L^{p})}\leq C\|(u_{0},B_0)\|_{\dot{B}_{p,1}^{-1+3/p}}.
\end{align}
Moreover, from the definition $D_t=\partial_t+u\cdot \nabla $, \eqref{A5.10} and \eqref{A5.12111}, it holds that
\begin{align*}
\|(D_tu,D_tB)\|_{L^{q,1}_T(L^{p})}\leq&\,\|(\partial_t u,\partial_t B)\|_{L^{q,1}_T(L^{p})}+\|(u,B)\|_{L^{\infty}_T(\dot{B}_{p,1}^{-1+3/p})}\|(\nabla^2 u,\nabla^2 B)\|_{L^{q,1}_T(L^{p})} \nonumber\\
\leq&\, C\|(u_{0},B_0)\|_{\dot{B}_{p,1}^{-1+3/p}}.
\end{align*}

Finally, it follows from \eqref{A5.12111},  Gagliardo-Nirenberg inequality, Lemma \ref{lemmabesov} and $\dot{B}_{p,1}^{-1+3/p}\hookrightarrow L^3$ that
\begin{align}\label{A5.17}
\int_{0}^{T}\|(u,B) \|_{L^{\infty}}^{2} {\rm d}t& \,\leq C\int_{0}^{T}\|(u,B) \|_{L^3}^{2-q}\|(\nabla^2 u,\nabla^2 B) \|_{L^{p}}^{q} {\rm d}t' \nonumber\\
&\,\leq C\|(u,B)\|_{L^{\infty}_T(\dot{B}_{p,1}^{-1+3/p})}^{2-q}\|(\nabla^2 u,\nabla^2 B)\|_{L^{q,1}_T(L^{p})}^{q} \nonumber\\
&\leq  C\|(u_{0},B_0)\|_{\dot{B}_{p,1}^{-1+3/p}}^2.    
\end{align}
Owing to \eqref{A5.12111}--\eqref{A5.17}, we eventually obtain \eqref{A5.3}.
\end{proof}

Next, we establish higher-order weighted estimates of $(u,B,\nabla P)$, which leads to the time integrability of the Lipschitz norms for $(u, B)$ and the weighted time integrability for $(D_t u, D_tB)$.

\begin{prop}\label{P5.2}
Under the assumptions of Proposition \ref{P5.1}, for any $T\in(0,T^*)$, it holds that
\begin{align}
\|t(u,B)\|_{L^{\infty}_T(\dot{B}_{m,1}^{1+3/m})}+\|((tu)_{t},(tB)_{t},t\nabla^{2}u,t\nabla^{2}B, t\nabla P)\|_{L^{s,1}_T(L^{m})}&\leq \|(u_{0},B_0)\|_{\dot{B}_{p,1}^{-1+3/p}}, \label{A5.18}
\end{align}
and
\begin{align}
\|t^{1/2}(u,B)\|_{L^{\infty}_T(\dot{B}_{m,1}^{3/m})}+\|(\nabla u,\nabla B)\|_{L^1_T(L^{\infty})}+\|t (\nabla u,\nabla B)\|_{L^2_T(L^{\infty})}&\leq C\|(u_{0},B_0)\|_{\dot{B}_{p,1}^{-1+3/p}}.\label{A5.19}
\end{align}   
 
\end{prop}
\begin{proof}
Multiplying \eqref{I-1}$_2$ and $\eqref{I-1}_3$ by the time $t$ respectively, we obtain
\begin{align}\label{A5.20}
(tu)_t-\Delta(tu)+\nabla(tP)=-(\rho-1)(tu)_t+\rho u-\rho u\cdot\nabla (tu)+B\cdot\nabla (tB),   
\end{align}
and
\begin{align}\label{A5.21}
(tB)_t-\Delta(tB)=B+B\cdot\nabla (tu)-u\cdot\nabla (tB).   
\end{align}
Then, employing Lemmas \ref{L2.8}--\ref{L2.9} to \eqref{A5.20} and \eqref{A5.21}, we arrive at
\begin{align}\label{A5.22}
&\|(tu,tB)\|_{L^{\infty}_T(\dot{B}_{m,1}^{1+3/m})}+\|((tu)_{t},(tB)_{t},\nabla^{2}(tu),\nabla^2(tB),\nabla(tP))\|_{L^{s,1}_T(L^{m})} \nonumber\\
&\quad\leq C\|\rho-1\|_{L^{\infty}_T(L^{\infty})}\|(tu)_{t}\|_{L^{s,1}_T(L^{m})}+C\|tB\cdot(\nabla u,\nabla B)\|_{L^{s,1}_T(L^{m})} \nonumber\\
&\quad\quad+C\big(\|\rho-1\|_{L^{\infty}_T(L^{\infty})}+1\big)\big(\|(u,B)\|_{L^{s,1}_T(L^{m})}+\|tu\cdot(\nabla u,\nabla B)\|_{L^{s,1}_T(L^{m})}\big).
\end{align}
Using  H\"{o}lder’s inequality and the properties given in Lemmas \ref{lemmabesov} and \ref{D A.9}, we have
\begin{align}\label{A5.23}
&\|tu\cdot(\nabla u,\nabla B)\|_{L^{s,1}_T(L^{m})}+\|tB\cdot(\nabla u,\nabla B)\|_{L^{s,1}_T(L^{m})}\nonumber\\
\leq&\, C\|(u,B)\|_{L^{s,1}_T(L^{m})}\|t(\nabla u,\nabla B)\|_{L^{\infty}_T(L^{\infty})} \nonumber\\
\leq&\, C\|(u,B)\|_{L^{s,1}_T(L^{m})}\|(tu,tB)\|_{L^{\infty}_T(\dot{B}_{m,1}^{1+3/m})},
\end{align}
which, together with \eqref{A5.1},  \eqref{A5.11}, \eqref{A5.22}, yields the desired estimate \eqref{A5.18}.

To prove \eqref{A5.19}, we obtain from \eqref{A5.3} and \eqref{A5.18}, Lemma \ref{L2.7} and $p<m$ that
\begin{align}
\|t^{1/2}( u,  B)\|_{L^{\infty}_T(\dot{B}_{m,1}^{\frac{3}{m}})}&\,\leq C\|( u,  B)\|_{L^{\infty}_T(\dot{B}_{m,1}^{-1+3/m})}^{1/2} \|t(  u,  B)\|_{L^{\infty}_T(\dot{B}_{m,1}^{1+3/m})}^{1/2} \nonumber \\
&\,\leq C \|( u,B)\|_{L^{\infty}_T(\dot{B}_{p,1}^{-1+3/p})}^{1/2} \|t(u, B)\|_{L^{\infty}_T(\dot{B}_{m,1}^{1+3/m})}^{1/2} \nonumber \\
&\,\leq C\|(u_{0},B_0)\|_{\dot{B}_{p,1}^{-1+3/p}}.\nonumber
\end{align}
In addition, we take advantage of the following Gagliardo-Nirenberg inequality:
\begin{align*}
\|(u,B)\|_{L^{\infty}}\leq C\|(\nabla u,\nabla B)\|_{L^{p}}^{\frac{p(m-3)}{3(m-p)}}\|(\nabla^2 u,\nabla^2 B)\|_{L^{m}}^{\frac{m(3-p)}{3(m-p)}}.
\end{align*}
Here, we used $p<3<m$. Using \eqref{A5.3}, \eqref{A5.23},
Lemma \ref{L2.6} and the fact that $t^{-a}\in L^{1/\alpha,\infty}(\mathbb{R}_{+})$ with $\alpha=\frac{m(3-p)}{3(m-p)}$,
we derive
\begin{align*}
\int_{0}^{T}\|(\nabla u,\nabla B)\|_{L^{\infty}} {\rm d}t
\leq\,& C\int_{0}^{T}t^{-\frac{m(3-p)}{3(m-p)}}\|(\nabla^2 u,\nabla^2 B)\|_{L^{p}}^{\frac{p(m-3)}{3(m-p)}}\|t(\nabla^2 u,\nabla^2 B) \|_{L^{m}}^{\frac{m(3-p)}{3(m-p)}} {\rm d}t \nonumber \\
\leq\,& C\|(\nabla^2 u,\nabla^2 B)\|_{L^{q,1}_T(L^{p})}^{\frac{p(m-3)}{3(m-p)}}\|t (\nabla^2 u,\nabla^2 B)\|_{L^{s,1}_T(L^{m})}^{\frac{m(3-p)}{3(m-p)}} \nonumber\\
\leq\,&C\|(u_{0},B_0)\|_{\dot{B}_{p,1}^{-1+3/p}},
\end{align*}
from which and \eqref{A5.18} one infers
\begin{align*}
\int_{0}^{T}t\|(\nabla u,\nabla B) \|_{L^{\infty}}^{2} {\rm d}t&\,\leq C\|(tu,tB)\|_{L^{\infty}_T(\dot{B}_{m,1}^{1+3/m})}\int_{0}^{T}\|(\nabla u,\nabla B)\|_{L^{\infty}}{\rm d}t \nonumber\\
&\leq C\|(u_{0},B_0)\|_{\dot{B}_{p,1}^{-1+3/p}}^{2}.
\end{align*} 
These lead to \eqref{A5.19} and  the proof of Proposition \ref{P5.2} is complete.
\end{proof}

Furthermore, we explore some weighed estimates for $D_t u$ and $D_t B$.

\begin{prop}\label{P5.3}
Under the assumptions of Proposition \ref{P5.1}, for any $T\in(0,T^*)$, it holds that
\begin{align}\label{A5.31}
&\|(tD_tu,tD_tB)\|_{L^{\infty}_T(\dot{B}_{p,1}^{-1+3/p})}+\|(t D_t u,tD_tB)\|_{L^{s,1}_T(L^{p})}\nonumber\\
&\quad\quad +\|(\partial_t(tD_tu),\partial_t(tD_tB),\nabla^{2}(tD_tu),\nabla^{2}(t D_tB)\|_{L^{q,1}_T(L^{p})}\leq C\|(u_{0},B_0)\|_{\dot{B}_{p,1}^{-1+3/p}},
\end{align}
and
\begin{align}\label{A5.32}
\|(t \nabla D_t u,t \nabla D_t B)\|_{L^{2}_T(L^{3})}+\|(t D_t u,t D_t B)\|_{L^{2}_T(L^{\infty})}\leq C\|(u_{0},B_0)\|_{\dot{B}_{p,1}^{-1+3/p}}.   
\end{align}
\end{prop}

\begin{proof}
Applying the operator $D_t$ to $\eqref{I-1}_2$ and $\eqref{I-1}_3$ yields
\begin{align}\label{A5.33}
\rho D_{tt}^2u-\Delta  D_t u+\nabla D_t P &\,=-\Delta u\cdot\nabla u-2\nabla u\cdot\nabla^2 u+\nabla u\cdot\nabla P+D_t(B\cdot\nabla B):=f,\\
D_{tt}^2 B-\Delta  D_{t} B&\,=-\Delta u\cdot\nabla B-2\nabla u\cdot\nabla^2B+D_t(B\cdot\nabla u):=h.\label{A5.34}
\end{align}
Multiplying \eqref{A5.33} and \eqref{A5.34} by $t$, we obtain
\begin{align}\label{A5.35}
 \rho\partial_t(t D_t u)-\Delta (t D_t u)+\nabla (t D_t P)&=-t\rho u\cdot\nabla  D_t u+\rho D_t u+tf,\\
  \partial_t(t D_t B)-\Delta (t D_t B)&
=-t u\cdot\nabla D_t B+D_t B+th.\label{A5.36}
\end{align}
To cancel the term $\nabla (t D_t P)$ in \eqref{A5.35}, we consider the
Helmholtz projectors $\mathbb{P}$ and $\mathbb{Q}$ as
\begin{align}\label{A5.37}
\mathbb{P}:=\mathrm{Id}+\nabla(-\Delta)^{-1}\dive, \quad \mathbb{Q}:=-\nabla(-\Delta)^{-1}\dive.
\end{align}
From \eqref{A5.35} and the facts that $\dive  u=0$ and $\dive { D_t u}={\rm Tr}(\nabla u\cdot\nabla u)$, it is clear that
\begin{align}\label{A5.38}
\nabla (t D_t P)=\mathbb{Q}\big(-\rho(t D_t u)_t+\Delta (t D_t u)-t\rho u\cdot\nabla  D_t u+\rho D_t u+tf                  \big),    
\end{align}
and
\begin{align}\label{A5.39}
\mathbb{Q}(t\Delta D_t u)=t\nabla {\rm Tr}(\nabla u\cdot\nabla u).
\end{align}
Therefore, we deduce from \eqref{A5.35}--\eqref{A5.39} that
\begin{align}\label{A5.42}
(t D_t u)_{t}-\Delta (t D_t u)=&\,\mathbb{P}\big((1-\rho)(t D_t u)_{t}-t\rho u\cdot\nabla D_t u+\rho D_t u+tf\big)-\nabla\mathrm{Tr}(t\nabla u\cdot\nabla u)\nonumber\\
&\,+\mathbb{Q}( D_t u+t\partial_t u\cdot\nabla u+tu\cdot\nabla \partial_t u),
\end{align}
where we have used 
\begin{align}
\mathbb{Q}(u\cdot\nabla \partial_t u)&=\mathbb{Q}(\partial_t u\cdot\nabla u),\nonumber\\
\mathbb{Q}(\rho\partial_t(t D_t u))&=\mathbb{Q}\big((\rho-1)\partial_t(t D_t u)+ t D_t u+ D_t u+t u\cdot\nabla \partial_t u+t\partial_t u\cdot\nabla u                 \big).\nonumber
\end{align}

We are in a position to prove \eqref{A5.31}. By utilizing the maximal regularity estimates in Lemma \ref{L2.9} to \eqref{A5.42} and the continuity of projectors $\mathbb{P}$ and $\mathbb{Q}$ on $L^{q,1}(0,T;L^{p})$, we obtain
\begin{align}\label{A5.43}
&\|(t D_t u,t D_t B)\|_{L^{\infty}_T(\dot{B}_{p,1}^{-1+3/p})}+\|(t D_t u,t D_t B)\|_{L^{s,1}_T(L^{m})} \nonumber\\
&\quad+\|((t D_t u)_{t},(t D_t B)_{t},t\nabla^{2} D_t u,t\nabla^{2}D_t B)\|_{L^{q,1}_T(L^{p})} \nonumber\\
&\quad\quad\leq C\|\rho-1\|_{L^{\infty}_T(L^{\infty})}\|\partial_t(t D_t u)\|_{L^{q,1}_T(L^{p})}+C\|t\rho u\cdot\nabla D_t u\|_{L^{q,1}_T(L^{p})}+C\|D_t B\|_{L^{q,1}_T(L^{p})} \nonumber\\
&\quad\quad\quad+C\|\rho D_t u\|_{L^{q,1}_T(L^{p})}+\|(tf,th)\|_{L^{q,1}_T(L^{p})}+\| D_t u+t\partial_t u\cdot\nabla u\|_{L^{q,1}_T(L^{p})} \nonumber\\
&\quad\quad\quad+C\|t \nabla\mathrm{Tr}(\nabla u\cdot\nabla u)\|_{L^{q,1}_T(L^{p})}+C\|t u\cdot\nabla D_t B\|_{L^{q,1}_T(L^{p})}.
\end{align}
The nonlinear terms on the right-hand side of \eqref{A5.43} are analyzed as follows.
First, by using the uniform bounds obtained in Propositions \ref{P5.1}--\ref{P5.2} as well as \eqref{A5.10}, and Lemma \ref{L2.6} and  H\"{o}lder’s inequality, we discover that
\begin{align}\label{A5.45}
\|t \nabla\mathrm{Tr}(\nabla u\cdot\nabla u)\|_{L^{q,1}_T(L^{p})} 
\leq\,&C\|t\nabla u\|_{L^{\infty}_T(L^{\infty})}\|\nabla^{2}u\|_{L^{q,1}_T(L^{p})}\nonumber \\
\leq\,&C\|tu\|_{L^{\infty}_T(\dot{B}_{m,1}^{1+3/m})}\|\nabla^{2}u\|_{L^{q,1}_T(L^{p})} \nonumber\\
\leq\,&C\|(u_{0},B_0)\|_{\dot{B}_{p,1}^{-1+3/p}}^{2},
\end{align}
and
\begin{align*}
\|t\partial_t u\cdot\nabla u\|_{L^{q,1}_T(L^{p})} \leq\,&C\|\partial_t u\|_{L^{q,1}_T(L^{p})}\|t\nabla u\|_{L^{\infty}_T(L^{\infty})}\nonumber \\
\leq\,&C\|\partial_t u\|_{L^{q,1}_T(L^p)}\|tu\|_{L^\infty_T(\dot{B}_{m,1}^{1+3/m})} \nonumber\\
\leq\,&C\|(u_{0},B_0)\|_{\dot{B}_{p,1}^{-1+3/p}}^{2}.
\end{align*}
Recalling the embedding \eqref{A5.10}, we also deduce that 
\begin{align}
&\|t\rho u\cdot\nabla D_t u\|_{L^{q,1}_T(L^{p})}+\|t u\cdot\nabla D_t B\|_{L^{q,1}_T(L^{p})} \nonumber\\
\leq\,&C\|t(\nabla D_t u,\nabla D_t B)\|_{L^{q,1}_T(L^{p^{*}})}\|(u,B)\|_{L^{\infty}_T(L^{3})} \nonumber\\
\leq\,&C\|t(\nabla^{2} D_t u,\nabla^2 D_t B)\|_{L^{q,1}_T(L^{p})}\|(u_{0},B_0)\|_{\dot{B}_{p,1}^{-1+3/p}}. \label{A5.46}
\end{align}
Here and below, let $1<p^{*}<\infty$ satisfy  $1/p^{*}=1/p-1/3$.
Note that $\|(tf,th)\|_{L^{q,1}_T(L^{p})}$ can be decomposed as
\begin{align}\label{A5.47}
&\|(tf,th)\|_{L^{q,1}_T(L^{p})}\nonumber\\
&\quad\leq  C\|-t\Delta u\cdot\nabla u-2t\nabla u\cdot\nabla^{2}u+t\nabla u\cdot\nabla P\|_{{L^{q,1}_T(L^{p})}}+C\|D_t(B\cdot\nabla B)\|_{{L^{q,1}_T(L^{p})}}\nonumber\\
&\quad\quad+C\|t\Delta u\cdot\nabla B+2t\nabla u\cdot\nabla^{2}B\|_{{L^{q,1}_T(L^{p})}}+C\|D_t(B\cdot\nabla u)\|_{{L^{q,1}_T(L^{p})}}.
\end{align}
Similar to \eqref{A5.45}, direct calculations give rise to
\begin{align}\label{A5.48}
&\|-t\Delta u\cdot\nabla u-2t\nabla u\cdot\nabla^{2}u+t\nabla u\cdot\nabla P\|_{{L^{q,1}_T(L^{p})}}\nonumber\\
\leq\,&C\|t\nabla u\|_{L^{\infty}_T(L^{\infty})}\|(\nabla^{2}u,\nabla^2 P)\|_{L^{q,1}_T(L^{p})}\nonumber \\
\leq\,&C\|tu\|_{L^{\infty}_T(\dot{B}_{m,1}^{1+3/m})}\|(\nabla^{2}u,\nabla^2P)\|_{L^{q,1}_T(L^{p})} \nonumber\\
\leq\,&C\|(u_{0},B_0)\|_{\dot{B}_{p,1}^{-1+3/p}}^{2},   
\end{align}
and
\begin{align}\label{A5.49}
\|t\Delta u\cdot\nabla B+2t\nabla u\cdot\nabla^{2}B\|_{{L^{q,1}_T(L^{p})}}
\leq\,& C\|t(\nabla u,\nabla B)\|_{L^{\infty}_T(L^{\infty})}\|(\nabla^{2}u,\nabla^2 B)\|_{L^{q,1}_T(L^{p})}\nonumber \\
\leq\,&C\|(u_{0},B_0)\|_{\dot{B}_{p,1}^{-1+3/p}}^{2}.
\end{align}

On the other hand, due to the facts that $\partial_t u= D_t u-u\cdot\nabla u$ and $\partial_t B=D_t B-u\cdot\nabla B$, it holds
\begin{align}\label{A5.50}
&\|D_t(B\cdot\nabla B)\|_{{L^{q,1}_T(L^{p})}}+\|D_t(B\cdot\nabla u)\|_{{L^{q,1}_T(L^{p})}}\nonumber\\
\leq\,& C\|(tu,tB)\cdot\nabla( D_t u,D_t B)\|_{{L^{q,1}_T(L^{p})}}+ C\|t(u,B)\otimes(u,B):\nabla^2(u,B)\|_{{L^{q,1}_T(L^{p})}}\nonumber\\
&+C\|t(u,B)\cdot\nabla({u},{B})\cdot(\nabla u,\nabla B)\|_{{L^{q,1}_T(L^{p})}}+ C\|t(\partial_t u,\partial_t B)\cdot(\nabla u,\nabla B)\|_{{L^{q,1}_T(L^{p})}}\nonumber\\
\leq\,&C\|t(\nabla D_t u,\nabla D_t B)\|_{L^{q,1}_T(L^{p^{*}})}\|(u,B)\|_{L^{\infty}_T(L^{3})}\nonumber\\
&+C\|\nabla^{2}({u},B)\|_{L^{q,1}_T(L^{p})}\|(u,B)\|_{L^{\infty}_T(L^{\infty})}\|t(u,B)\|_{L^{\infty}_T(L^{\infty})} \nonumber\\
&+C\|t(\nabla u,\nabla B)\|_{L^{\infty}_T(L^{\infty})}\|(u,B)\cdot\nabla({u},B)\|_{L^{q,1}_T(L^{p})}\nonumber\\
&+C\|({u}_t,{B}_t)\|_{L^{q,1}_T(L^{p})}\|t(\nabla u,\nabla B)\|_{L^{\infty}_T(L^{\infty})}\nonumber\\
\leq &\,C\|t(\nabla^{2} D_t u,\nabla^2 D_t B)\|_{L^{q,1}_T(L^{p})}\|(u_{0},B_0)\|_{\dot{B}_{p,1}^{-1+3/p}}\nonumber\\
&+C\|(\partial_t u,\partial_t B)\|_{L^{q,1}_T(L^p)}\|(tu,tB)\|_{L^\infty_T(\dot{B}_{m,1}^{1+3/m})} \nonumber\\
&+C\|(tu,tB)\|_{L^\infty_T(\dot{B}_{m,1}^{1+3/m})}\|(u,B)\|_{L^\infty_T(\dot{B}_{p,1}^{-1+3/p})}\|\nabla^2({u},{B})\|_{L^{q,1}_T(L^{p^{*}})}\nonumber\\
&+C\|\nabla^{2}({u},B)\|_{L^{q,1}_T(L^{p})}\|(tu,tB)\|_{L^\infty_T(\dot{B}_{m,1}^{1+3/m})}\|(u,B)\|_{L^\infty_T(\dot{B}_{m,1}^{-1+3/m})}\nonumber\\
\leq\,&C\|t(\nabla^{2} D_t u,\nabla^2 D_t B)\|_{L^{q,1}_T(L^{p})}\|(u_{0},B_0)\|_{\dot{B}_{p,1}^{-1+3/p}}\nonumber\\
&+C\|(u_{0},B_0)\|_{\dot{B}_{p,1}^{-1+3/p}}^2+C\|(u_{0},B_0)\|_{\dot{B}_{p,1}^{-1+3/p}}^3,
\end{align}
where we have used H\"{o}lder’s inequality,
Propositions \ref{P5.1}--\ref{P5.2}, Lemmas \ref{lemmabesov}--\ref{L2.6},
and \eqref{A5.10}.

Then, we substitute \eqref{A5.48}--\eqref{A5.50} into \eqref{A5.47} to obtain 
\begin{align}\label{A5.51}
&\|(tf,th)\|_{L^{q,1}_T(L^{p})}\leq C\|(u_{0},B_0)\|_{\dot{B}_{p,1}^{-1+3/p}}^2\Big(1+\|(u_{0},B_0)\|_{\dot{B}_{p,1}^{-1+3/p}}+\|t(\nabla^{2} D_t u,\nabla^2 D_t B)\|_{L^{q,1}_T(L^{p})}\Big).   
\end{align}
Hence, combining the above estimates \eqref{A5.45}--\eqref{A5.51} with \eqref{A5.43} and remembering the smallness condition in \eqref{A5.1}, we end up with \eqref{A5.31}. 

Let us prove the last estimate \eqref{A5.32}. In the case $1<p\leq 3/2$, one derives \eqref{A5.32} by using the embedding $\dot{B}_{p,1}^{-1+3/p}\hookrightarrow \dot{B}_{p_1,1}^{-1+3/p_1}$ for any $3/2<p_1<3$. For $3/2<p<3$ while $1<q<2$, it follows from the embedding properties in Lemma \ref{lemmabesov} that
\begin{align*}
 \|( D_t u, D_t B)\|_{L^{\infty}}&\leq C\|( D_t u, D_t B)\|_{L^{3}}^{1-{q}/{2}}\|(\nabla^{2} D_t u,\nabla^{2} D_t B)\|_{L^{p}}^{{q}/{2}}\nonumber\\
 &\leq C\|( D_t u, D_t B)\|_{\dot{B}_{p,1}^{-1+3/p}}^{1-{q}/{2}}\|(\nabla^{2} D_t u,\nabla^{2} D_t B)\|_{L^{p}}^{{q}/{2}},   
\end{align*}
which, together with \eqref{A5.31}, implies that
\begin{align*}
\|(t D_t u,t D_t B)\|_{L^{2}_T(L^{\infty})}
\leq&\, C\|(t D_t u,t D_t B)\|_{L^{\infty}_T(\dot{B}_{p,1}^{-1+3/p})}^{1-{q}/{2}}\|(\nabla^{2}(t D_t u),\nabla^{2}(t D_t B))\|_{L^{q}_T(L^{p})}^{{q}/{2}}\nonumber\\
\leq&\, C\|(u_{0},B_0)\|_{\dot{B}_{p,1}^{-1+3/p}}^2.    
\end{align*}
Similarly, one can obtain
\begin{align*}
\|(t\nabla D_t u,t\nabla D_t B)\|_{L^{2}_T(L^{3})}
\leq&\, C\|(t D_t u,t D_t B)\|_{L^{\infty}_T(\dot{B}_{p,1}^{-1+3/p})}^{\frac{2p-3}{3p-3}}\|(t\nabla^{2} D_t u,t\nabla^2 D_t B)\|_{L^{q}_T(L^{p})}^{\frac{p}{3p-3}}\nonumber\\
\leq&\, C\|(u_{0},B_0)\|_{\dot{B}_{p,1}^{-1+3/p}}.
\end{align*}
Thus, we conclude \eqref{A5.32} and finish the proof of Proposition \ref{P5.3}.
\end{proof}

\begin{rem}
In the case $1<p\leq 2$, one also has the following estimate: 
\begin{align}
&\|(\nabla u,\nabla B)\|_{L_T^{4}(L^{2})}+\|t^{1/4}(\nabla^2 u, \nabla^2 B,\partial_t u,\partial_t B)\|_{L^2_T(L^2)}\nonumber\\
&\quad+\|t^{3/4}(\nabla \partial_t u , \nabla \partial_t B, D_{t}u,\nabla D_{t}B)\|_{L_{T}^{2}(L^{2})}\leq C\|(u_{0},B_0)\|_{\dot{B}_{p,1}^{-1+3/p}},\label{uniquep}
\end{align}
which can be shown by the embedding $\dot{B}_{p,1}^{-1+p/3}\hookrightarrow \dot{H}^{1/2}\hookrightarrow \dot{B}_{2,\infty}^{1/2}
 $ for $1<p\leq 2$ and the inequalities \eqref{A4.53} and \eqref{A4.55}. Note that \eqref{uniquep} will be used in the proof of uniqueness. 
\end{rem}

\vspace{2mm}

Now, we are going to prove Theorem \ref{T2.3}.

\vspace{1mm}

\noindent
\begin{proof}[Proof of Theorem \ref{T2.3}]
Assume that $(\rho_0,u_0,B_0)$ satisfies the conditions \eqref{A1.13}. Let $K_\varepsilon$ be the standard molifier in $\mathbb{R}^3$. For $\varepsilon\in(0,1)$, we defined the regularized initial data by 
\begin{eqnarray}
(\rho_0^{\varepsilon},u_0^{\varepsilon},B_0^{\varepsilon})(x)=\chi(\varepsilon^{-1}x)(\rho_0\ast K_\varepsilon,u_0\ast K_\varepsilon,B_0\ast K_\varepsilon)(x),\label{appa}
\end{eqnarray}
where the cut-off function $\chi(x) \in \mathcal{S}(\mathbb{R}^{d})$ satisfies  $\chi(0)=1$ and Supp $\mathcal{F}(\chi)(\xi)\subset B(0,1)$. Then, one can check that $(\rho_0^{\varepsilon},u_0^{\varepsilon},B_0^{\varepsilon})$ fulfills the uniform bounds
\begin{equation}
\left\{
\begin{aligned}\label{A5.57}
&\|\rho_0^{\varepsilon}-1\|_{L^{\infty}}\leq \|\rho_0-1\|_{L^{\infty}},\\
&\|(u_0^{\varepsilon},B_0^{\varepsilon})\|_{\dot{B}^{-1+3/p}_{p,1}}\leq C\|\chi \|_{\dot{B}^{d/p}_{p,1}}\|(u_0,B_0)\|_{\dot{B}^{-1+3/p}_{p,1}}\leq C\|(u_0,B_0)\|_{\dot{B}^{-1+3/p}_{p,1}},
\end{aligned}
\right.
\end{equation}
where we have used the spatial scaling invariance of $\dot{B}^{d/p}_{p,1}$ such that $\|\chi(\varepsilon^{-1}\cdot)\|_{\dot{B}^{d/p}_{p,1}}\sim \|\chi(\cdot)\|_{\dot{B}^{d/p}_{p,1}}$. In addition, $\rho_0^{\varepsilon}-1$ and $(u_0^{\varepsilon},B_0^{\varepsilon})$, respectively, converges strongly to $\rho_0-1$ in $L^{l} $ ($1\leq l<\infty$) and $(u_0,B_0)$ in
$\dot{B}_{p,1}^{-1+3/p}$. 

According the construction \eqref{appa}, for any $1<p<3$, $(\rho_0^{\varepsilon},u_0^{\varepsilon},B_0^{\varepsilon})$ is smooth, and we have $\inf_{x\in\mathbb{R}^3}\rho^{\varepsilon}_0(x) >0$, $\rho_0^{\varepsilon}-1\in H^{3/2+s}$ and $(u_0^{\varepsilon},B_0^{\varepsilon}) \in  H^{1/2+s}$ with arbitrary $s>0$. According to \cite[Theorem 1.1]{AH-AMBP-2007}, for fixed $\varepsilon\in(0,1)$, there exists a maximal time $T^\varepsilon$ such that the system \eqref{I-1} associated with the initial data $(\rho_0^{\varepsilon},u_0^{\varepsilon},B_0^{\varepsilon})$ has a unique smooth solution $(\rho^{\varepsilon},u^{\varepsilon},B^{\varepsilon},\nabla P^{\varepsilon})$ on $[0,T^\varepsilon)\times\mathbb{R}^3$.

Denote 
\begin{align*} 
\dot{W}_{p,(q,r)}^{2,1}((0,T)\times\mathbb{R}^3):=&\,\Big\{u,B\in C(0,T;\dot{B}_{p,r}^{2-2/q}): \partial_t u,\partial_t B,\nabla^2u,\nabla^2 B\in L^{q,r}(0,T;L^p)        \Big\},\\
\dot{W}_{p,r}^{2,1}((0,T)\times\mathbb{R}^3):=&\,\Big\{u,B\in C(0,T;\dot{B}_{p,r}^{2-2/r}): \partial_t u,\partial_t B,\nabla^2u,\nabla^2 B\in L^{r}(0,T;L^p)        \Big\}.
\end{align*}
As pointed out in \cite{DW-CMMP-2023}, we consider here the larger space $\dot{W}_{p,r}^{2,1}((0,T)\times\mathbb{R}^3)$ to overcome the lack of reflexivity in Lorentz spaces. We aim to establish the uniform estimate of 
$(u^{\varepsilon},B^{\varepsilon},\nabla P^{\varepsilon})$ in the space $E_{T^*}^{p,r}$ given by
\begin{align*}
E_T^{p,r}:=\Big\{(u,B,\nabla P):u,B\in\dot{W}_{p,r}^{2,1}((0,T)\times\mathbb{R}^3),\quad \nabla P\in L^{r}(0,T;L^p)     \Big\},
\end{align*}
for $r\geq 1$.

From \eqref{A5.57} and the transport nature of $\eqref{I-1}_1$, we have for any $T\in(0,T^\varepsilon)$, 
\begin{align*}
\|(\rho^{\varepsilon}-1)(t)\|_{L^{\infty}}=\|\rho^{\varepsilon}_0-1\|_{L^{\infty}}\leq 
\|a_0\|_{L^{\infty}},\quad t\in[0,T],
\end{align*}
and due to Proposition \ref{P5.1}, $(u^{\varepsilon},B^{\varepsilon},\nabla P^{\varepsilon})$ belongs to 
\begin{align*}
\|(u^{\varepsilon},B^{\varepsilon})\|_{\dot{W}_{p,(q,1))}^{2,1}(0,T)\times\mathbb{R}^3)}+\|\nabla P^{\varepsilon}\|_{L^{q,1}_T(L^p)}\leq C\|(u_0,B_0)\|_{\dot{B}^{-1+3/p}_{p,1}}.
\end{align*}
By taking the same progress as that in \cite[Section 5]{DW-CMMP-2023}, we can obtain
\begin{align}\label{A5.62}
\|(u^{\varepsilon},B^{\varepsilon},\nabla P^{\varepsilon})\|_{E_T^{p,r}}^r\leq C\|(u_0^{\varepsilon},B_0^{\varepsilon})\|_{\dot{B}_{p,r}^{2-2/r}}^r\exp\bigg\{C\int_0^T\|(u^{\varepsilon},B^{\varepsilon})\|_{L^{\frac{3r}{r-1}}}^{2r}{\rm d}t\bigg\}.    
\end{align}
It follows from Lemma \ref{lemmabesov} and Gagliardo-Nirenberg inequality that
\begin{align*}
\int_{0}^{T}\|(u^{\varepsilon},B^{\varepsilon})\|_{L^{\frac{3r}{r-1}}}^{2r}{\rm d}t\leq\|(u^{\varepsilon},B^{\varepsilon})\|_{L^{2}_T(L^{\infty})}^{2}\|(u^{{\varepsilon}},B^{{\varepsilon}})\|_{L^{\infty}_T(\dot{B}_{p,1}^{-1+3/p})}^{2(r-1)}.   
\end{align*}
According to \eqref{A5.3} in Proposition \ref{P5.1},
we get
\begin{align*}
\int_{0}^{T}\|(u^{\varepsilon},B^{\varepsilon})\|_{L^{\frac{3r}{r-1}}}^{2r}{\rm d}t\leq C \|(u_0,B_0)\|_{\dot{B}_{p,1}^{-1+3/p}}^{2r}.   
\end{align*}
This, together with \eqref{A5.62} and a continuous argument, implies that
$(u^{\varepsilon},B^{\varepsilon},\nabla P^{\varepsilon})$ belongs to $\dot{W}_{p,r}^{2,1}$ 
with $1<r<\infty$ and $T_{\varepsilon}^{*}=\infty$. 
Owing to \eqref{A5.57} and the smoothness of $(u^{\varepsilon},B^{\varepsilon},\nabla P^{\varepsilon})$,
we establish the {\emph{a priori}} estimates presented in Propositions \ref{P5.1}--\ref{P5.3}  are uniform in $\varepsilon$ and $T^\varepsilon$. Then, based on these bounds, we are able to obtain higher-order $H^s$ regularity estimates of $(\rho^\varepsilon, u^{\varepsilon}, B^{\varepsilon}, \nabla P^{\varepsilon})$ for any suitably large $s$. If $T^\varepsilon<\infty$, then in accordance with the local existence theorem again, one can extend the solution $(\rho^\varepsilon, u^{\varepsilon}, B^{\varepsilon}, \nabla P^{\varepsilon})$ beyond $T^\varepsilon$, which leads to the contradiction of the maximality of $T^\varepsilon$. Thus, we have $T^\varepsilon=\infty$ and justify that  $(\rho^\varepsilon, u^{\varepsilon}, B^{\varepsilon}, \nabla P^{\varepsilon})$ is indeed a global-in-time smooth solution. 

According to the uniform boundedness of $\rho^\varepsilon$ in $L^{\infty}(0,T;L^{\infty})$ and $(u^{\varepsilon},B^{\varepsilon},\nabla P^{\varepsilon})$ in $E_T^{p,r}$ for any $0<T<\infty$, there exists a limit $(\rho,u,B,\nabla P)$ such that up to a subsequence, the following convergence properties hold:
\begin{equation}
\left\{
\begin{aligned}\label{A5.65}
\rho^{\varepsilon}-1\rightharpoonup&\, \rho-1 && \mathrm{weak}^{*}\  \mathrm{in} &&L^{\infty}(\mathbb{R}_{+};L^{\infty}), \\
u^{\varepsilon}\rightharpoonup& \,u && \mathrm{weak}^{*}\ \mathrm{in} &&L^{\infty}(\mathbb{R}_{+};\dot{B}_{p,q}^{2-2/q}), \\
\nabla P^{\varepsilon}\rightharpoonup&\,\nabla P
&& \mathrm{weakly}\  {\rm in}  &&L^{q}(\mathbb{R}_{+};L^{p}),\\
B^{\varepsilon}\rightharpoonup& \,B 
&&
\mathrm{weak}^{*}\ \mathrm{in}  &&L^{\infty}(\mathbb{R}_{+};\dot{B}_{p,q}^{2-2/q}),\\
(\partial_{t}u^{\varepsilon},\nabla^{2}u^{\varepsilon})\rightharpoonup&\,(\partial_{t}u,\nabla^{2}u)
&&\mathrm{weakly}\  {\rm in}  && L^{q}(\mathbb{R}_{+};L^{p}), \\
(\partial_{t}B^{\varepsilon},\nabla^{2}B^{\varepsilon})\rightharpoonup&\,(\partial_{t}B,\nabla^{2}B)
&& \mathrm{weakly}\ {\rm in}  &&L^{q}(\mathbb{R}_{+};L^{p}),   
\end{aligned}
\right.
\end{equation}

 Using Fatou's Lemma and Propositions 
\ref{P5.1}--\ref{P5.3}, we also have
\begin{align}\label{A5.67}
u,B\in\dot{W}_{p,(q,1)}^{2,1}(\mathbb{R}_{+}\times\mathbb{R}^{3}),\quad\nabla P\in L^{q,1}(\mathbb{R}_{+};L^{p}).
\end{align}
Furthermore, it follows from the Aubin–Lions Lemma ({\!\!\cite{Lions-book-1996}})
and the Cantor diagonal formula that 
\begin{equation}
\left\{
\begin{aligned}\label{A5.68}
u^{\varepsilon}\to&\, u\quad\quad{\rm in}\quad L^{\infty}_{\rm loc}(\mathbb{R}_{+};L^{3-\eta}_{\rm loc}), \\
\nabla u^{\varepsilon}\to&\,\nabla u\,\quad{\rm in}\quad L^{q}_{\rm loc}(\mathbb{R}_{+};L^{p^{*}-\eta}_{\rm loc}),\\
\nabla B^{\varepsilon}\to&\,\nabla B\quad{\rm in}\quad L^{q}_{\rm loc}(\mathbb{R}_{+};L^{p^{*}-\eta}_{\rm loc}), 
\end{aligned}
\right.
\end{equation}
where $p^{*}\in(1,\infty)$ is given by $1/p^{*}=1/p-1/3$ and $\eta>0$ is a suitably small constant. Applying the Diperna-Lions theory ({\!\!\cite{DL-IM-1998}}), we also have
\begin{align}
\rho^{\varepsilon}-1\to \rho-1 \quad {\rm in}\quad L^{\alpha}_{\rm loc}(\mathbb{R}_+;L^{\alpha}_{\rm loc}),\label{astrongLp}
\end{align}
for any $1<\alpha<\infty$. Combining \eqref{A5.65}--\eqref{astrongLp}, one can show the convergence of  the system \eqref{I-1} satisfied by $(\rho^{\varepsilon},u^{\varepsilon},B^{\varepsilon},\nabla P^{\varepsilon})$ to \eqref{I-1} for $(\rho,u,B,\nabla P)$ in the sense of distributions, and therefore $(\rho,u,B,\nabla P)$ is indeed a global weak solution to \eqref{I-1} supplemented with the original initial data $(\rho_0,u_0,B_0)$. For brevity, we omit some details here. Applying Propositions 
\ref{P5.1}--\ref{P5.3} to $(\rho,u,B,\nabla P)$ again, one can show that $(\rho,u,B,\nabla P)$ satisfies the bounds exhibited in \eqref{1.7}--\eqref{1.9}. We also have the property $u, B\in C(\mathbb{R}_{+};\dot{B }^{-1+3/p}_{p,1})$ due to Lemmas \ref{L2.8}--\ref{L2.9}.

To complete the proof of Theorem \ref{T2.3}, we finally explain the uniqueness part for $1<p\leq 2$. Let $(\rho_1,u_1,B_1,\nabla P_1)$ and $(\rho_2,u_2,B_2,\nabla P_2)$ be two solutions to \eqref{I-1} on $[0,T]\times\mathbb{R}^3$ subject to the same initial data $(\rho_0,u_0,B_0)$. We shall apply Theorem \ref{thmunique} to achieve the goal.
According to \eqref{A5.19} and \eqref{uniquep}, in order to verify the conditions in Theorem \ref{thmunique}, it suffices to justify $\eqref{1.23}_5$--$\eqref{1.23}_7$ for the difference $(\rho_1-\rho_2,u_1-u_2,B_1-B_2)$. To this end, as in \cite{HSWZ-arXiv-2024-08},  we first integrate the equations of $\rho_i$ ($i=1,2$) and taking the $\dot{W}^{-1,3}$ norms for $\rho_i$ ($i=1,2$) to get
\begin{align}
\|(\rho^i-\rho_0)(t)\|_{\dot{W}^{-1,3}}\leq \int_0^t \|\rho^i u^i\|_{L^3}{\rm d}t'\leq \|\rho^i\|_{L^{\infty}_t(L^{\infty})}\|u^i\|_{L^{\infty}_t(\dot{B}^{-1+3/p}_{p,1})}<\infty,\label{3.64}
\end{align}
where we have used the embedding chain $\dot{B}^{-1+3/p}_{p,1} (1<p\leq2) \hookrightarrow \dot{H}^{1/2}\hookrightarrow L^3$. This implies that $t^{-3/4}(\rho_1-\rho_2)\in L^{\infty}(0,T;\dot{W}^{-1,3})$. On the other hand,  we obtain
\begin{align}
&\|\sqrt{\rho_1}(u_1-u_2)(t)\|_{L^2}^2\nonumber\\
=\,&\|\sqrt{\rho_0}(u_1-u_2)(t)\|_{L^2}^2+\|\sqrt{\rho_1-\rho_0}(u_1-u_2)(t)\|_{L^2}^2\nonumber\\
\leq\,& C \|\sqrt{\rho_0}(u_1-u_0)(t)\|_{L^2}^2+\|\sqrt{\rho_0}(u_2-u_0)(t)\|_{L^2}^2+C\|\sqrt{\rho_1-\rho_0}(u_1-u_2)(t)\|_{L^2}^2.\label{3.66}
\end{align}
Here, for $i=1,2$, it follows that
\begin{align}
\|\sqrt{\rho_0}(u_i-u_0)(t)\|_{L^2}&=\Big\|\sqrt{\rho_0}\int_0^t \partial_s u_i\,{\rm d}s \Big\|_{L^2}\nonumber\\
&\leq C\int_0^t s^{-1/4}\|s^{1/4}\sqrt{\rho_0} \partial_s u_i\|_{L^2}{\rm d}s\nonumber\\
&\leq C\Big(\int_0^t s^{-1/2}\,{\rm d}s\Big)^{1/2} \|s^{1/4}\sqrt{\rho_0} \partial_s u_i\|_{L^2_t(L^2)}\nonumber\\
&\leq C \|s^{1/4}\sqrt{\rho_0} \partial_s u_i\|_{L^2_t(L^2)}<\infty,\label{3.67}
\end{align}
where we have used
\begin{align*}
 & \|s^{1/4}\sqrt{\rho_0} \partial_s u_i\|_{L^2_t(L^2)}^2\\
 =\,&\int_0^t\int_{\mathbb{R}^3} s^{1/2}\rho_i |\partial_s u_i|^2\,{\rm d}x{\rm d}s+\int_0^t\int_{\mathbb{R}^3} s^{1/2}(\rho_0-\rho_i) |u_t|^2\,{\rm d}x {\rm d}s\\
 \leq \,&\|s^{1/4}\sqrt{\rho_i}\partial_s u_i\|_{L^2_s(L^2)}^2+C\|s^{-1}(\rho_i-\rho_0)\|_{L^{\infty}_t(\dot{W}^{-1,3})}\|s^{3/4}\nabla \partial_s u_i\|_{L^2_t(L^2)}^2<\infty,
\end{align*}
derived from \eqref{uniquep} and \eqref{3.64}. Similarly, one also has
\begin{align}
\|(B_1-B_2)(t)\|_{L^2}&\leq C\|(B_1-B_0)(t)\|_{L^2}^2+C\|(B_2-B_0)(t)\|_{L^2}\nonumber\\
&\leq C\int_0^t\|(\partial_s B_1,\partial_s B_2)(s)\|_{L^2}\,{\rm d}s\nonumber\\
&\leq C\Big(\int_0^t s^{-1/2}\,{\rm d}s\Big)^{1/2} \|s^{1/4}(\partial_s B_1,\partial_s B_1)\|_{L^2_t(L^2)}<\infty.\label{3.68}
\end{align}
By \eqref{3.66}--\eqref{3.68}, one gets $\sqrt{\rho_1}(u_1-u_2), B_1-B_2\in L^{\infty}(0,T;L^2)$. Finally, we have $\nabla (u_1-u_2),\, \nabla (B_1-B_2)\in  L^{2}(0,T;L^2)$ due to the facts that $\nabla u_1,\nabla u_2\in L^4(0,T;L^2)$ and $T<\infty$. Thus, all the conditions in \eqref{1.23} are satisfied, and hence Theorem \ref{thmunique} implies that  $(\rho_{1},u_{1},B_{1},\nabla P_{1})(t,x)=(\rho_{2},u_{2},B_{2},\nabla P_{2})(t,x)$  on $[0,T]\times\mathbb{R}^3$. The whole proof of Theorem \ref{T2.3} is completed. 
\end{proof}

\section{Proof of Theorem \ref{T2.1}}

The purpose of this section is to show Theorem \ref{T2.1}, which relies on the {\emph{a priori}} bounds established  Propositions
\ref{P3.1}--\ref{P3.6}. We shall follow the rescaling argument \eqref{2.1}--\eqref{2.2} and always assume $\mu=\nu=1$ in the computations.

\begin{prop}\label{P3.1}
Let $(\rho,u,B,\nabla P)$ be a smooth solution of the system \eqref{I-1} on $[0,T^*)\times\mathbb{R}^{3}$ with some given  $T^*>0$. There exists a uniform constant $\varepsilon_2>0$ such that if $(\rho_0,u_0,B_0)$ satisfies \eqref{A1.2} and 
\begin{align*}
\|(u_0,B_0)\|_{\dot{B}_{2,1}^{1/2}}<\varepsilon_2, 
\end{align*}
then it holds that
\begin{align}\label{A2.1}
\|(u,B)\|_{\widetilde{L}_{T}^{\infty}(\dot{B}^{1/2}_{2,1})}+\|(\nabla u,\nabla B)\|_{\widetilde{L}_{T}^{2}(\dot{B}^{1/2}_{2,1})}\leq C\|(u_{0},B_0)\|_{\dot{B}^{1/2}_{2,1}},
\end{align}
where $C>0$ is a generic constant.
\end{prop}
\begin{proof}
First, in view of \eqref{A1.2} and the classical theory on transport equation,
\eqref{A1.4} holds for $0<t<T^{*}$.
Then, we consider the following coupled system of $(u_{j},B_j,\nabla P_{j})$ 
\begin{equation}\label{A2.2}
\left\{
\begin{aligned}
& \rho(\partial_tu_j+u\cdot\nabla u_j)-\Delta u_j+\nabla P_j=B\cdot\nabla B_j, \\
& \partial_t B_j+u\cdot \nabla B_j-\Delta B_j=B\cdot\nabla u_j, \\
& \dive u_j=\dive B_j=0,\\
& u_j|_{t=0}=\dot{\Delta}_j u_0, \quad B_j|_{t=0}=\dot{\Delta}_j B_0.
\end{aligned}\right.
\end{equation}
One infers from the uniqueness of local smooth solution to \eqref{I-1} that
\begin{gather}\label{A2.3}
u=\sum_{j\in\mathbb Z}u_j, \quad B=\sum_{j\in\mathbb{Z}}B_j,      \quad\nabla P=\sum_{j\in\mathbb Z}\nabla P_j.    
\end{gather}
Then, by taking $L^{2}$ inner product of \eqref{A2.2} with $u_{j}$ and $B_j$ respectively and using \eqref{I-1}$_1$, we get
\begin{align}\label{A2.4}
\frac{1}{2}\frac{\rm d}{{\rm d}t}\bigg(\int_{\mathbb{R}^3}\rho|u_j|^2\, {\rm d}x+\int_{\mathbb{R}^3}|B_j|^2\, {\rm d}x\bigg)+\int_{\mathbb{R}^3}|\nabla u_j|^2\, {\rm d}x+\int_{\mathbb{R}^3}|\nabla B_j|^2\, {\rm d}x=0.   
\end{align}
Integrating \eqref{A2.4} over $[0,t]$ gives rise to
\begin{align*}
\frac{1}{2}\big(&\|\sqrt{\rho}u_{j}(t)\|_{L^{2}}^{2}+\|B_{j}(t)\|_{L^{2}}^{2}\big)+\|\nabla u\|_{L_{t}^{2}(L^{2})}^{2}+\|\nabla B\|_{L_{t}^{2}(L^{2})}^{2}\nonumber\\
&=\frac{1}{2}\big(\|\sqrt{\rho_{0}}\dot{\Delta}_{j}u_{0}\|_{L^{2}}^{2}+\|\dot{\Delta}_{j}B_{0}\|_{L^{2}}^{2}\big),
\end{align*}
and thus it yields
\begin{align}\label{A2.6}
\|(u_j,B_j)\|_{L_t^\infty(L^2)}+\|(\nabla u_j,\nabla B_j)\|_{L_t^2(L^2)}
&\leq\, C\|\dot{\Delta}_j (u_0,B_0)\|_{L^2}\nonumber\\
&\leq Cd_j2^{-j/2}\|(u_0,B_0)\|_{\dot{B}^{1/2}_{2,1}}.  
\end{align}
Taking $L^{2}$ inner product of  \eqref{A2.2} with $\partial_{t}u_{j}$ and $\partial_{t}B_{j}$ respectively leads to 
\begin{align}\label{A2.7}
&\frac{1}{2}\frac{\rm d}{{\rm d}t}\big(\|\nabla u_{j}(t)\|_{L^{2}}^{2}+\|\nabla B_{j}(t)\|_{L^{2}}^{2}\big)+\|\sqrt{\rho}\partial_{t}u_{j}\|_{L^{2}}^{2}+\|\partial_{t}B_{j}\|_{L^{2}}^{2}\nonumber\\
=\,&-\int_{\mathbb{R}^{3}}(\rho u\cdot\nabla u_{j})\cdot \partial_{t}u_{j}\, {\rm d}x
+\int_{\mathbb{R}^{3}}(B \cdot\nabla B_{j})\cdot \partial_{t}u_{j}\, {\rm d}x\nonumber\\
\,&-\int_{\mathbb{R}^{3}}( u\cdot\nabla B_{j})\cdot \partial_{t}B_{j}\, {\rm d}x
+\int_{\mathbb{R}^{3}}(B \cdot\nabla u_{j})\cdot \partial_{t}B_{j}\, {\rm d}x\nonumber\\
\leq\,&C \|u\|_{L^3}\|\nabla u_{j}\|_{L^{6}}\|\sqrt{\rho}\partial_{t}u_j\|_{L^2}+C\|B\|_{L^3}\|\nabla B_{j}\|_{L^{6}}\|\sqrt{\rho}\partial_{t}u_j\|_{L^2}\nonumber\\
&\quad+C \|u\|_{L^3}\|\nabla B_{j}\|_{L^{6}}\|\partial_{t}B_j\|_{L^2}+C\|B\|_{L^3}\|\nabla u_{j}\|_{L^{6}}\|\partial_{t}B_j\|_{L^2}\nonumber\\
\leq\,& C\|(u,B)\|_{\dot{H}^{1/2}}\|(\nabla^{2}u_{j}, \nabla^{2} B_j)\|_{L^{2}}\|(\sqrt{\rho}\partial_{t}u_{j},\partial_t B_j)\|_{L^{2}}^{2}.
\end{align}

We look at the elliptic structure for $\eqref{A2.2}$:
\begin{equation}\label{A2.42}
\left\{
\begin{aligned}
&-\Delta u_j+\nabla P_j=-\rho (\partial_t u_j+u\cdot \nabla u_j)+B\cdot\nabla B_j, \\
&-\Delta B_j=-(\partial_t B_j+u\cdot\nabla B_j)+B\cdot\nabla u_j,\\
&\dive u_j=\dive B_j=0.
\end{aligned}\right.
\end{equation}
On the one hand, it follows from the standard Stokes estimates for \eqref{A2.42}, that
\begin{align*}
\|\nabla^{2}u_{j}\|_{L^{2}}+\|\nabla P_{j}\|_{L^{2}}\leq&\,C\big(\|\sqrt{\rho}\partial_{t}u_{j}\|_{L^{2}}+\|u\cdot\nabla u_{j}\|_{L^{2}}+\|B\cdot\nabla B_{j}\|_{L^{2}}\big)\nonumber\\
\leq&\,C\big(\|\sqrt{\rho}\partial_{t}u_{j}\|_{L^{2}}+\|(u,B)\|_{\dot{H}^{1/2}}\|(\nabla^{2}u_{j}, \nabla^{2} B_j)\|_{L^{2}}\big).
\end{align*}
On the other hand, it is direct to see that
\begin{align*}
\|\nabla^{2}B_j\|_{L^{2}}\leq&\,C\big(\|\partial_{t}B_{j}\|_{L^{2}}+\|u\cdot\nabla B_{j}\|_{L^{2}}+\|B\cdot\nabla u_{j}\|_{L^{2}}\big)             \nonumber\\
\leq&\,C\big(\|\partial_t B_j\|_{L^{2}}+\|(u,B)\|_{L^3}\|(\nabla u_{j}, \nabla  B_j)\|_{L^{6}}\big).
\end{align*}
Then, there exists a constant $C_2^*>0$ such that
\begin{align}
&\|\nabla^2(u_j,B_j)\|_{L^2}+\|\nabla P_j\|_{L^2}\nonumber\\
&\quad\leq C_2^*\big(\|(\sqrt{\rho}\partial_tu_j,\partial_tB_j)\|_{L^2}+\|(u,B)\|_{\dot{B}^{1/2}_{2,1}}\|(\nabla u_{j}, \nabla B_j)\|_{L^{6}}\big).\label{AA} 
\end{align}

We denote
\begin{align}\label{A2.9}
T_{2}^{*}:=\sup\bigl\{t\in (0,T^*)\,:\,\|(u,B)(t)\|_{\dot{H}^{1/2}}\leq  c_{2}\bigr\}.    
\end{align}
For $c_2$ being sufficiently small  such that $c_{2}C_2^*\leq {1}/{2}$ and for $t\leq T_{2}^{*}$,  we derive from \eqref{AA}  that 
\begin{align}\label{A2.10}
\|\nabla^2(u_j,B_j)\|_{L^2}+\|\nabla P_j\|_{L^2}\leq C\|(\sqrt{\rho}\partial_tu_j,\partial_tB_j)\|_{L^2}. 
\end{align}
Substituting the above inequality \eqref{A2.10} into \eqref{A2.7}, we infer that there exists a  constant $C^*_3=C^*_3(\rho_0)>0$ such that, for all $t\leq T_2^*$,
\begin{align}\label{3.13}
&\frac{1}{2}\frac{\rm d}{{\rm d}t}\big(\|\nabla u_{j}(t)\|_{L^{2}}^{2}+\|\nabla B_{j}(t)\|_{L^{2}}^{2}\big)+\|(\sqrt{\rho}\partial_{t}u_{j},\partial_t B_j)\|_{L^{2}}^{2}\nonumber\\
&\quad\leq C_3^*\|(u,B)\|_{\dot{H}^{1/2}}\|(\sqrt{\rho}\partial_{t}u_{j},\partial_t B_j)\|_{L^{2}}^{2}. 
\end{align}
Using \eqref{3.13} and choosing $c_2$ in \eqref{A2.9} such that 
$$
c_2:=\frac{1}{2}\min\Big\{\frac{1}{C_2^*},\frac{1}{C_3^*}\Big\},
$$
we arrive at
\begin{align}\label{A2.12}
\frac{\rm d}{{\rm d}t}\big(\|\nabla u_{j}(t)\|_{L^{2}}^{2}+\|\nabla B_{j}(t)\|_{L^{2}}^{2}\big)+\|(\sqrt{\rho}\partial_{t}u_{j},\partial_t B_j)\|_{L^{2}}^{2}\leq0,\quad t\leq T_2^*.
\end{align}
From  \eqref{A2.10} and \eqref{A2.12}, we infer that
\begin{align}\label{A2.13}
&\|(\nabla u_{j},\nabla B_j)\|_{L^{\infty}(L^{2})}+\|(\nabla^{2}u_{j},\nabla^2 B_j)\|_{L_{t}^{2}(L^{2})} +\|(\sqrt{\rho}\partial_{t}u_{j},\partial_t B_j,\nabla P_{j})\|_{L_{t}^{2}(L^{2})}\nonumber\\
&\quad\leq C\|(\nabla u_{j},\nabla B_j)(0)\|_{L^{2}}
\leq C\|\nabla\dot{\Delta}_{j}u_{0}\|_{L^{2}}+C\|\nabla\dot{\Delta}_{j}B_0\|_{L^{2}}
\leq C d_{j}2^{j/2}\|(u_{0},B_0)\|_{\dot{B}^{1/2}_{2,1}}.    
\end{align}
By virtue of \eqref{A2.3}, \eqref{A2.6}, \eqref{A2.13}, and applying Bernstein's
inequality, we arrive at
\begin{align}\label{A2.14}
&\|\dot{\Delta}_{j}(u,B)\|_{L_{t}^{\infty}(L^{2})}+\|\nabla\dot{\Delta}_{j}(u,B)\|_{L_{t}^{2}(L^{2})}\nonumber\\
\leq\,&C\sum_{j^{\prime}\geq j}\big(\|\dot{\Delta}_{j}(u_{j^{\prime}},B_{j^{\prime}})\|_{L_{t}^{\infty}(L^{2})}
+\|\nabla\dot{\Delta}_{j}(u_{j^{\prime}},B_{j^{\prime}})\|_{L_{t}^{2}(L^{2})}\big) \nonumber\\
&+C2^{-j}\sum_{j^{\prime}\leq j}\big(\|\nabla\dot{\Delta}_{j}(u_{j^{\prime}},B_{j^{\prime}})\|_{L_{t}^{\infty}(L^{2})}
+\|\nabla^{2}\dot{\Delta}_{j}(u_{j^{\prime}},B_{j^{\prime}})\|_{L_{t}^{2}(L^{2})}\big) \nonumber\\
\leq\,&C\sum_{j^{\prime}\geq j}\big(\|(u_{j^{\prime}},B_{j^{\prime}})\|_{L_{t}^{\infty}(L^{2})}+\|(\nabla u_{j^{\prime}},B_{j^{\prime}})\|_{L_{t}^{2}(L^{2})}\big) \nonumber\\
&+C2^{-j}\sum_{j^{\prime}\leq j}\big(\|\nabla (u_{j^{\prime}},B_{j^{\prime}})\|_{L_{t}^{\infty}(L^{2})}
+\|\nabla^{2}(u_{j^{\prime}},B_{j^{\prime}})\|_{L_{t}^{2}(L^{2})}\big) \nonumber\\
\leq\,&C d_{j}2^{-j/2}\|(u_{0},B_0)\|_{\dot{B}^{1/2}_{2,1}}.
\end{align}
This means that
$$
\|(u,B)\|_{\widetilde{L}_{t}^{\infty}(\dot{B}^{1/2}_{2,1})}+\|(\nabla u,\nabla B)\|_{\widetilde{L}_{t}^{2}(\dot{B}^{1/2}_{2,1})}\leq C\|(u_{0},B_0)\|_{\dot{B}^{1/2}_{2,1}}
$$
holds for all $t\in (0,T_1^*)$.  Then, as $(u,B)\in C([0,T^*);\dot{B}^{1/2}_{2,1})$, a standard continuous argument leads to $T_2^*=T$ as long as $C\|(u_0,B_0)\|_{\dot{B}^{1/2}_{2,1}}<c_2$, and therefore the desired bound \eqref{A2.1} follows. 
\end{proof}

Next, we establish higher order estimates of $(u,B,\nabla P)$ as follows.
\begin{prop}
Under the assumptions of Proposition \ref{P3.1}, we have
\begin{align}\label{A2.20}
\|{t^{1/2}}(\nabla u,\nabla B)\|_{\widetilde{L}_{T}^{\infty}(\dot{B}^{{1}/{2}}_{2,1})}+\|{t^{1/2}}(\partial_{t}u,\partial_t B)\|_{\widetilde{L}_{T}^{2}(\dot{B}^{\frac{1}{2}}_{2,1})}\leq\,& C \|(u_{0},B_0)\|_{\dot{B}^{1/2}_{2,1}}.
\end{align}
\end{prop}

\begin{proof}
We deal with the inequality \eqref{A2.7} in the following way:
\begin{align*}
&\frac{1}{2}\frac{\rm d}{{\rm d}t}\big(\|\nabla u_{j}(t)\|_{L^{2}}^{2}+\|\nabla B_{j}(t)\|_{L^{2}}^{2}\big)+\|\sqrt{\rho}\partial_{t}u_{j}\|_{L^{2}}^{2}+\|\partial_{t}B_{j}\|_{L^{2}}^{2}\nonumber\\
\leq\, &C\|u\|_{L^{\infty}}\|\nabla u_{j}\|_{L^{2}}\|\sqrt{\rho}\partial_{t}u_{j}\|_{L^{2}}+C\|B\|_{L^{\infty}}\|\nabla B_{j}\|_{L^{2}}\|\sqrt{\rho}\partial_{t}u_{j}\|_{L^{2}} \nonumber\\
&+C\|u\|_{L^{\infty}}\|\nabla B_{j}\|_{L^{2}}\|\partial_{t}B_{j}\|_{L^{2}}+C\|B\|_{L^{\infty}}\|\nabla u_{j}\|_{L^{2}}\|\partial_{t}B_{j}\|_{L^{2}}                         \nonumber\\
\leq\,&C\|(u,B)\|_{L^{\infty}}\|\nabla(u_j,B_j)\|_{L^2}\|(\sqrt{\rho}\partial_tu_j,\partial_t B_j)\|_{L^2} \nonumber\\
\leq\, &C\|(u,B)\|_{\dot{B}^{3/2}_{2,1}}^{2}\|(\nabla u_{j},\nabla B_j)\|_{L^{2}}^{2}+\frac{1}{2}\|(\sqrt{\rho}\partial_{t}u_{j},\partial_t B_j)\|_{L^{2}}^{2},    
\end{align*}
which implies that
\begin{align}\label{A2.17}
&\frac{1}{2}\frac{\rm d}{{\rm d}t}\big(\|\nabla u_{j}(t)\|_{L^{2}}^{2}+\|\nabla B_{j}(t)\|_{L^{2}}^{2}\big)+\|\sqrt{\rho}\partial_{t}u_{j}\|_{L^{2}}^{2}+\|\partial_{t}B_{j}\|_{L^{2}}^{2}\nonumber\\
&\quad \leq C\|(u,B)\|_{\dot{B}^{3/2}_{2,1}}^{2}\|(\nabla u_{j},\nabla B_j)\|_{L^{2}}^{2}.  
\end{align}
Multiplying \eqref{A2.17} by $t$, we have
\begin{align*}
&\frac{\rm d}{{\rm d}t}\big(t\|(\nabla u_{j},\nabla B_j)(t)\|_{L^{2}}^{2}\big)+t\|(\sqrt{\rho}\partial_{t}u_{j},\partial_t B_j)\|_{L^{2}}^{2}\nonumber\\
&\quad \leq C\|(\nabla u_{j},\nabla B_j)(t)\|_{L^{2}}^{2}+C\|(u,B)\|_{\dot{B}^{3/2}_{2,1}}^{2}t\|(\nabla u_{j},\nabla B_j)\|_{L^{2}}^{2}.
\end{align*}
Applying Gronwall's inequality and then using \eqref{A2.1}, \eqref{A2.6} and \eqref{A2.10}, we arrive at
\begin{align}
&\|{t^{1/2}}(\nabla u_{j},\nabla B_j)\|_{L_{t}^{\infty}(L^{2})}^{2}+\|{t^{1/2}}(\nabla^2u_j,\nabla^2 B_j,\nabla P_j,\partial_{t}u_{j},\partial_t B_j)\|_{L_{t}^{2}(L^{2})}^{2}\nonumber\\
\leq\, &C\|(\nabla u_{j},\nabla B_j)\|_{L_{t}^{2}(L^{2})}^{2}\exp\Big\{C\|(u,B)\|_{L_{t}^{2}(\dot{B}^{3/2}_{2,1})}^{2}\Big\} \nonumber\\
\leq\, &Cd_{j}^{2}2^{-j}\|(u_{0},B_0)\|_{\dot{B}^{1/2}_{2,1}}^{2}.\label{4.22}
\end{align}



In order to derive \eqref{A2.20}, following a similar line in \eqref{A2.14}, we shall establish the weighted $L^2$-estimate of $(\nabla^2u,\nabla^2 B)$. To this end, we apply $\partial_{t}$ to  \eqref{A2.2}$_1$ and \eqref{A2.2}$_2$ to get
\begin{align}\label{A2.21}
&\rho\partial_t^2u_j+\rho u\cdot\nabla\partial_tu_j-\Delta\partial_tu_j+\nabla\partial_t P_j\nonumber\\
&\quad\quad=-\rho_tD_tu_j-\rho \partial_t u\cdot\nabla u_j+\partial_t B\cdot\nabla B_j+B\cdot\nabla\partial_t B_j,\\ \label{A2.22}
&\rho\partial_t^2B_j+\partial_t u\cdot\nabla B_j+u\cdot\nabla\partial_tB_j-\Delta\partial_tB_j
=\partial_t B\cdot\nabla u_j+B\cdot\nabla\partial_t u_j.
\end{align}
By taking $L^{2}$ inner product of \eqref{A2.21} and \eqref{A2.22} with $\partial_{t}u_{j}$ and $\partial_t B_j$ respectively and using \eqref{I-1}$_1$, we get
\begin{align}\label{A2.23}
&\frac{1}{2}\frac{\rm d}{{\rm d}t}\bigg(\int_{\mathbb{R}^{3}}\rho|\partial_{t}u_{j}|^{2}\, {\rm d}x+\int_{\mathbb{R}^{3}}|\partial_{t}B_{j}|^{2}\, {\rm d}x\bigg)+\|\nabla\partial_{t}u_{j}\|_{L^{2}}^{2}+\|\nabla\partial_{t}B_{j}\|_{L^{2}}^{2}\nonumber\\
=\,&-\int_{\mathbb{R}^{3}}(\rho_{t}D_{t}u_{j})\cdot
\partial_{t}u_{j}\, {\rm d}x-\int_{\mathbb{R}^{3}}(\rho \partial_t u\cdot\nabla u_{j})\cdot\partial_{t}u_{j}\, {\rm d}x-\int_{\mathbb{R}^{3}}(\partial_t B\cdot\nabla B_{j})\cdot\partial_{t}u_{j}\, {\rm d}x\nonumber\\
&\,+\int_{\mathbb{R}^{3}}(\partial_t B\cdot\nabla u_{j})\cdot\partial_{t}B_{j}\, {\rm d}x-\int_{\mathbb{R}^{3}}(\partial_t u\cdot\nabla B_{j})\cdot\partial_{t}B_{j}\, {\rm d}x\nonumber\\
\equiv:\,&\sum_{j=1}^5I_j.
\end{align}
Below we handle with the terms on the right-hand side of \eqref{A2.23} one by one. First, by integration by parts and $\eqref{I-1}_1$, one has
\begin{align}\label{A2.24}
I_1=\,&-\int_{\mathbb{R}^{3}}(\rho_{t}\partial_{t}u_{j})\cdot\partial_{t}u_{j}\, {\rm d}x -\int_{\mathbb{R}^{3}}(\rho_{t}u\cdot\nabla u_{j})\cdot\partial_{t}u_{j}\, {\rm d}x\nonumber\\
=\,&-2\int_{\mathbb{R}^{3}}(\rho u\cdot\nabla\partial_{t}u_{j})\cdot\partial_{t}u_{j}\, {\rm d}x -\int_{\mathbb{R}^{3}}(\rho u\cdot\nabla u)\cdot\nabla u_j\cdot\partial_{t}u_{j}\, {\rm d}x\nonumber\\
\,&-\int_{\mathbb{R}^{3}}(\rho (u\otimes u):\nabla^2u_{j})\cdot\partial_{t}u_{j}\, {\rm d}x-\int_{\mathbb{R}^{3}}(\rho u \cdot\nabla \partial_tu_{j})\cdot(u\cdot \nabla u_{j})\, {\rm d}x,
\end{align}
which, together with \eqref{A2.10} and Young's inequality, gives rise to
\begin{align}\label{A2.25}
I_1\leq\,& C\|u\|_{L^{\infty}}\|\sqrt{\rho}\partial_{t}u_{j}\|_{L^{2}}\|\nabla\partial_{t}u_{j}\|_{L^{2}}+C\|u\|_{L^{\infty}}\|\nabla u\|_{L^{3}}\|\nabla u_{j}\|_{L^{6}}\|\sqrt{\rho}\partial_{t}u_{j}\|_{L^{2}}  \nonumber\\
\,&+C\|u\|_{L^{\infty}}^2\|\nabla^2 u_j\|_{L^{2}}\|\sqrt{\rho}\partial_t u_j\|_{L^{2}}+C\|u\|_{L^{\infty}}\|\nabla u_j\|_{L^{6}}\|u\|_{L^{3}}\|\nabla\partial_{t}u_{j}\|_{L^{2}}\nonumber\\
\leq\,&C\|u\|_{\dot{B}^{3/2}_{2,1}}^{2}\|(\sqrt{\rho}\partial_{t}u_{j},\partial_t B_j)\|_{L^{2}}^{2}+C\|u\|_{\dot{B}^{1/2}_{2,1}}^{2}\|u\|_{\dot{B}^{3/2}_{2,1}}^{2}\|(\sqrt{\rho}\partial_{t}u_{j},\partial_t B_j)\|_{L^{2}}^{2}+\frac{1}{4}\|\nabla\partial_{t}u_{j}\|_{L^{2}}^{2}\nonumber\\
\leq\,&C\Big(1+\|u\|_{\dot{B}^{1/2}_{2,1}}^{2}\Big)\|u\|_{\dot{B}^{3/2}_{2,1}}^{2}\|(\sqrt{\rho}\partial_{t}u_{j},\partial_t B_j)\|_{L^{2}}^{2}+\frac{1}{4}\|\nabla\partial_{t}u_{j}\|_{L^{2}}^{2}.
\end{align}
Similarly, we obtain
\begin{align}\label{A2.26}
I_2+I_3\leq\,& C\|(\partial_t u,\partial_t B)\|_{L^{2}}\|(\nabla u_{j},\nabla B_j)\|_{L^{3}}\|\partial_{t}u_{j}\|_{L^{6}}\nonumber\\ 
\leq\,& C\|(\partial_t u,\partial_t B)\|_{L^{2}}\|(\nabla u_{j},\nabla B_j)\|_{L^{2}}^{1/2}\|(\nabla^{2}u_{j}, \nabla^{2} B_j)\|_{L^{2}}^{1/2}\|\nabla\partial_{t}u_{j}\|_{L^{2}} \nonumber\\
\leq\,& C\|(\partial_t u,\partial_t B)\|_{L^{2}}^{2}\|(\nabla u_{j},\nabla B_j)\|_{L^{2}}\|(\nabla^{2}u_{j}, \nabla^{2} B_j)\|_{L^{2}}+\frac{1}{4}\|\nabla\partial_{t}u_{j}\|_{L^{2}}^{2} \nonumber\\
\leq\,& C\|(\partial_t u,\partial_t B)\|_{L^{2}}^{2}\|(\nabla u_{j},\nabla B_j)\|_{L^{2}}\|(\sqrt{\rho}\partial_{t}u_{j},\partial_t B_j)\|_{L^{2}}+\frac{1}{4}\|\nabla\partial_{t}u_{j}\|_{L^{2}}^{2},
\end{align}
and
\begin{align}
I_4+I_5\leq\,& C\|(\partial_t u,\partial_t B)\|_{L^{2}}\|(\nabla u_{j},\nabla B_j)\|_{L^{3}}\|\partial_{t}B_{j}\|_{L^{6}}\nonumber\\ 
\leq\,& C\|(\partial_t u,\partial_t B)\|_{L^{2}}\|(\nabla u_{j},\nabla B_j)\|_{L^{2}}^{1/2}\|(\nabla^{2}u_{j}, \nabla^{2} B_j)\|_{L^{2}}^{1/2}\|\nabla\partial_{t}B_{j}\|_{L^{2}} \nonumber\\
\leq\,& C\|(\partial_t u,\partial_t B)\|_{L^{2}}^{2}\|(\nabla u_{j},\nabla B_j)\|_{L^{2}}\|(\nabla^{2}u_{j}, \nabla^{2} B_j)\|_{L^{2}}+\frac{1}{2}\|\nabla\partial_{t}B_{j}\|_{L^{2}}^{2} \nonumber\\
\leq\,& C\|(\partial_t u,\partial_t B)\|_{L^{2}}^{2}\|(\nabla u_{j},\nabla B_j)\|_{L^{2}}\|(\sqrt{\rho}\partial_{t}u_{j},\partial_t B_j)\|_{L^{2}}+\frac{1}{2}\|\nabla\partial_{t}B_{j}\|_{L^{2}}^{2}.\label{4.291}
\end{align}
Inserting \eqref{A2.25}--\eqref{4.291} into \eqref{A2.24}, we arrive at
\begin{align}\label{A2.28}
&\frac{\rm d}{{\rm d}t}\bigg(\int_{\mathbb{R}^{3}}\rho|\partial_{t}u_{j}|^{2}\, {\rm d}x+\int_{\mathbb{R}^{3}}|\partial_{t}B_{j}|^{2}\, {\rm d}x\bigg)+\|\nabla\partial_{t}u_{j}\|_{L_{t}^{2}(L^{2})}^{2}
+\|\nabla\partial_{t}B_{j}\|_{L_{t}^{2}(L^{2})}^{2}\nonumber\\
 &\quad  \leq C\Big(1+\|u\|_{\dot{B}^{1/2}_{2,1}}^{2}\Big)\|u\|_{\dot{B}^{3/2}_{2,1}}^{2}\|(\sqrt{\rho}\partial_{t}u_{j},\partial_t B_j)\|_{L^{2}}^{2}\nonumber\\
&\quad \quad +C\|(\partial_t u,\partial_t B)\|_{L^{2}}^{2}\|(\nabla u_{j},\nabla B_j)\|_{L^{2}}\|(\sqrt{\rho}\partial_{t}u_{j},\partial_t B_j)\|_{L^{2}}.
\end{align}
Multiplying \eqref{A2.28} by $t$ and using \eqref{A2.1} gives
\begin{align*}
&\frac{\rm d}{{\rm d}t}\|{t^{1/2}}(\sqrt{\rho}\partial_{t}u_{j},\partial_t B_j)\|_{L^{2}}^{2}+\|{t^{1/2}}(\nabla\partial_{t}u_{j},\nabla\partial_{t}B_j)\|_{L_{t}^{2}(L^{2})}^{2}\nonumber\\
&\quad \leq \|(\sqrt{\rho}\partial_{t}u_{j},\partial_t B_j)\|_{L^{2}}^{2}+C\|t^{1/4}(\partial_t u,\partial_t B)\|_{L^{2}}^{2}\|(\nabla u_{j},\nabla B_j)\|_{L^{2}}^{2}\nonumber\\
&\quad\quad+C\Big(\|u\|_{\dot{B}^{3/2}_{2,1}}^{2}+\|t^{1/4}(\partial_t u,\partial_t B)\|_{L^{2}}^{2}\Big)\|{t^{1/2}}({\sqrt{\rho}}\partial_{t}u_{j},\partial_t B_j)\|_{L^{2}}^{2}.
\end{align*}
Applying Gronwall's inequality and making use of \eqref{A1.4}, we achieve
\begin{align*}
&\|{t^{1/2}}(\partial_{t}u_{j},\partial_t B_j)\|_{L_{t}^{\infty}(L^{2})}^{2}+\|{t^{1/2}}(\nabla\partial_{t}u_{j},\nabla\partial_t B_j)\|_{L_{t}^{2}(L^{2})}^{2}\nonumber\\
& \quad \leq C\exp\bigg\{\int_{0}^{t}\Big(\|u \|_{\dot{B}^{3/2}_{2,1}}^{2}+\| ^{1/4}(\partial_t u,\partial_t B) \|_{L^{2}}^{2}\Big){{\rm d}t}^{\prime}\bigg\}\nonumber\\
&\quad \quad \times\Big(\|(\sqrt{\rho}\partial_{t}u_{j},\partial_t B_j)\|_{L_{t}^{2}(L^{2})}^{2}+\|t^{1/4}(\partial_t u,\partial_t B)\|_{L_{t}^{2}(L^{2})}^{2}\|(\nabla u_{j},\nabla B_j)\|_{L_{t}^{\infty}(L^{2})}^{2}\Big).
\end{align*}
With the aid of \eqref{A2.13} and \eqref{4.22}, it holds that
\begin{align*}
\|t^{1/4}(\partial_{t}u_{j},\partial_t B_j)\|_{L_{t}^{2}(L^{2})} \leq&\,\|(\partial_{t}u_{j},\partial_{t}B_j)\|_{L_{t}^{2}(L^{2})}^{1/2}\|{t^{1/2}}(\partial_{t}u_{j},\partial_t B_j)\|_{L_{t}^{2}(L^{2})}^{1/2}\nonumber\\ 
\leq&\, Cd_{j}\|(u_{0},B_0)\|_{\dot{B}^{1/2}_{2,1}}, 
\end{align*}
from which we infer
\begin{align}\label{A2.32}
\|t^{1/4}(\partial_{t}u,\partial_t B)\|_{L_{t}^{2}(L^{2})}\leq\sum_{j\in\mathbb{Z}}\|t^{1/4}(\partial_{t}u_{j},\partial_t B_j)\|_{L_{t}^{2}(L^{2})}\leq C\|(u_{0},B_0)\|_{\dot{B}^{1/2}_{2,1}}.    
\end{align}
It follows from \eqref{A2.1}, \eqref{A2.13} and \eqref{A2.32} that
\begin{align}\label{A2.33}
&\|{t^{1/2}}(\partial_{t}u_{j},\partial_t B_j)\|_{L_{t}^{\infty}(L^{2})}+\|{t^{1/2}}(\nabla\partial_{t}u_{j},\nabla\partial_{t}B_j)\|_{L_{t}^{2}(L^{2})}\nonumber\\
\leq\,& Cd_{j}2^{j/2}\exp\Big\{C\|(u_{0},B_0)\|_{\dot{B}^{1/2}_{2,1}}^{2}\Big\}\|(u_{0},B_0)\|_{\dot{B}^{1/2}_{2,1}} \nonumber\\
\leq\,& Cd_{j}2^{j/2}\|(u_{0},B_0)\|_{\dot{B}^{1/2}_{2,1}}.
\end{align}
Combining \eqref{4.22} and \eqref{A2.33} and carrying out a frequency splitting argument similar to \eqref{A2.14}, we get \eqref{A2.20}.
\end{proof}

\begin{prop}\label{prop4.33}
Under the assumptions of Proposition \ref{P3.1}, we have for any $T\in(0,T^*)$, 
\begin{align}\label{A2.35}
\|t(D_{t} u, D_{t} B)\|_{\widetilde{L}_{t}^{\infty}(\dot{B}^{1/2}_{2,1})}+\|t (\nabla D_{t} u, \nabla D_{t} B)\|_{\widetilde{L}_{t}^{2}(\dot{B}^{1/2}_{2,1})}\leq&\, C\|(u_{0},B_0)\|_{\dot{B}^{1/2}_{2,1}}.
\end{align}

\end{prop}
\begin{proof}
We first handle the weighted $L^2$ estimates of $(D_{t} u_j, D_{t} B_j)$. Multiplying \eqref{A2.28} by $t^2$ gives
\begin{align}\label{A2.37}
&\frac{\rm d}{{\rm d}t}\big(\|t\sqrt{\rho}\partial_{t}u_{j}\|_{L^{2}}^{2}+\|t\partial_{t}B_{j}\|_{L^{2}}^{2}\big)+\|t\nabla\partial_{t}u_{j}\|_{L^{2}}^{2}+\|t\nabla\partial_{t}B_{j}\|_{L^{2}}^{2}\nonumber\\
&\quad\leq 2\|{t^{1/2}}(\sqrt{\rho}\partial_{t}u_{j},\partial_t B_j)\|_{L^{2}}^{2}+C\|t^{1/4}(\partial_t u,\partial_t B)\|_{L^{2}}^{2}\|{t^{1/2}}(\nabla u_{j},\nabla B_j)\|_{L^{2}}^{2}\nonumber\\
&\quad\quad+C\big(\|u\|_{\dot{B}^{3/2}_{2,1}}^{2}+\|t^{1/4}(\partial_t u,\partial_t B)\|_{L^{2}}^{2}\big)\|t(\sqrt{\rho}\partial_{t}u_{j},\partial_t B_j)\|_{L^{2}}^{2}.
\end{align}
Applying Gronwall's inequality to \eqref{A2.37} implies that
\begin{align*}
&\|t(\partial_{t}u_{j},\partial_{t}B_j)\|_{L_{t}^{\infty}(L^{2})}^{2}+\|t(\nabla\partial_{t}u_{j},\nabla\partial_{t}B_j)\|_{L_{t}^{2}(L^{2})}^{2}\nonumber
\\
&\quad \leq\exp\bigg\{C\int_{0}^{t}\Big(\|u \|_{\dot{B}^{3/2}_{2,1}}^{2}+\|t^{1/4}(\partial_t u,\partial_t B) \|_{L^{2}}^{2}\Big){{\rm d}t}^{\prime}\bigg\}\nonumber\\
\,&\quad \quad \times\Big(\|{t^{1/2}}(\partial_{t}u_{j},\partial_t B_j)\|_{L_{t}^{2}(L^{2})}^{2}+C\|t^{1/4}(\partial_t u,\partial_t B)\|_{L_{t}^{2}(L^{2})}^{2}\|{t^{1/2}}(\nabla u_{j},\nabla B_j)\|_{L_{t}^{\infty}(L^{2})}^{2}\Big).
\end{align*}
This, together with \eqref{A2.1} and \eqref{A2.32}, leads to
\begin{align}\label{A2.39}
\|t(\partial_tu_j,\partial_t B_j)\|_{L_t^\infty(L^2)}^2+\|t\nabla\partial_t(u_j,B_j)\|_{L_t^2(L^2)}^2\leq Cd_j^22^{-j}\|(u_0,B_0)\|_{\dot{B}^{1/2}_{2,1}}^2.  
\end{align}
In view of \eqref{A2.1} and \eqref{A2.39}, it holds that
\begin{align*}
&\|t(D_{t} u_{j}, D_{t} B_j)\|_{L_{t}^{\infty}(L^{2})}\nonumber\\
 \leq\,&C \|t(\partial_{t}u_{j},\partial_t B_j)\|_{L_{t}^{\infty}(L^{2})}+\|{t^{1/2}}u\|_{L_{t}^{\infty}(\dot{B}^{3/2}_{2,1})}\|{t^{1/2}}(\nabla u_{j},\nabla B_j)\|_{L_{t}^{\infty}(L^{2})} \nonumber\\
\leq\,& Cd_{j}2^{-j/2}\|(u_{0},B_0)\|_{\dot{B}^{1/2}_{2,1}}.
\end{align*}
In addition, it follows from \eqref{A2.10} and standard embeddings that
\begin{align*}
&\|t (\nabla D_{t} u_{j}, \nabla D_{t} B_j)\|_{L_{t}^{2}(L^{2})}\\
\leq\, & C\|t(\nabla\partial_{t}u_{j},\nabla\partial_{t}B_j)\|_{L_{t}^{2}(L^{2})}+C\|u\|_{L_{t}^{2}(L^{\infty})}\|t(\nabla^{2}u_{j}, \nabla^{2} B_j)\|_{L_{t}^{\infty}(L^{2})}  \nonumber\\
\leq\,&C\|t(\nabla\partial_{t}u_{j},\nabla\partial_{t}B_j)\|_{L_{t}^{2}(L^{2})}+C\|u\|_{L_{t}^{2}(\dot{B}^{3/2}_{2,1})}\|t(\partial_{t}u_{j},\partial_t B_j)\|_{L_{t}^{\infty}(L^{2})} \nonumber\\
\leq\,& Cd_{j}2^{-j/2}\|(u_{0},B_0)\|_{\dot{B}^{1/2}_{2,1}}.
\end{align*}

Our next goal is to derive the weighted $L^2$ estimate of $(\nabla D_t u,\nabla D_t B)$. Applying the operator $D_t=\partial_t+u\cdot\nabla$ 
to \eqref{A2.2}$_1$ and \eqref{A2.2}$_2$, we get
\begin{align}\label{A2.45}
&\rho D_t^2u_j-\Delta D_tu_j+\nabla D_t P_j\nonumber\\
&\quad=-\Delta u\cdot\nabla u_j-2\sum_{i=1}^3\partial_iu\cdot\nabla\partial_iu_j+\nabla u\cdot\nabla P_j+D_t(B\cdot\nabla B_j):=f_j,
\end{align}
and
\begin{align}
& D_t^2B_j-\Delta D_tB_j=\,-\Delta u\cdot\nabla B_j-2\sum_{i=1}^3\partial_iu\cdot\nabla\partial_iB_j+D_t(B\cdot\nabla u_j):=h_j.\label{A2.46}
\end{align}
After taking the $L^{2}$ inner product of \eqref{A2.45} and \eqref{A2.46} with $D_{t}^{2}u_{j}$ and $D_t^2 B_j$ respectively, we deduce that
\begin{align*}
&\frac{1}{2}\frac{{\rm d}}{{\rm d}t}\big(\|\nabla D_t u_j(t)\|_{L^2}^2+\|\nabla D_t B_j(t)\|_{L^2}^2         \big)+\|\sqrt{\rho}D_t^2 u_j\|_{L^2}^2+\|D_t^2 B_j\|_{L^2}^2\nonumber\\
&\quad =\,-\int_{\mathbb{R}^3}\nabla D_t u_j\cdot([\nabla, D_t]D_t u_j)\mathrm{d}x
-\int_{\mathbb{R}^3}\nabla D_t B_j\cdot([\nabla, D_t]D_t B_j)\mathrm{d}x\nonumber\\
&\quad\quad \,-\int_{\mathbb{R}^3}\nabla D_t P_j\cdot D_t^2 u_j\mathrm{d}x+\int_{\mathbb{R}^3}f_j\cdot D_t^2 u_j\mathrm{d}x+\int_{\mathbb{R}^3}h_j\cdot D_t^2 B_j\mathrm{d}x,
\end{align*}
which yields
\begin{align}\label{A2.48}
&\frac{1}{2}\frac{{\rm d}}{{\rm d}t}\big(\|t\nabla D_t u_j(t)\|_{L^2}^2+\|t\nabla D_t B_j(t)\|_{L^2}^2         \big)+\|t\sqrt{\rho}D_t^2 u_j\|_{L^2}^2+\|tD_t^2 B_j\|_{L^2}^2\nonumber\\
  \leq\,&\|{t^{1/2}}\nabla D_t u_j\|_{L^2}^2+\|{t^{1/2}}\nabla D_t B_j\|_{L^2}^2-t^2\int_{\mathbb{R}^3}\nabla D_t u_j\cdot([\nabla, D_t]D_t u_j)\mathrm{d}x
\nonumber\\
&  -t^2\int_{\mathbb{R}^3}\nabla D_t B_j\cdot([\nabla, D_t]D_t B_j)\mathrm{d}x-t^2\int_{\mathbb{R}^3}\nabla D_t P_j\cdot D_t^2 u_j\mathrm{d}x+t^2\int_{\mathbb{R}^3}f_j\cdot D_t^2 u_j\mathrm{d}x\nonumber\\
& +t^2\int_{\mathbb{R}^3}h_j\cdot D_t^2 B_j\mathrm{d}x\nonumber\\
\equiv:\,&\|{t^{1/2}}\nabla D_t u_j\|_{L^2}^2+\|{t^{1/2}}\nabla D_t B_j\|_{L^2}^2+\sum_{j=6}^{10} I_j.
\end{align}

We handle the terms $I_j$ ($j=6, \dots, 10)$ one by one. Let us first estimate $\int _{0}^{t}I_9 {{\rm d}t}^{\prime }$. It follows from \eqref{A2.2} that
\begin{align*}
&\|{t^{1/2}}(\nabla^{2}u_{j},\nabla^2 B_j,\nabla P_{j})\|_{L_{t}^{2}(L^{6})}\nonumber\\
\leq& C\big(\|{t^{1/2}}\nabla\partial_{t}u_{j}\|_{L_{t}^{2}(L^{2})}+\|{t^{1/2}}u\cdot\nabla u_{j}\|_{L_{t}^{2}(L^{6})}+\|{t^{1/2}}B\cdot\nabla u_j\|_{L_{t}^{2}(L^{6})}\big) \nonumber\\
\,&+C\big(\|{t^{1/2}}\nabla\partial_{t}B_{j}\|_{L_{t}^{2}(L^{2})}+\|{t^{1/2}}u\cdot\nabla B_{j}\|_{L_{t}^{2}(L^{6})}+\|{t^{1/2}}B\cdot\nabla B_j\|_{L_{t}^{2}(L^{6})}\big)            \nonumber\\
\leq\,& C\big(\|{t^{1/2}}(\nabla\partial_{t}u_{j},\nabla\partial_{t}B_j)\|_{L_{t}^{2}(L^{2})}+\|(u,B)\|_{L_{t}^{2}(L^{\infty})}\|{t^{1/2}}(\nabla^{2}u_{j}, \nabla^{2} B_j)\|_{L_{t}^{\infty}(L^{2})}\big)\nonumber\\
\leq\,& C\Big(\|{t^{1/2}}(\nabla\partial_{t}u_{j},\nabla\partial_{t}B_j)\|_{L_{t}^{2}(L^{2})}+\|(u,B)\|_{L_{t}^{\infty}(\dot{B}^{3/2}_{2,1})}\|{t^{1/2}}(\nabla^{2}u_{j}, \nabla^{2} B_j)\|_{L_{t}^{\infty}(L^{2})}\Big),
\end{align*}
which, together with \eqref{A2.1} and \eqref{A2.33}, ensures that
\begin{align}\label{A2.50}
\|{t^{1/2}}(\nabla^{2}u_{j},\nabla^2 B_j,\nabla P_{j})\|_{L_{t}^{2}(L^{6})}\leq Cd_{j}2^{j/2}\|(u_{0},B_0)\|_{\dot{B}^{1/2}_{2,1}}.    
\end{align}
Similar to the proof of \cite[Proposition 2.3]{Zp-Adv-2020}, we can handle the estimate of $\|\nabla D_t P_j\|_{L^2}$ as follows:
\begin{align}\label{A2.51}
\|\nabla D_{t} P_{j}\|_{L^{2}}\leq C\big(\|\nabla\mathrm{Tr}(\nabla u\cdot\nabla u_{j})\big\|_{L^{2}}+\|\rho D_{t}^{2}u_{j}\|_{L^{2}}+\|f_{j}\|_{L^{2}}\big).    
\end{align}
Then, by virtue of \eqref{A2.20} and \eqref{A2.50}, we have
\begin{align}\label{A2.52}
&\|t\nabla\mathrm{Tr}(\nabla u\cdot\nabla u_{j})\|_{L_{t}^{2}(L^{2})}+\|t\Delta u\cdot\nabla u_{j}\|_{L_{t}^{2}(L^{2})}+\sum_{i=1}^{3}\|t\partial_{i}u\cdot\nabla\partial_{i}u_{j}\|_{L_{t}^{2}(L^{2})} \nonumber\\
\leq\,&\|{t^{1/2}}\nabla^{2}u\|_{L_{t}^{2}(L^{3})}\|{t^{1/2}}\nabla u_{j}\|_{L_{t}^{\infty}(L^{6})}+\|{t^{1/2}}\nabla u\|_{L_{t}^{\infty}(L^{3})}\|{t^{1/2}}\nabla^{2}u_{j}\|_{L_{t}^{2}(L^{6})} \nonumber\\
\leq\,& C\Big(\|{t^{1/2}}\nabla^{2}u\|_{L_{t}^{2}(L^{3})}\|{t^{1/2}}\nabla^{2}u_{j}\|_{L_{t}^{\infty}(L^{2})}+\|{t^{1/2}}\nabla u\|_{L_{t}^{\infty}(\dot{B}^{1/2}_{2,1})}\|{t^{1/2}}\nabla^{2}u_{j}\|_{L_{t}^{2}(L^{6})}\Big)\nonumber\\
\leq\,&Cd_j2^{j/2}\|(u_0,B_0)\|_{\dot{B}^{1/2}_{2,1}},
\end{align}
where we have used the Stokes estimate
\begin{align*}
\|{t^{1/2}}\nabla^{2}u\|_{L_{t}^{2}(L^{3})}&\leq C\|t^{1/2}\partial_t u\|_{L_{t}^{2}(L^3)}+\|u\|_{L^2_t(L^{\infty})}\|t^{1/2}u\|_{L^{\infty}_t(L^3)}\\
&\leq C\|t^{1/2}\partial_t u\|_{L_{t}^{2}(\dot{B}^{1/2}_{2,1})}+C \|u\|_{L_{t}^{2}(\dot{B}^{3/2}_{2,1})}\|t^{1/2}u\|_{L_{t}^{\infty}(\dot{B}^{3/2}_{2,1})}\\
&\leq C \|(u_0,B_0)\|_{\dot{B}^{1/2}_{2,1}},
\end{align*}
derived from  \eqref{A2.1} and \eqref{A2.20}.

Furthermore, we deduce from \eqref{A2.20} and \eqref{A2.50} that
\begin{align}\label{A2.53}
\|t\nabla u\cdot\nabla P_{j}\|_{L_{t}^{2}(L^{2})}\leq\,&\|{t^{1/2}}\nabla u\|_{L_{t}^{\infty}(L^{3})}\|{t^{1/2}}\nabla P_{j}\|_{L_{t}^{2}(L^{6})}\nonumber\\
\leq\,&Cd_{j}2^{j/2}\|{t^{1/2}}\nabla u\|_{L_{t}^{\infty}(\dot{B}^{1/2}_{2,1})}\|(u_{0},B_0)\|_{\dot{B}^{1/2}_{2,1}}\nonumber\\
\leq\,&Cd_{j}2^{j/2}\|(u_{0},B_0)\|_{\dot{B}^{1/2}_{2,1}}.
\end{align}
Similar to \eqref{A2.52} and \eqref{A2.53}, we can also obtain
\begin{align}\label{A2.54}
&\|tD_t(B\cdot\nabla B_j)\|_{{L_t^2}(L^2)}\nonumber\\
\leq\,& C\|{t^{1/2}}\partial_t B\|_{{L_t^2}(L^3)}\|{t^{1/2}}\nabla B_j\|_{{L_t^{\infty}}(L^6)} +C\|{t^{1/2}}u\|_{{L_t^{\infty}}(L^{\infty})}\|\nabla B\|_{L^2_t(L^3)}\|{t^{1/2}}\nabla B_j\|_{{L_t^{\infty}}(L^6)}\nonumber\\
 &+C\|{t^{1/2}}u\|_{{L_t^{\infty}}(L^{\infty})}\| B\|_{L^2_t(L^{\infty})}\|{t^{1/2}}\nabla^2 B_j\|_{{L_t^{\infty}}(L^2)}+C\|{t^{1/2}}B\|_{{L_t^{\infty}}(L^{\infty})}\|{t^{1/2}}\nabla \partial_t B_j\|_{{L_t^{2}}(L^2)}\nonumber\\
\leq\,& Cd_{j}2^{j/2}\|(u_{0},B_0)\|_{\dot{B}^{1/2}_{2,1}}.
\end{align}
Combining the definition of $f_j$ in \eqref{A2.45} and \eqref{A2.52}--\eqref{A2.54},
we get
\begin{align}\label{A2.55}
 \|tf_j\|_{L_t^2(L^2)}\leq Cd_j2^{\frac{i}{2}}\|(u_0,B_0)\|_{\dot{B}^{1/2}_{2,1}}, 
\end{align}
which leads to
\begin{align}\label{A2.56}
\bigg|\int _{0}^{t}I_9 {{\rm d}t}^{\prime }\bigg| \leq   
Cd_{j}^{2}2^{j}\|(u_{0},B_0)\|_{\dot{B}^{1/2}_{2,1}}^{2}+\frac{1}{6}\|t\sqrt{\rho}D_{t}^{2}u_{j}\|_{L_{t}^{2}(L^{2})}^{2}.
\end{align}

Next, we turn to estimate $\int _{0}^{t}I_{10} {{\rm d}t}^{\prime }$.
Similar to \eqref{A2.52} and \eqref{A2.54}, one has
\begin{align}\label{A2.57}
&\|t\Delta u\cdot\nabla B_{j}\|_{L_{t}^{2}(L^{2})}+\sum_{i=1}^{3}\|t\partial_{i}u\cdot\nabla\partial_{i}B_{j}\|_{L_{t}^{2}(L^{2})} \nonumber\\
\leq\,&\|{t^{1/2}}\nabla^{2}u\|_{L_{t}^{2}(L^{3})}\|{t^{1/2}}\nabla B_{j}\|_{L_{t}^{\infty}(L^{6})}+\|{t^{1/2}}\nabla u\|_{L_{t}^{\infty}(L^{3})}\|{t^{1/2}}\nabla^{2}B_{j}\|_{L_{t}^{2}(L^{6})} \nonumber\\
\leq\,& C\Big(\|{t^{1/2}}\nabla^{2}u\|_{L_{t}^{2}(L^{3})}\|{t^{1/2}}\nabla^{2}u_{j}\|_{L_{t}^{\infty}(L^{2})}+\|{t^{1/2}}\nabla u\|_{L_{t}^{\infty}(\dot{B}^{1/2}_{2,1})}\|{t^{1/2}}\nabla^{2}B_{j}\|_{L_{t}^{2}(L^{6})}\Big)\nonumber\\
\leq\,&Cd_j2^{j/2}\|(u_0,B_0)\|_{\dot{B}^{1/2}_{2,1}},
\end{align}
and
\begin{align}\label{A2.58}
&\|tD_t(B\cdot\nabla u_j)\|_{{L_t^2}(L^2)}\nonumber\\
\leq\,& C\|{t^{1/2}}\partial_t B\|_{{L_t^2}(L^3)}\|{t^{1/2}}\nabla u_j\|_{{L_t^{\infty}}(L^6)} +C\|{t^{1/2}}u\|_{{L_t^{\infty}}(L^{\infty})}\|\nabla B\|_{L^2_t(L^3)}\|{t^{1/2}}\nabla u_j\|_{{L_t^{\infty}}(L^6)}\nonumber\\
\,&+C\|{t^{1/2}}u\|_{{L_t^{\infty}}(L^{\infty})}\| B\|_{L^2_t(L^{\infty})}\|{t^{1/2}}\nabla^2 u_j\|_{{L_t^{\infty}}(L^2)}+C\|{t^{1/2}}B\|_{{L_t^{\infty}}(L^{\infty})}\|{t^{1/2}}\nabla \partial_t u_j\|_{{L_t^{\infty}}(L^2)}\nonumber\\
\leq\,&Cd_{j}2^{j/2}\|(u_{0},B_0)\|_{\dot{B}^{1/2}_{2,1}}.
\end{align}
With \eqref{A2.57} and \eqref{A2.58} in hand, we deduce from the definition
of $h_j$ in \eqref{A2.46} that
\begin{align*}
 \|t h_j\|_{L_t^2(L^2)}\leq Cd_j2^{{i}/{2}}\|(u_0,B_0)\|_{\dot{B}^{1/2}_{2,1}}, 
\end{align*}
which gives
\begin{align}\label{A2.60}
\bigg|\int _{0}^{t}I_{10} {{\rm d}t}^{\prime }\bigg| \leq   
Cd_{j}^{2}2^{j}\|(u_{0},B_0)\|_{\dot{B}^{1/2}_{2,1}}^{2}+\frac{1}{4}\|tD_{t}^{2}B_{j}\|_{L_{t}^{2}(L^{2})}^{2}.    
\end{align}
According to \eqref{A2.51}, \eqref{A2.52} and \eqref{A2.55}, we have
\begin{align}\label{A2.61}
&\|t\nabla^2 D_{t} u_{j}\|_{L_t^{2}(L^2)}+\|t\nabla D_{t} P_{j}\|_{L_t^{2}(L^2)}\nonumber\\
\leq\,&C\big(\|t\nabla\mathrm{Tr}(\nabla u\cdot\nabla u_{j})\big\|_{_{L_t^{2}(L^2)}}+\|t\rho D_{t}^{2}u_{j}\|_{_{L_t^{2}(L^2)}}+\|tf_{j}\|_{_{L_t^{2}(L^2)}}\big)\nonumber\\
\leq\,&Cd_j2^{{i}/{2}}\|(u_0,B_0)\|_{\dot{B}^{1/2}_{2,1}}+C\|t\sqrt{\rho} D_{t}^{2}u_{j}\|_{_{L_t^{2}(L^2)}}. 
\end{align}

Notice that
\begin{align*}
|I_6|\leq C\|\nabla u\|_{L^3}\|t\nabla D_tu_j\|_{L^3}^2\leq C\|\nabla u\|_{L^3}\|t\nabla D_tu_j\|_{L^2}\|t\nabla^2 D_tu_j\|_{L^2}.
\end{align*}
As a consequence, for $\int _{0}^{t}I_{6} {{\rm d}t}^{\prime }$, it holds
\begin{align}\label{A2.63}
\bigg|\int _{0}^{t}I_{6} {{\rm d}t}^{\prime }\bigg| 
\leq\,&   
Cd_{j}^{2}2^{j}\|(u_{0},B_0)\|_{\dot{B}^{1/2}_{2,1}}^2+C\int_{0}^{t}\| u\|_{\dot{B}^{3/2}_{2,1}}^2\|t\nabla D_{t}u_{j}\|_{L^{2}}^{2}{{\rm d}t}^{\prime}+\frac{1}{6}\|t\sqrt{\rho}D_{t}^{2}u_{j}\|_{L_{t}^{2}(L^{2})}^{2}.    
\end{align}

A similar computation yields
\begin{align}\label{A2.64}
\bigg|\int _{0}^{t}I_{7} {{\rm d}t}^{\prime }\bigg| 
\leq\,&   
Cd_{j}^{2}2^{j}\|(u_{0},B_0)\|_{\dot{B}^{1/2}_{2,1}}^2+C\int_{0}^{t}\| u\|_{\dot{B}^{3/2}_{2,1}}^2\|t\nabla D_{t}B_{j}\|_{L^{2}}^{2}{{\rm d}t}^{\prime}+\frac{1}{4}\|tD_{t}^{2}B_{j}\|_{L_{t}^{2}(L^{2})}^{2}.    
\end{align}

Finally, we deal with the remaining term  $\int _{0}^{t}I_{8} {{\rm d}t}^{\prime }$.
Using integration by parts, we obtain
\begin{align}\label{A2.65}
&\int_{\mathbb{R}^{3}}\nabla D_{t} P_{j}\cdot D_{t}^{2}u_{j}\, {\rm d}x\nonumber\\
=\,&\int_{\mathbb{R}^{3}}\nabla D_{t} P_{j}\cdot u\cdot\nabla D_{t}u_{j}\, {\rm d}x-\int_{\mathbb{R}^3}D_{t} P_{j}\cdot \partial_{t} \mathrm{div} D_{t}u_{j}\, {\rm d}x  \nonumber\\
=\,&\int_{\mathbb{R}^3}\nabla D_{t} P_{j}\cdot\partial_{t}u\cdot\nabla u_{j}\, {\rm d}x+\int_{\mathbb{R}^3}\nabla D_{t} P_{j}\cdot\partial_{t}u_{j}\cdot\nabla u\, {\rm d}x +\int_{\mathbb{R}^{3}}\nabla D_{t} P_{j}\cdot u\cdot\nabla D_{t}u_{j}\, {\rm d}x.
\end{align}
It follows from \eqref{A2.20}, \eqref{A2.33}, \eqref{A2.61} and
\eqref{A2.65} that
\begin{align}\label{A2.66}
\bigg|\int _{0}^{t}I_{8} {{\rm d}t}^{\prime }\bigg|
\leq\,&\|t\nabla D_{t} P_{j}\|_{L_{t}^{2}(L^{2})}\|{t^{1/2}}\partial_t u\|_{L_{t}^{2}(\dot{B}^{1/2}_{2,1})}\|{t^{1/2}}\nabla^{2}u_{j}\|_{L_{t}^{\infty}(L^{2})}\nonumber\\
&+\|t\nabla D_{t} P_{j}\|_{L_{t}^{2}(L^{2})}\|{t^{1/2}}\nabla\partial_{t}u_{j}\|_{L_{t}^{2}(L^{2})}\|{t^{1/2}}\nabla u\|_{L_{t}^{\infty}({\dot{B}^{1/2}_{2,1}})}\nonumber\\
&+\int_{0}^{t}\|u\|_{L^{\infty}}\|t\nabla D_{t}u_{j}\|_{L^{2}}\|t\nabla D_{t} P_{j}\|_{L^{2}} {{\rm d}t}^{\prime}\nonumber\\
\leq\,&C\int_{0}^{t}\|u\|_{L^{\infty}}^{2}\|t\nabla D_{t}u_{j}\|_{L^{2}}^{2} {{\rm d}t}^{\prime}+Cd_{j}^{2}2^{j}\|(u_{0},B_0)\|_{\dot{B}^{1/2}_{2,1}}^{2}+\frac{1}{6}\|t\sqrt{\rho}D_{t}^{2}u_{j}\|_{L_{t}^{2}(L^{2})}^{2}.
\end{align}
Substituting \eqref{A2.56}, \eqref{A2.60}, \eqref{A2.63}, \eqref{A2.64} and \eqref{A2.66} into \eqref{A2.48}, and then integrating the resulting inequality  over $[0,t]$, we obtain
\begin{align}\label{A2.67}
\|t\nabla D_{t}&(u_{j},B_j)\|_{L_{t}^{\infty}(L^{2})}^{2}+\|t(\sqrt{\rho}D_{t}^{2}u_{j},D_t^2B_j)\|_{L_{t}^{2}(L^{2})}^{2}\nonumber\\
\leq\, & Cd_{j}^{2}2^{j}\|(u_{0},B_0)\|_{\dot{B}^{1/2}_{2,1}}^{2}+\|{t^{1/2}} (\nabla D_{t} u_{j}, \nabla D_{t} B_j)\|_{L_{t}^{2}(L^{2})}^{2} \nonumber\\
&+C\int_{0}^{t}\|(u,B) \|_{\dot{B}^{3/2}_{2,1}}^{2}\|t\nabla (D_{t}u_{j},D_t B_j) \|_{L^{2}}^{2} \mathrm{d}t^{\prime}.
\end{align}
By the estimates \eqref{A2.1}, \eqref{4.22} and \eqref{A2.33}, the second term on the right-hand side of \eqref{A2.67} is bounded by
\begin{align*}
&\,\|{t^{1/2}}(\nabla D_{t}u_{j},\nabla D_t B_j)\|_{L_{t}^{2}(L^{2})}\nonumber\\
\leq&\,\|{t^{1/2}}(\nabla\partial_{t}u_{j},\nabla\partial_{t}B_j)\|_{L_{t}^{2}(L^{2})}+\|{t^{1/2}}u\cdot(\nabla u_{j},\nabla B_j)\|_{L_{t}^{2}(\dot{H}^1)} \nonumber\\
\leq&\,C\|{t^{1/2}}(\nabla\partial_{t}u_{j},\nabla\partial_{t}B_j)\|_{L_{t}^{2}(L^{2})}+C\|u\|_{L_{t}^{2}(\dot{B}^{3/2}_{2,1})}\|{t^{1/2}}(\nabla u_{j},\nabla B_j)\|_{L_{t}^{\infty}(\dot{H}^1)} \nonumber\\
\leq&\,Cd_{j}2^{{i}/{2}}\|(u_{0},B_0)\|_{\dot{B}^{1/2}_{2,1}},
\end{align*}
which, together with \eqref{A2.1}, \eqref{A2.67} and Gronwall's inequality, leads to
\begin{align}\label{A2.69}
&  \|t (\nabla D_{t} u_{j}, \nabla D_{t} B_j)\|_{L_{t}^{\infty}(L^{2})}^{2}+\|t(D_{t}^{2}u_{j},D_{t}^{2}B_j)\|_{L_{t}^{2}(L^{2})}^{2}\nonumber\\
\leq \,& Cd_{j}^{2}2^{j}\|(u_{0},B_0)\|_{\dot{B}^{1/2}_{2,1}}^{2}\exp\bigg\{C\|(u,B)\|_{L_{t}^{2}(\dot{B}^{3/2}_{2,1})}^{2}\bigg\} \nonumber\\
\leq\, & Cd_{j}^{2}2^{j}\|(u_{0},B_0)\|_{\dot{B}^{1/2}_{2,1}}^{2}.
\end{align}
From \eqref{A2.61} and \eqref{A2.69}, one also obtains
\begin{align}\label{A2.70}
\|t\nabla^{2}(D_{t} u_{j}, D_{t} B_j)\|_{L_{t}^{2}(L^{2})}\leq Cd_{j}2^{j/2}\|(u_{0},B_0)\|_{\dot{B}^{1/2}_{2,1}}. 
\end{align}
With the help of \eqref{A2.35}, \eqref{A2.69} and \eqref{A2.70}, we arrive at the desired estimate \eqref{A2.35} by arguing similarly to that in \eqref{A2.14}.
\end{proof}

Finally, we give some qualitative regularity estimates based on the uniform bounds obtained in Propositions \ref{P3.1}--\ref{prop4.33}.
\begin{prop}\label{P3.6}
Under the assumptions of Proposition \ref{P3.1}, we have for any $T\in(0,T^*)$,
\begin{align}
\|\nabla u\|_{L^4_T(L^2)}&\leq C \|(u_{0},B_0)\|_{\dot{B}^{1/2}_{2,1}},\label{uL4L2:thm2}\\
\|{t^{1/2}}(\nabla u,\nabla B)\|_{\widetilde{L}_{T}^{\infty}(\dot{B}^{{1}/{2}}_{2,1})}&\leq C\|(u_0,B_0)\|_{\dot{B}^{1/2}_{2,1}},\label{4.52}\\
\|t^{1/4}(\nabla u,\nabla B)\|_{{L_T^2}(L^{\infty})}&\leq C\|(u_{0},B_0)\|_{\dot{B}^{{1/2}}_{2,1}}, \label{B3.7211} \\ \label{B3.72} 
\|(\nabla u,\nabla B)\|_{L_T^{1}(L^{\infty})}&\leq  C\|(u_{0},B_0)\|_{\dot{B}^{1/2}_{2,1}},\\\label{NJK4.74}
\|t^{1/4} (\nabla^2 u, \nabla^2 B, \nabla P, \partial_t u,\partial_t B)\|_{L^{2}_T(L^2)} &\leq  C\|(u_{0},B_0)\|_{\dot{B}^{1/2}_{2,1}},\\\label{NJK4.75}
\|t^{3/4}(\partial_{t}u,\partial_t B,\nabla^{2}u,\nabla^2 B)\|_{L_{T}^{\infty}(L^{2})}&\leq  C\|(u_{0},B_0)\|_{\dot{B}^{1/2}_{2,1}},\\\label{NJK4.76}
\|t^{1/4}(\nabla u,\nabla B)\|_{L^{\infty}_T(L^2)}+\|t^{3/4}(\nabla^{2}u,\nabla^2B , \nabla P)\|_{L_{T}^{2}(L^{6})}&\leq  C\|(u_{0},B_0)\|_{\dot{B}^{1/2}_{2,1}},\\ \label{NJK4.77} 
\|t^{3/4}(D_{t}u,D_{t}B)\|_{L_{T}^{\infty}(L^{2})}&\leq  C\|(u_{0},B_0)\|_{\dot{B}^{1/2}_{2,1}},\\
{\| t^{3/4}(\nabla D_t u,\nabla  D_t B)\|_{L_T^2(L^2)}}&\leq C\|(u_0,B_0)\|_{\dot{B}^{1/2}_{2,1}}.\label{34}
\end{align}
\end{prop}

\begin{proof}
First, thanks to the interpolation property in Lemma \ref{L2.7}, one has
\begin{align*}
 {\|(\nabla u,\nabla B)\|_{L^4_t(L^2)}}\leq& \, C \|(u,B)\|_{{\widetilde L}_t^{\infty} (\dot{B}_{2,1}^{1/2})}\|(u,B)\|_{\widetilde{L}_t^{\infty} (\dot{B}_{2,1}^{3/2})},
\end{align*}
and
\begin{align*}
 {\|t^{1/4}(\nabla u, \nabla B)\|_{\widetilde{L}^{\infty}_t(\dot{B}^{0}_{2,1})}}\leq& \, C \|(u,B)\|_{{\widetilde L}_t^{\infty} (\dot{B}_{2,1}^{1/2})}\|t^{1/2}(\nabla u,\nabla B)\|_{\widetilde{L}_t^{\infty} (\dot{B}_{2,1}^{3/2})}.
\end{align*}
This, as well as \eqref{A2.1} and \eqref{A2.20}, yields \eqref{uL4L2:thm2} and \eqref{4.52}. 

Next, it follows from \eqref{4.22} that
\begin{align}
\|t^{1/2}(\nabla u_j,\nabla B_j)\|_{L^2_T(L^6)}\leq C d_j 2^{-j/2} \|(u_0,B_0)\|_{\dot{B}^{1/2}_{2,1}}.\label{mmmmooof}
\end{align}
Combining \eqref{A2.50} and \eqref{mmmmooof}, one gains 
\begin{align*}
\|t^{1/2}(\nabla u,\nabla B)\|_{L^2_t(\dot{B}^{1/2}_{6,1})}\leq C \|(u_0,B_0)\|_{\dot{B}^{1/2}_{2,1}}.
\end{align*}
This ensures \eqref{B3.7211} due to the embedding $\dot{B}^{1/2}_{6,1}\hookrightarrow L^{\infty}$.

In order to prove the $L^1(0,T;L^{\infty})$ bound of $(\nabla u,\nabla B)$, we recall that $u_j$ and $B_j$ satisfy the elliptic system \eqref{A2.42}. Applying \eqref{A2.42} and the standard regularity theory of Stokes operator  to \eqref{A2.42} implies that
\begin{align}\label{tttttt}
&\|(t\nabla^{2}u_{j},t\nabla^2 B_j,t\nabla P_{j})\|_{L_t^{\infty}(L^{6})}\nonumber\\
\leq \,& C\|t\sqrt{\rho}\partial_{t}u_{j}\|_{L_t^{\infty}(L^{6})}+C\|tB\cdot\nabla B_j\|_{L_t^{\infty}(L^{6})}+\,C\|t D_{t}B_{j}\|_{L_t^{\infty}(L^{6})}+C\|tB\cdot\nabla u_j\|_{L_t^{\infty}(L^{6})}\nonumber\\
\leq\,& C\|t(\nabla\partial_t u_j,\nabla\partial_t B_j)\|_{L_t^{\infty}(L^{2})}+C\|(u,B)\|_{L_t^{2}(\dot{B}^{3/2}_{2,1})}\|t \nabla^2(u_j,B_j)\|_{L_t^{\infty}(L^2)},
\end{align}
and
\begin{align}\label{tttttt1}
&\|(t\nabla^{2}u_{j},t\nabla^2 B_j,t\nabla P_{j})\|_{L_t^{2}(L^{6})}\nonumber\\
\leq \,& C\|t\sqrt{\rho}\partial_{t}u_{j}\|_{L_t^{2}(L^{6})}+C\|tB\cdot\nabla B_j\|_{L_t^{2}(L^{6})}+\,C\|t\partial_{t}B_{j}\|_{L_t^{2}(L^{6})}+C\|tB\cdot\nabla u_j\|_{L_t^{2}(L^{6})}\nonumber\\
\leq\,& C\|t(\nabla\partial_t u_j,\nabla\partial_t B_j)\|_{L_t^{2}(L^{2})}+C\|(u,B)\|_{L_t^{2}(\dot{B}^{1/2}_{2,1})}\|t\nabla^2(u_j,B_j)\|_{L_t^{2}(L^2)}.
\end{align}
To control the terms involving in \eqref{tttttt} and \eqref{tttttt1}, we recall that, the estimates in \eqref{A2.10}, \eqref{A2.33}, \eqref{A2.35} and \eqref{A2.69} read
\begin{align}
&\|{t^{1/2}}(\nabla^{2}u_{j}, \nabla^{2} B_j)\|_{L_t^{2}(L^2)}+\|t\nabla(\partial_{t}u_{j},\partial_{t}B_{j})\|_{L_t^{2}(L^2)}\leq\, C d_{j}2^{-j/2}\|(u_{0},B_0)\|_{\dot{B}_{2,1}^{1/2}}, \label{4.80}
\end{align}
and
\begin{align}
&\|(\nabla^{2}u_{j}, \nabla^{2} B_j)\|_{L_t^{2}(L^2)}+\|{t^{1/2}}(\nabla^{2}u_{j}, \nabla^{2} B_j)\|_{L_t^{\infty}(L^{2})}+\|t(\nabla\partial_{t}u_{j},\nabla\partial_t B_j)\|_{L_t^{\infty}(L^{2})} \nonumber\\
&\quad \leq\, C d_{j}2^{j/2}\|(u_{0},B_0)\|_{\dot{B}_{2,1}^{1/2}}.\label{4.81}
\end{align}
In addition, one also has
\begin{align}
&\quad\, \|{t}(\nabla^2 u_j,\nabla^2 B_j)\|_{L_{t}^{\infty}(L^{2})} \nonumber\\
& \leq C\|{t}(D_{t}u_j,D_t B_j)\|_{L_{t}^{\infty}(L^{2})}+C\|(u,B)\|_{L_{t}^{\infty}(L^{3})}\|{t}(\nabla u,\nabla B)\|_{L_{t}^{\infty}(L^{6})}  \nonumber\\
&\leq C\|{t}(D_{t}u_j,D_t B_j)\|_{L_{t}^{2}(\dot{B}^{1/2}_{2,1})}+C\|u\|_{L_{t}^{\infty}(\dot{B}^{1/2}_{2,1})}\|{t^{1/2}}(\nabla^2 u_j,\nabla^2 B_j)\|_{L_{t}^{\infty}(L^2)}. \nonumber
\end{align}
As $\|u\|_{L_{t}^{\infty}(\dot{B}^{1/2}_{2,1})}$ is sufficiently small, one obtains
\begin{align}
\|{t}(\nabla^2 u_j,\nabla^2 B_j)\|_{L_{t}^{\infty}(L^{2})} \leq\, C d_{j}2^{-j/2}\|(u_{0},B_0)\|_{\dot{B}_{2,1}^{1/2}}.\label{4.83}
\end{align}
The combination of \eqref{tttttt}--\eqref{4.83} gives rise to
\begin{align}
\|(t\nabla^{2}u_{j},t\nabla^2 B_j,t\nabla P_{j})\|_{L_t^{\infty}(L^{6})}&\leq \, d_j 2^{j/2} \|(u_0,B_0)\|_{\dot{B}^{1/2}_{2,1}},\label{ujLinftyL6}\\
\|(t\nabla^{2}u_{j},t\nabla^2 B_j,t\nabla P_{j})\|_{L_t^{2}(L^{6})}&\leq C d_j 2^{-j/2}\|(u_0,B_0)\|_{\dot{B}^{1/2}_{2,1}}.\label{ujL2L6}
\end{align}
The above inequalities allow us to prove the $L^1(0,T;L^{\infty})$ bound of $(\nabla u,\nabla B)$ by using an interpolation argument. Indeed, since we have $L^{4,1}(0,T;L^2)=\big(L^2(0,T;L^2),L^\infty(0,T;L^2)\big)_{1/2,1}$ derived from Lemma \ref{D A.9}, by \eqref{4.80} and \eqref{4.81}, the following estimate holds:
\begin{align}\label{D4.90}
\|{t^{1/2}}(\nabla^{2}u_{j}, \nabla^{2} B_j,\nabla P_j)\|_{L_t^{4,1}(L^{2})}\leq\,& C\|{t^{1/2}}(\nabla^{2}u_{j}, \nabla^{2} B_j)\|_{L_t^{2}(L^{2})}^{1/2}\|{t^{1/2}}(\nabla^{2}u_{j}, \nabla^{2} B_j)\|_{L_t^{\infty}(L^{2})}^{1/2}\nonumber\\
\leq\,& C d_{j}\|(u_{0},B_0)\|_{\dot{B}_{2,1}^{1/2}}.
\end{align}
Note that Lemma \ref{D A.9} also guarantees $
L^{4,1}(0,T;L^6)=\big(L^2(0,T;^6),L^\infty(0,T;L^6)\big)_{1/2,1}$. Thus, \eqref{ujLinftyL6} and \eqref{ujL2L6} ensure that
\begin{align}\label{D4.94}
\|t (\nabla^{2}u_{j}, \nabla^{2} B_j,\nabla P_j)\|_{L_t^{4,1}(L^{6})}\leq\,& C\|t (\nabla^{2}u_{j}, \nabla^{2} B_j)\|_{L_t^{2}(L^{6})}^{1/2}\|t (\nabla^{2}u_{j}, \nabla^{2} B_j)\|_{L_t^{\infty}(L^{6})}^{1/2}\nonumber\\
\leq\,& C d_{j}\|(u_{0},B_0)\|_{\dot{B}_{2,1}^{1/2}}.
\end{align}
By summing \eqref{D4.90} and \eqref{D4.94} over $j\in\mathbb{Z}$, 
we conclude that
\begin{align}\label{D4.95}
\|t(\nabla^{2}u,\nabla^2 B,\nabla P)\|_{L_t^{4,1}(L^{6})}+\|{t^{1/2}}(\nabla^{2}u,\nabla^2 B,\nabla P_j)\|_{L_t^{4,1}(L^{2})}\leq C\|(u_{0},B_0)\|_{\dot{B}_{2,1}^{1/2}}.   
\end{align}
Consequently, in light of \eqref{D4.95}, Lemma \ref{D A.9}, the Gagliardo-Nirenberg inequality $$\|f\|_{L^{\infty}}\leq C\|f\|_{L^2}^{1/2}\|\nabla f\|_{L^6}^{1/2},$$ 
and
\eqref{D4.95}, we obtain
\begin{align*}
 \int_0^t\|(\nabla u,\nabla B)\|_{L^\infty} {\rm d}t^{\prime} 
\leq\,&C\int_0^t\|\nabla^2(u,B)\|_{L^2}^{1/2}\|\nabla^2(u,B)\|_{L^6}^{1/2}{\rm d}t^{\prime} \nonumber\\
\leq\, &C\int_{0}^{t}t^{-3/4}\|{t^{1/2}}(\nabla^{2}u,\nabla^{2}B)\|_{L^{2}}^{1/2}\|t(\nabla^{2}u,\nabla^{2}B)\|_{L^{6}}^{1/2}{\rm d}t^{\prime} \nonumber\\
\leq\,& C\|t^{-3/4}\|_{L^{\frac{4}{3},\infty}}\|{t^{1/2}}(\nabla^{2}u,\nabla^{2}B)\|_{L^{4,1}(L^{2})}^{1/2}\|t(\nabla^{2}u,\nabla^{2}B)\|_{L^{4,1}(L^{6})}^{1/2} \nonumber\\
\leq\,& C\|(u_{0},B_0)\|_{\dot{B}_{2,1}^{1/2}}.
\end{align*}
This gives \eqref{B3.72} directly.

Moreover, according to \eqref{A2.10}, \eqref{A2.13} and \eqref{4.22}, we conclude that
\begin{align*} 
&\|t^{1/4} (\nabla^2 u, \nabla^2 B, \nabla P, \partial_t u,\partial_t B)\|_{L_t^{2}(L^2)}^2\nonumber\\ 
\leq&\,2\sum_{j\in{\mathbb{Z}}}\sum_{k\leq j}\int_{0}^t\int_{\mathbb{R}^3}{t^{\prime}}^{{1}/{2}}(\nabla^2 u_j, \nabla^2 B_j, \nabla P_j, \partial_t u_j,\partial_t B_j)(\nabla^2 u_k, \nabla^2 B_k, \nabla P_k, \partial_t u_k,\partial_t B_k){\rm d}x{\rm d}t   \nonumber\\
\leq&\,C\sum_{j\in{\mathbb{Z}}}\sum_{k\leq j} \|t^{\frac{1}{2}}(\nabla^2 u_j, \nabla^2 B_j, \nabla P_j, \partial_t u_j,\partial_t B_j)\|_{L_t^2(L^2)} \|(\nabla^2 u_k, \nabla^2 B_k, \nabla P_k, \partial_t u_k,\partial_t B_k)\|_{L_t^2(L^2)}       \nonumber\\
\leq &\, C\|(u_{0},B_0)\|_{\dot{B}^{1/2}_{2,1}}^2,
\end{align*}
which leads to \eqref{NJK4.74}.

Similar to \eqref{A2.28},
one easily deduces from \eqref{A2.1} and Lemma \ref{L2.3} that
\begin{align}\label{NJKS4.96}
&\frac{\rm d}{{\rm d}t}\|(\sqrt{\rho}\partial_t u,\partial_t B)\|_{L^2}^2+\|(\nabla\partial_{t}u,\nabla\partial_t B)\|_{L^{2}}^{2}\nonumber\\
\leq\,& C\Big(1+\|u\|_{\dot{B}^{1/2}_{2,1}}^{2}\Big)\|u\|_{\dot{B}^{3/2}_{2,1}}^{2}\|(\sqrt{\rho}\partial_{t}u,\partial_t B)\|_{L^{2}}^{2}\nonumber\\
&\,+C\|\nabla(u,B)\|_{\dot{H}^{{1}/{2}}}^2\|(\sqrt{\rho}\partial_{t}u,\partial_t B)\|_{L^{2}}^2\nonumber\\
\leq\,&C\|(u,B)\|_{\dot{B}^{{3}/{2}}_{2,1}}^2\|(\sqrt{\rho}\partial_{t}u,\partial_t B)\|_{L^{2}}^{2}.
\end{align}
Here  we have used the fact that
\begin{align}\label{NJK 4.96}
\|(\nabla^2u,\nabla^2B,\nabla P)(t)\|_{L^2}\leq C\|(\partial_t u,\partial_t B)(t)\|_{L^2},    
\end{align}
by an argument similar to \eqref{A2.10}.
By multiplying by \eqref{NJKS4.96} $t^{3/2}$, we obtain
\begin{align*}
&\frac{\rm d}{{\rm d}t}\|t^{3/4}(\sqrt{\rho}\partial_t u,\partial_t B)\|_{L^2}^2+\|t^{{3}/{4}}(\nabla\partial_{t}u,\nabla\partial_t B)\|_{L^{2}}^{2}\nonumber\\
\,&\quad \leq\frac{3}{2}\|t^{1/4}(\partial_t u,\partial_t B)\|_{L^2}^2+C\|(u,B)\|_{\dot{B}^{3/2}_{2,1}}^2\|t^{3/4}(\sqrt{\rho}\partial_{t}u,\partial_t B)\|_{L^{2}}^{2},
\end{align*}
which together with Gronwall's inequality, \eqref{A2.1} and \eqref{NJK4.74}  yields
\begin{align}\label{NJKS4.98}
&\|t^{3/4}(\partial_t u,\partial_t B)\|_{L_t^{\infty}(L^2)}^2+\|t^{3/4}(\nabla \partial_t u,\nabla\partial_t B)\|_{L_t^2(L^2)} \nonumber\\
\leq\,& C\|t^{1/4}(u_t,B_t)\|_{L_t^2(L^2)}\exp\Big\{C\|(u,B)\|_{L^2_t(\dot{B}_{2,1}^{3/2})}^2\Big\}\nonumber\\
\leq&\, C\|(u_0,B_0)\|_{\dot{B}_{2,1}^{{1}/{2}}}^2.
\end{align}
From \eqref{NJK 4.96}, we directly obtain
\begin{align*}
&\|t^{3/4}(\nabla^2 u,\nabla^2 B)\|_{L_t^{\infty}(L^2)}^2
\leq C\|(u_0,B_0)\|_{\dot{B}_{2,1}^{{1}/{2}}}^2,    
\end{align*}
which gives rise to \eqref{NJK4.75}.

To prove \eqref{NJK4.76}, by using Lemma \ref{L2.7}, \eqref{A2.1} and \eqref{A2.20}, one has
\begin{align}\label{NJKS4.100}
\|t^{1/4}(\nabla u,\nabla B)\|_{L_t^{\infty}(L^2)}^2\leq C\|u,B\|_{L_t^{\infty}(\dot{B}^{{1}/{2}}_{2,1})}\|t^{{1}/{2}}(u,B)\|_{L^{\infty}_{t}(\dot{B}^{3/2}_{2,1})}\leq C\|(u_0,B_0)\|_{\dot{B}_{2,1}^{{1}/{2}}}^2.   
\end{align}
Using the classical theory on Stokes and Laplacian operators again, we deduce that
\begin{align}\label{NJKS4.101}
&\|t^{{3}/{4}}(\nabla^2u,\nabla^2B,\nabla P)\|_{L_t^{2}(L^6)}\nonumber\\
\leq&\, C\|t^{{3}/{4}}(\partial_t u,\partial_t B)\|_{L_t^2(L^6)}+C\|t^{{1}/{4}}(u,B)\|_{L^{\infty}_t(L^6)}\|(t^{{1}/{2}}\nabla u,t^{{1}/{2}}\nabla B)\|_{L_t^2(L^{\infty})}  \nonumber\\
\leq&\,C\|t^{{1}/{4}}(\nabla u,\nabla B)\|_{L^{\infty}_t(L^2)}\|t^{{1}/{4}}(\nabla u,\nabla B)\|_{L_t^2(L^{6})}^{{1}/{2}} 
\|t^{{3}/{4}}(\nabla^2 u,\nabla^2 B)\|_{L_t^2(L^{6})}^{{1}/{2}}\nonumber\\
&+\, C\|t^{{3}/{4}}(\nabla\partial_t u,\nabla\partial_t B)\|_{L_t^2(L^2)}.
\end{align}
Combining Young's inequality,  \eqref{NJK4.74}, \eqref{NJKS4.98},
\eqref{NJKS4.100}  with \eqref{NJKS4.101},
we end up with
\begin{align*}
&\|t^{3/4}(\nabla^2u,\nabla^2B,\nabla P)\|_{L_t^{2}(L^6)}\nonumber\\
\leq&\,C\|(u_0,B_0)\|_{\dot{B}_{2,1}^{{1}/{2}}}^2\|t^{{1}/{4}}(\nabla^2 u,\nabla^2 B)\|_{L_t^2(L^{2})}+ C\|t^{{3}/{4}}(\nabla\partial_t u,\nabla\partial_t B)\|_{L_t^2(L^2)}\nonumber\\
\leq&\, C\|(u_0,B_0)\|_{\dot{B}_{2,1}^{{1}/{2}}}.
\end{align*}
By Lemma \ref{L2.6}, Lemma \ref{L2.5}, \eqref{A2.1} and \eqref{NJK4.75}, it also holds
\begin{align*}
\|t^{{3}/{4}}(D_tu,D_t B)\|_{L_t^{\infty}(L^2)}\leq&\, C\|t^{{3}/{4}}(\partial_t u,\partial_t B)\|_{L_t^{\infty}(L^2)}+C\|t^{{3}/{4}}u\cdot\nabla(u,B)\|_{L_t^{\infty}(L^2)} \nonumber\\
\leq&\, C\|t^{{3}/{4}}(\partial_t u,\partial_t B)\|_{L_t^{\infty}(L^2)}+C\|u\|_{L_t^{\infty}(L^{3,\infty})}\|t^{{3}/{4}}(\nabla u,\nabla B)\|_{L_t^{\infty}(L^{6,2})}\nonumber\\
\leq&\, C\|t^{{3}/{4}}(\partial_t u,\partial_t B)\|_{L_t^{\infty}(L^2)}+C\|u\|_{L_t^{\infty}(\dot{B}_{2,1}^{{1}/{2}})}\|t^{{3}/{4}}(\nabla^2 u,\nabla^2 B)\|_{L_t^{\infty}(L^2)}\nonumber\\
\leq&\, C\|(u_0,B_0)\|_{\dot{B}_{2,1}^{{1}/{2}}}.
\end{align*}
Thus, \eqref{NJK4.77} holds. Furthermore, using  Lemma \ref{L2.7}, we infer that
\begin{align*}
 {\| t^{3/4}(\nabla D_t u,\nabla  D_t B)\|_{L_t^2(L^2)}}\leq\,& C  {\| t^{1/2}(D_t u, D_t B)\|^{1/2}_{L_t^2({\dot{B}^{1/2}_{2,1}})}} {\| t(D_t u, D_t B)\|^{1/2}_{L_t^2({\dot{B}^{3/2}_{2,1}})}}. 
\end{align*}
This, as well as \eqref{A2.35}, leads to \eqref{34} and hence we complete the proof of Proposition \ref{P3.6}.
\end{proof}

\vspace{2mm}

Now we are in a position to prove Theorem \ref{T2.1}.

\vspace{1mm}

 \begin{proof}[Proof of Theorem \ref{T2.1}]
  For the regularized initial data $(\rho^\varepsilon_{0},u^\varepsilon_{0},B^\varepsilon_0)$ given in \eqref{appa} with  $0<\varepsilon<1$, the local well-posedness theory (see \cite[Theorem 1.1]{AH-AMBP-2007}) implies that we can construct the approximate regular solutions $(\rho^\varepsilon,u^\varepsilon,B^\varepsilon,\nabla P^\varepsilon)$ on $[0,T_\varepsilon)$ to \eqref{I-1} subject to the initial data $(\rho^\varepsilon_{0},u^\varepsilon_{0},B^\varepsilon_0)$, where $T_{\varepsilon}$ is the maximal time. According to the estimates in Propositions \ref{P3.1}--\ref{P3.6}, one can establish higher-order estimates dependent on $\varepsilon$ (remarking that some scaling invariant regularities, like $L^4(0,T;\dot{H}^1)$, can prevent the blowing up of higher-order norms), and then we can prove $T^\varepsilon=\infty$ and that the solution $(\rho^\varepsilon,u^\varepsilon,B^\varepsilon,\nabla P^\varepsilon)$ is global.

Moreover, in view of the uniform estimates \eqref{A2.1}, \eqref{A2.20}, \eqref{A2.35} \eqref{uL4L2:thm2}--\eqref{34} for $(\rho^{\varepsilon},u^{\varepsilon},B^{\varepsilon},\linebreak \nabla P^{\varepsilon})$ and a standard compactness argument, we can pass the limit as $\varepsilon \rightarrow 0$ and prove that there is a limit $(\rho,u,B,\nabla P)$ which is in fact a global weak solution to the original Cauchy problem \eqref{I-1}--\eqref{d}. Due to Fatou's property, one knows that $(\rho,u,B,\nabla P)$ satisfies the uniform bounds \eqref{A1.4}--\eqref{T2.1:e2}. We now explain the property $u, B\in C(\mathbb{R}_{+};\dot{B}^{1/2}_{2,1})$. For any small $\eta>0$, due to $u\in \widetilde{L}^{\infty}(\mathbb{R}_{+};\dot{B}^{1/2}_{2,1})$, there exists a $j_\eta\in \mathbb{Z}$ such that 
$$
\sum_{|j|\geq j_\eta+1}2^{j/2} \|\dot{\Delta}_j u\|_{L^{\infty}(\mathbb{R}_{+};L^{\infty})}<\eta.
$$
On the other hand, using $t^{1/4}u_t\in L^{2}(\mathbb{R}_{+};L^2)$, one knows that for any $0\leq t_1\leq t_2<\infty$, 
\begin{align*}
\sum_{|j|\leq j_\eta} 2^{j/2} \|\dot{\Delta}_j u(t_1)-\dot{\Delta}_ju(t_2)\|_{L^2}&\leq \sum_{|j|\leq j_\eta} 2^{j/2} \int_{t_1}^{t_2} \|\dot{\Delta}_j  \partial_s u(s)\|_{L^2}\,{\rm ds}\\
&\leq  \sum_{|j|\leq j_\eta} 2^{j/2} \Big(\int_{t_1}^{t_2}s^{-1/2}\,{\rm ds}\Big)^{1/2} \Big(\int_{t_1}^{t_2} \|s^{1/4} \partial_s u(s)\|_{L^2}^2 \,{\rm ds} \Big)^{1/2}\\
&\leq C2^{j_\eta/2} |t_1-t_2|^{1/4} \|s^{1/4} \partial_s u(s)\|_{L^2(\mathbb{R}_{+};L^2)}^2\rightarrow 0,
\end{align*}
as $|t_1-t_2|\rightarrow 0$. Since $\eta$ is arbitrary, we have $u\in C(\mathbb{R}_{+};\dot{B}^{1/2}_{2,1})$. A similar calculation also yields $B\in C(\mathbb{R}_{+};\dot{B}^{1/2}_{2,1})$. Finally, the uniqueness can be proved by Theorem \ref{thmunique} and a similar argument to that in \eqref{3.64}--\eqref{3.68}. Hence, the proof of Theorem \ref{T2.1} is finished. 
\end{proof}

\section{Proof of Theorem \ref{T2.2}}
In this section, we shall prove Theorem \ref{T2.2} on the global existence of Fujita-Kato type solutions under a refined smallness condition, inspired by \cite{AGZ-arXiv-2024}. In view of the rescaling argument \eqref{2.1}--\eqref{2.2}, our proof reduces to the case $\mu=\nu=1$.

\begin{prop}\label{P4.1}
Let $(\rho,u,B,\nabla P)$ be a smooth enough solution of \eqref{I-1} on $[0,T^*)\times\mathbb{R}^3$. Then
under the assumptions of \eqref{A1.6}--\eqref{A1.7}, the estimate \eqref{A1.8} holds. Moreover,
\begin{gather}
\|(u,B)\|_{\widetilde{L}_{T}^{\infty}(\dot{B}_{2,r}^{1/2})}+\|(u,B)\|_{\widetilde{L}_{T}^{2}(\dot{B}_{2,r}^{3/2})}\leq C\|(u_{0},B_0)\|_{\dot{B}_{2,r}^{1/2}},\label{A4.4}\\
\|(u,B)\|_{\widetilde{L}_{T}^{\infty}(\dot{H}^{1/2})}+\|(u,B)\|_{L_{T}^{2}(\dot{H}^{3/2})}+\|t^{-{1}/{4}}(\nabla u,\nabla B)\|_{L_{T}^{2}(L^{2})}\leq C\|(u_{0},B_0)\|_{\dot{H}^{1/2}}, \label{A4.6}
\end{gather}
hold for any $2\leq r\leq \infty$ and $0<T<T^*$.
\end{prop}
\begin{proof}
Using similar computations to that in \eqref{A2.2}--\eqref{A2.6}, we conclude that 
\begin{align}\label{A4.7}
&\|(u_j,B_j)\|_{L_t^\infty(L^2)}+\|(\nabla  u_j, \nabla B_j)\|_{L_t^2(L^2)}
\leq\, C\|\dot{\Delta}_j (u_0,B_0)\|_{L^2}. 
\end{align}
In particular, we have
\begin{align}\label{A4.8}
&\|(u_j,B_j)\|_{L_t^\infty(L^2)}+ \|(\nabla  u_j, \nabla B_j)\|_{L_t^2(L^2)}
\leq\,  C2^{-j/2}\|(u_0,B_0)\|_{\dot{B}^{1/2}_{2,\infty}}. 
\end{align}

On the other hand, it follows from \eqref{A2.7} and Lemma \ref{L2.3} that
\begin{align}\label{A4.9}
&\frac{1}{2}\frac{\rm d}{{\rm d}t}\Big(\|\nabla u_{j}(t)\|_{L^{2}}^{2}+\nabla B_{j}(t)\|_{L^{2}}^{2}\Big)+\|\sqrt{\rho}\partial_{t}u_{j}\|_{L^{2}}^{2}+\|\partial_{t}B_{j}\|_{L^{2}}^{2}\nonumber\\
\leq\,& C\big(\|u\cdot\nabla(u_j,B_j)\|_{L^{2}}+\|B\cdot\nabla(u_j,B_j)\|_{L^{2}}\big)\big(\|\sqrt{\rho}\partial_{t}u_{j}\|_{L^{2}}+\|\partial_{t}B_{j}\|_{L^{2}}\big) \nonumber\\
\leq\,& C\|(u,B)\|_{\dot{B}^{1/2}_{2,\infty}}\|(\nabla^{2}u_{j}, \nabla^{2} B_j)\|_{L^{2}}\big(\|\sqrt{\rho}\partial_{t}u_{j}\|_{L^{2}}+\|\partial_{t}B_{j}\|_{L^{2}}\big).
\end{align}
Making use of the elliptic estimates for \eqref{A2.42} and  Lemma \ref{L2.3}, we discover that there exists a generic constant $C_4^*$ such that 
\begin{align}\label{A4.10}
&\|(\nabla^{2}u_{j}, \nabla^{2} B_j)\|_{L^{2}}+\|\nabla P_{j}\|_{L^{2}}\nonumber\\
\leq\,&C\big(\|\sqrt{\rho}\partial_{t}u_{j}\|_{L^{2}}+\|u\cdot\nabla u_{j}\|_{L^{2}}+\|B\cdot\nabla B_{j}\|_{L^{2}}\big)\nonumber\\
&+C\big(\|\partial_{t}B_{j}\|_{L^{2}}+\|u\cdot\nabla B_{j}\|_{L^{2}}+\|B\cdot\nabla u_{j}\|_{L^{2}}\big)             \nonumber\\
\leq\,&C_4^*\Big(\|(\sqrt{\rho}\partial_{t}u_{j},B_j)\|_{L^{2}}+\|(u,B)\|_{\dot{B}^{1/2}_{2,\infty}}\|(\nabla^{2}u_{j}, \nabla^{2} B_j)\|_{L^{2}}\Big).
\end{align}

Now we denote
\begin{align}\label{A4.13}
T_{3}^*:=\sup\Big\{ t\in(0,T^{*}):\|(u,B)(t)\|_{\dot{B}_{2,\infty}^{1/2}}\leq c_3\Big\}   
\end{align}
for some $c_3>0$.
Then, for any $t\in(0,T_3^*)$, \eqref{A4.10} ensures that
\begin{align}
\|(\nabla^{2}u_{j}, \nabla^{2} B_j)\|_{L^{2}}+\|\nabla P_{j}\|_{L^{2}}\leq C\|(\sqrt{\rho}\partial_{t}u_{j},B_j)\|_{L^{2}}.\label{5.888}
\end{align}
By \eqref{A4.9} and \eqref{5.888}, one gets
\begin{align*}
&\frac{1}{2}\frac{{\rm d}}{{\rm d}t}\|(\nabla u_{j},\nabla B_j)\|_{L^{2}}^{2}+2c_{4}\|(\sqrt{\rho} \partial_{t}u_{j},\partial_t B_j,\nabla^{2}u_{j},\nabla^2 B_j,\nabla P_{j})\|_{L^{2}}^{2}\nonumber\\
&\quad  \leq  C\|(u,B)\|_{\dot{B}_{2,\infty}^{1/2}}\|(\nabla^{2}u_{j}, \nabla^{2} B_j)\|_{L^{2}}\|(\sqrt{\rho} \partial_{t}u_{j},\partial_t B_j)\|_{L^{2}}\\
&\quad  \quad+C\|(u,B)\|_{\dot{B}_{2,\infty}^{1/2}}^{2}\|(\nabla^{2}u_{j}, \nabla^{2} B_j)\|_{L^{2}}^{2},
\end{align*}
from which we obtain
\begin{align}\label{A4.12}
&\frac{{\rm d}}{{\rm d}t}\|(\nabla  u_j, \nabla B_j)\|_{L^2}^2+\|(\sqrt{\rho} \partial_{t}u_{j},\partial_t B_j,\nabla^{2}u_{j},\nabla^2 B_j,\nabla P_{j})\|_{L^{2}}^{2}\nonumber\\
&\quad \leq  C_5^*\|(u,B)\|_{\dot{B}{_{2,\infty}^{1/2}}}^2\|\nabla^2(u_j,B_j)\|_{L^2}^2.
\end{align}
for some constant $C_5^*=C_5^*(\rho_0)$.

We now claim that $T_{3}^*=T^{*}$. Otherwise if $T_{3}^*<T^{*}$, 
then for any $ 0\leq t\leq T_{3}^*$, it follows from \eqref{A4.12} that
\begin{align}\label{A4.14}
\frac{{\rm d}}{{\rm d}t}\|(\nabla  u_j, \nabla B_j)\|_{L^2}^2+\frac{1}{2}\|(\sqrt{\rho} \partial_{t}u_{j},\partial_t B_j,\nabla^{2}u_{j},\nabla^2 B_j,\nabla P_{j})\|_{L^{2}}^{2}\leq0,   
\end{align}
as long as $C_5^*c_3\leq \frac{1}{2}$. Integrating \eqref{A4.14} over $[0,t]$ for $t\leq T_{3}^*$, we discover that
\begin{align*}
&\quad\|(\nabla u_{j},\nabla B_j)\|_{L_{t}^{\infty}(L^{2})}+\|(\sqrt{\rho} \partial_{t}u_{j},\partial_t B_j)\|_{_{L_t^2(L^2)}}\leq C2^{j}\|\dot{\Delta}_{j}(u_{0},B_0)\|_{L^{2}}, 
\end{align*}
which, together with \eqref{5.888}, implies that
\begin{align}\label{A4.16}
&\quad\|(\nabla u_{j},\nabla B_j)\|_{L_{t}^{\infty}(L^{2})}+\|(\sqrt{\rho} \partial_{t}u_{j},\partial_t B_j,\nabla^{2}u_{j},\nabla^2 B_j,\nabla P_{j})\|_{_{L_t^2(L^2)}} \leq C2^{j/2}\|(u_0,B_0)\|_{\dot{B}_{2,\infty}^{1/2}}.   
\end{align}
By Bernstein's inequality, \eqref{A2.3}, \eqref{A4.8} and \eqref{A4.16},
we arrive at
\begin{align*}
&\|\dot{\Delta}_{j}(u,B)\|_{L_{t}^{\infty}(L^{2})}+\|\nabla\dot{\Delta}_{j}(u,B)\|_{L_{t}^{2}(L^{2})}\nonumber\\
\leq\,&C\sum_{j^{\prime}\geq j}\big(\|\dot{\Delta}_{j}(u_{j^{\prime}},B_{j^{\prime}})\|_{L_{t}^{\infty}(L^{2})}
+\|\nabla\dot{\Delta}_{j}(u_{j^{\prime}},B_{j^{\prime}})\|_{L_{t}^{2}(L^{2})}\big) \nonumber\\
& +C2^{-j}\sum_{j^{\prime}\leq j}\big(\|\nabla\dot{\Delta}_{j}(u_{j^{\prime}},B_{j^{\prime}})\|_{L_{t}^{\infty}(L^{2})}
+\|\nabla^{2}\dot{\Delta}_{j}(u_{j^{\prime}},B_{j^{\prime}})\|_{L_{t}^{2}(L^{2})}\big) \nonumber\\
\leq\,&C\sum_{j^{\prime}\geq j}\big(\|(u_{j^{\prime}},B_{j^{\prime}})\|_{L_{t}^{\infty}(L^{2})}+\|(\nabla u_{j^{\prime}},B_{j^{\prime}})\|_{L_{t}^{2}(L^{2})}\big) \nonumber\\
&+C2^{-j}\sum_{j^{\prime}\leq j}\big(\|\nabla (u_{j^{\prime}},B_{j^{\prime}})\|_{L_{t}^{\infty}(L^{2})}
+\|\nabla^{2}(u_{j^{\prime}},B_{j^{\prime}})\|_{L_{t}^{2}(L^{2})}\big) \nonumber\\
\leq\,&C2^{-j/2}\|(u_{0},B_0)\|_{\dot{B}^{1/2}_{2,\infty}}.
\end{align*}
Consequently, it holds that
\begin{align*}
\|(u,B)\|_{L_{t}^{\infty}(\dot{B}_{2,\infty}^{1/2})}+\|(u,B)\|_{\tilde{L}_{t}^{2}(\dot{B}_{2,\infty}^{3/2})}\leq C\|(u_{0},B_0)\|_{\dot{B}_{2,\infty}^{1/2}},\quad 0\leq t<T^*_3.
\end{align*}
Since $u,B\in C([0,T^*);\dot{B}^{1/2}_{2,\infty})$ and $\|(u_{0},B_0)\|_{\dot{B}_{2,\infty}^{1/2}}$ is suitably small, a standard continuity argument implies that there exists some $t\in (T_3^*,T^*)$ such that $\|(u,B)(t)\|_{\dot{B}^{1/2}_{2,\infty}}\leq c_3$, which contradicts with the definition of $T^*_3$ in \eqref{A4.13}. Thus, we conclude that $T_3^*=T^*$ and   \eqref{A4.4} holds with $r=\infty$. As $u_0,B_0\in \dot{H}^{1/2}=\dot{B}^{1/2}_{2,2}$, one has $ u_0,B_0\in\dot{B}_{2,r}^{1/2}$ with $2\leq r\leq\infty$. Combining \eqref{A4.7} with \eqref{A4.16}, we further infer
\begin{align*}
&\|\dot{\Delta}_{j}(u,B)\|_{L_{t}^{\infty}(L^{2})} +\|\nabla\dot{\Delta}_{j}(u,B)\|_{L_{t}^{2}(L^{2})}\nonumber\\
\,\leq\,& C\sum_{j^{\prime}\geq j}\|\dot{\Delta}_{j}(u_{0}, B_0)\|_{L^{2}}+C2^{-j}\sum_{j^{\prime}\leq j}\|\nabla\dot{\Delta}_{j}(u_{0},B_0)\|_{L^{2}} \nonumber \\
 \leq\,&C\sum_{j^{\prime}\geq j}c_{j^{\prime},r}2^{-{j^{\prime}}/{2}}\|(u_{0},B_0)\|_{B_{2,r}^{1/2}}+C2^{-j}\sum_{j^{\prime}\leq j}c_{j^{\prime},r}2^{{j^{\prime}}/{2}}\|(u_{0},B_0)\|_{B_{2,r}^{1/2}}\nonumber\\
 \leq\,& Cc_{j,r}2^{-j/2}\|(u_{0},B_0)\|_{B_{2,r}^{1/2}}.
\end{align*}
This yields \eqref{A4.4} with $2\leq r<\infty$.

Now we turn to the proof of \eqref{A4.6}. Let $0<\alpha<1/2$ and $t_j=2^{-2j}$. Then, we obtain
\begin{align*}
&\quad\int_{0}^{t}\tau^{-2\alpha}\|(\nabla u_{j},\nabla B_j)\|_{L^{2}}^{2}\,{\rm d}\tau\, \\
& \leq\int_{0}^{t_j}\tau^{-2\alpha}\|(\nabla u_{j},\nabla B_j)\|_{L^{2}}^{2}\,{\rm d}\tau+\int_{t_j}^{t}\tau^{-2\alpha}\|(\nabla u_{j},\nabla B_j)\|_{L^{2}}^{2}\,{\rm d}\tau  \nonumber\\
& \leq t_0^{1-2\alpha}\,{\rm d}\tau\,\|(\nabla u_{j},\nabla B_j)\|_{L^{\infty}_t(L^{2})}^{2}+t_0^{-2\alpha}\|(\nabla u_{j},\nabla B_j)\|_{L^2_t(L^{2})}^{2},
\end{align*}
which gives rise to
\begin{align}\label{A4.22}
\|\tau^{-\alpha}(\nabla u_{j},\nabla B_j)\|_{L^2_t(L^{2})}\leq C2^{2\alpha j}\|\dot{\Delta}_{j}(u_{0},B_0)\|_{L^{2}},\quad 0\leq t<T^*,
\end{align}
where \eqref{A4.8} and \eqref{A4.16} have been used. Owing to \eqref{A4.7} and \eqref{A4.22}, it holds that
\begin{align*}
\|t^{-{1}/{4}}(\nabla u,\nabla B)\|_{L_{t}^{2}(L^{2})}^{2}& \leq2\sum_{k\in\mathbb{Z}}\| \nabla (u_{k},B_k)\|_{L_{t}^{2}(L^{2})}\sum_{j\leq k}\|t^{-{1}/{2}}(\nabla u_{j},\nabla B_j)\|_{L_{t}^{2}(L^{2})}  \nonumber\\
&\leq C\|(u_{0},B_0)\|_{\dot{H}^{1/2}}^{2}\sum_{k\in\mathbb{Z}}2^{-{k}/{2}}c_{k,2}\sum_{j\leq k}2^{{j}/{2}}c_{j,2}\nonumber \\
&\leq C\|(u_{0},B_0)\|_{\dot{H}^{1/2}}^{2}\sum_{k\in\mathbb{Z}}c_{k,2}^{2}\nonumber\\
&\leq C\|(u_{0},B_0)\|_{\dot{H}^{1/2}}^{2}.
\end{align*}
This completes the proof of Proposition \ref{P4.1}.
\end{proof}

\begin{prop}\label{P4.2}
Under the assumptions in Proposition \ref{P4.1}, it holds for any $T\in(0,T^*)$ and $r\in[2,\infty]$ that
\begin{align}\label{A4.29}
\|t^{1/2}(u,B)\|_{\widetilde{L}_{T}^{\infty}(\dot{B}_{2,r}^{3/2})}+\|t^{1/2}(\partial_{t}u, \partial_{t}B)\|_{\widetilde{L}_{T}^{2}(\dot{B}_{2,r}^{1/2})}\leq Ce^{C\|(u_{0},B_0)\|_{\dot{H}^{1/2}}^{2}}\|(u_{0},B_0)\|_{\dot{B}_{2,r}^{1/2}}.
\end{align}

\end{prop}
\begin{proof}
Multiplying \eqref{A4.14} by $t$, we have
\begin{align}\label{A4.30}
&\frac{{\rm d}}{{\rm d}t}\|t^{1/2}(\nabla u_{j},\nabla B_j)\|_{L^{2}}^{2}+{c}_{4}\|t^{1/2}(\sqrt{\rho} \partial_{t}u_{j},\partial_t B_j, \nabla^{2}u_{j},\nabla^2B_j, \nabla P_{j})\|_{L^{2}}^{2} \leq\, \|(\nabla u_{j},\nabla B_j)\|_{L^{2}}^{2}.    
\end{align}
Integrating \eqref{A4.30} over $[0,t]$ and using \eqref{A4.7}, we get
\begin{align}\label{A4.31}
&\|t^{1/2}(\nabla u_{j},\nabla B_j)\|_{L_{t}^{\infty}(L^{2})}+\|t^{1/2}(\partial_{t}u_{j},\partial_t B_j,\nabla^{2}u_{j},\nabla^2 B_j,\nabla  P_{j})\|_{L_{t}^{2}(L^{2})}\nonumber\\
&\quad \leq  C\|(\nabla u_{j},\nabla B_j)\|_{L_{t}^{2}(L^{2})}\leq C\|\dot{\Delta}_{j}(u_{0},B_0)\|_{L^{2}}.
\end{align}

To achieve \eqref{A4.29}, we need to establish higher-order estimates of $(u,B)$. One deduces from \eqref{A2.23} that 
\begin{align}\label{A4.32}
&\frac{1}{2}\frac{\rm d}{{\rm d}t}\bigg(\int_{\mathbb{R}^{3}}\rho|\partial_{t}u_{j}|^{2}\, {\rm d}x+\int_{\mathbb{R}^{3}}|\partial_{t}B_{j}|^{2}\, {\rm d}x\bigg)+\|\nabla\partial_{t}u_{j}\|_{L^{2}}^{2}+\|\nabla\partial_{t}B_{j}\|_{L^{2}}^{2}\nonumber\\
=\,&-\int_{\mathbb{R}^{3}}(\rho_{t}\partial_t u_{j})\cdot
\partial_{t}u_{j}\, {\rm d}x-\int_{\mathbb{R}^{3}}(\rho_{t}\partial_t u_{j})\cdot
u\nabla u_{j}\, {\rm d}x
-\int_{\mathbb{R}^{3}}(\rho \partial_t u\cdot\nabla u_{j})\cdot\partial_{t}u_{j}\, {\rm d}x\nonumber\\
&+\int_{\mathbb{R}^{3}}(\partial_t B\cdot\nabla u_{j})\cdot\partial_{t}B_{j}\, {\rm d}x-\int_{\mathbb{R}^{3}}(\partial_t u\cdot\nabla B_{j})\cdot\partial_{t}B_{j}\, {\rm d}x-\int_{\mathbb{R}^{3}}(\partial_t B\cdot\nabla B_{j})\cdot\partial_{t}u_{j}\, {\rm d}x\nonumber\\
\equiv:\,&\sum_{j=11}^{16}I_j.
\end{align}
The terms $I_j$ ($i=11,\dots, 16$) are analyzed as follows. First, we easily obtain that
\begin{align}\label{A4.33}
I_{11}+I_{12}+I_{13}\leq&\, C\|(\partial_t u,\partial_t B)\|_{L^2}\|\nabla \partial_t (u_j,B_j)\|_{L^2}
\|(\nabla  u_j, \nabla B_j)\|_{L^2}^{1/2}\|\nabla^2 (u_j,B_j)\|_{L^2}^{1/2}\nonumber\\
&\,+C\|(u,B)\|_{\dot{B}^{1/2}_{2,\infty}}\|(u,B)\|_{\dot{B}^{3/2}_{2,\infty}}\|\nabla \partial_t (u_j,B_j)\|_{L^2}
\|\nabla^2 (u_j,B_j)\|_{L^2}\nonumber\\
&\,+C\|(u,B)\|_{\dot{B}^{1/2}_{2,\infty}}
\|\nabla^2 (u_j,B_j)\|_{L^2}^2.
\end{align}
By using Sobolev's and Gagliardo-Nirenberg inequalities, we have
\begin{align}\label{A4.34}
I_{14}+I_{15}+I_{16}\leq\,& C\|\partial_t B\|_{L^2}\|\partial_t B_j\|_{L^6}\|\nabla u_j\|_{L^3}
+C\|\partial_t u\|_{L^2}\|\partial_t B_j\|_{L^6}\|\nabla B_j\|_{L^3}\nonumber\\
&+\|\partial_t B\|_{L^2}\|\partial_t u_j\|_{L^6}\|\nabla B_j\|_{L^3}\nonumber\\
\leq&\,C\|(\partial_t u,\partial_t B)\|_{L^2}\|\nabla\partial_t (u_j,B_j)\|_{L^2}\|(\nabla  u_j, \nabla B_j)\|_{L^2}^{1/2}\|\nabla^2 (u_j,B_j)\|_{L^2}^{1/2}.
\end{align}
It thus follows from \eqref{A4.32}--\eqref{A4.34} that
\begin{align*}
&\frac{1}{2}\frac{{\rm d}}{{\rm d}t}\|(\sqrt{\rho} \partial_{t}u_{j},\partial_t B_j)\|_{L^{2}}^{2}+\|(\nabla\partial_{t}u_{j},\nabla\partial_{t}B_j)\|_{L^{2}}^{2}\nonumber\\
\leq\,& C\|(\partial_t u,\partial_t B)\|_{L^2}\|\nabla \partial_t (u_j,B_j)\|_{L^2}
\|(\nabla  u_j, \nabla B_j)\|_{L^2}^{1/2}\|\nabla^2 (u_j,B_j)\|_{L^2}^{1/2}\nonumber\\
& +C\|(u,B)\|_{\dot{B}^{1/2}_{2,\infty}}\|(u,B)\|_{\dot{B}^{3/2}_{2,\infty}}\|\nabla \partial_t (u_j,B_j)\|_{L^2}\|\nabla^2 (u_j,B_j)\|_{L^2}\nonumber\\
& +C\|(u,B)\|_{\dot{B}^{1/2}_{2,\infty}}
\|\nabla^2 (u_j,B_j)\|_{L^2}^2\nonumber\\
\leq\,& C\|(\partial_t u,\partial_t B)\|_{L^2}^2
\|(\nabla  u_j, \nabla B_j)\|_{L^2}\| (\sqrt{\rho}\partial_tu_j,\partial_t B_j)\|_{L^2}\nonumber\\
& +C\|(u,B)\|_{\dot{B}^{1/2}_{2,\infty}}^2\|(u,B)\|_{\dot{B}^{3/2}_{2,\infty}}^2\| (\sqrt{\rho}\partial_tu_j,\partial_t B_j)\|_{L^2}^2,
\end{align*}
which allows us to write
\begin{align}\label{A4.36}
&\frac{{\rm d}}{{\rm d}t}\|t^{1/2}(\sqrt{\rho} \partial_{t}u_{j},\partial_t
 B_j)\|_{L^{2}}^{2}+\|t^{1/2}(\nabla\partial_{t}u_{j},\nabla\partial_{t}B_j)\|_{L^{2}}^{2}\nonumber\\
 \leq\,&\|(\sqrt{\rho} \partial_{t}u_{j},\partial_t B_j)\|_{L^{2}}^{2}+C\|(u,B)\|_{\dot{B}_{2,\infty}^{1/2}}^{2}\|(u,B)\|_{\dot{B}_{2,\infty}^{3/2}}^{2}\|t^{1/2}(\sqrt{\rho} \partial_{t}u_{j},\partial_t B_j)\|_{L^{2}}^{2}\nonumber\\
 \,&+C\|t^{1/4} (\partial_t u,\partial_t B)\|_{L^{2}}^{2}\|(\nabla u_{j},\nabla B_j)\|_{L^{2}}\|t^{1/2}(\sqrt{\rho} \partial_{t}u_{j},\partial_t B_j)\|_{L^{2}}\nonumber\\
 \leq\,& C\Big(\|(u,B)\|_{\dot{B}_{2,\infty}^{1/2}}^{2}\|(u,B)\|_{\dot{B}_{2,\infty}^{3/2}}^{2}+\|t^{1/4}(\partial_t u,\partial_t B)\|_{L^{2}}^{2}\Big)\|t^{1/2}(\sqrt{\rho} \partial_{t}u_{j},\partial_t B_j)\|_{L^{2}}^{2}\nonumber\\
 &+C\big(\|(\partial_{t}u_{j},\partial_t B_j)\|_{L^{2}}^{2}+\|t^{1/4}(\partial_t u,\partial_t B)\|_{L^{2}}^{2}\|(\nabla u_{j},\nabla B_j)\|_{L^{2}}^{2}\big).
\end{align}
Employing Gronwall’s inequality for \eqref{A4.36}, we have
\begin{align}\label{A4.37}
&\|t^{1/2}(\partial_{t}u_{j},\partial_t B_j)\|_{L_{t}^{\infty}(L^{2})}^{2}+\|t^{1/2}(\nabla\partial_{t}u_{j},\nabla\partial_{t}B_j)\|_{L_{t}^{2}(L^{2})}^{2} \nonumber\\
&\quad\leq C\|(\partial_{t}u_{j},\partial_t B_j)\|_{L_{t}^{2}(L^{2})}^{2}+C\|t^{1/4}(\partial_t u,\partial_t B)\|_{L_{t}^{2}(L^{2})}^{2}\|(\nabla u_{j},\nabla B_j)\|_{L_{t}^{\infty}(L^{2})}^{2} \nonumber\\
&\quad\quad \times\exp\bigg\{C\|(u,B)\|_{L_{t}^{\infty}(\dot{B}_{2,\infty}^{1/2})}^{2}\|(u,B)\|_{L_{t}^{2}(\dot{B}_{2,\infty}^{3/2})}^{2}+C\| t^{1/4}(\partial_t u,\partial_t B)\|_{L_{t}^{2}(L^{2})}^{2}\bigg\}.
\end{align}
It follows from \eqref{A1.7}, \eqref{A4.4} and \eqref{A4.6} that
\begin{align}\label{A4.38}
\|(u,B)\|_{L_{t}^{\infty}(\dot{B}_{2,\infty}^{1/2})}^{2}\|(u,B)\|_{L_{t}^{2}(\dot{B}_{2,\infty}^{3/2})}^{2}\leq\,& C\|(u,B)\|_{L_{t}^{\infty}(\dot{B}_{2,\infty}^{1/2})}^{2}\|(u,B)\|_{L_{t}^{2}(\dot{H}^{3/2})}^{2}\nonumber\\
\leq\,& C\|(u_{0},B_0)\|_{\dot{B}_{2,\infty}^{1/2}}^{2}\|(u_{0},B_0)\|_{\dot{H}^{1/2}}^{2}\nonumber\\
\leq\,& C\|(u_{0},B_0)\|_{\dot{H}^{1/2}}^{2}.    
\end{align} 
Also, thanks to \eqref{A4.16} and \eqref{A4.31}, it holds
\begin{align}\label{A4.39}
\|t^{1/4} (\partial_t u,\partial_t B)\|_{L_{t}^{2}(L^{2})}^{2}
\leq&\,C\sum_{j\in\mathbb{Z}}\sum_{k\leq j}\|t^{1/2}(\partial_{t}u_{j},\partial_{t}B_j)\|_{L_{t}^{2}(L^{2})}\|\partial_{t}(u_{k},B_k)\|_{L_{t}^{2}(L^{2})}\nonumber\\
\leq&\,C\sum_{j\in\mathbb{Z}}\sum_{k\leq j}\|\dot{\Delta}_{j}(u_{0},B_0)\|_{L^{2}}\|\nabla\dot{\Delta}_{k}(u_{0},B_0)\|_{L^{2}} \nonumber\\
\leq&\,C\|(u_{0},B_0)\|_{\dot{H}^{1/2}}^{2}\sum_{j\in\mathbb{Z}}c_{j}2^{-j/2}\sum_{k\leq j}c_{k}2^{{k}/{2}}\nonumber\\
\leq&\,C\|(u_{0},B_0)\|_{\dot{H}^{1/2}}^{2}\sum_{j\in\mathbb{Z}}c_{j}^{2}\nonumber\\
\leq &\,C\|(u_{0},B_0)\|_{\dot{H}^{1/2}}^{2},
\end{align}
which, together with \eqref{5.888} and Lemma \ref{L2.3}, ensures that
\begin{align}\label{A4.40}
&\|t^{1/4} (\nabla^{2}u,\nabla^2 B,\nabla P)\|_{L_{t}^{2}(L^{2})}\nonumber\\
 \leq\,& C\|t^{1/4} (\partial_t u,\partial_t B)\|_{L_{t}^{2}(L^{2})}+\|t^{1/4} (u,B)\cdot(\nabla u,\nabla B)\|_{L_{t}^{2}(L^{2})}\nonumber\\
 \leq\,& C\|t^{1/4} (\partial_t u,\partial_t B)\|_{L_{t}^{2}(L^{2})}+\|(u,B)\|_{L_{t}^{\infty}(\dot{B}_{2,\infty}^{1/2})}\|t^{1/4}(\nabla^2 u,\nabla^2 B)\|_{L_{t}^{2}(L^{2})}.
\end{align}
Collecting \eqref{A4.4} and \eqref{A4.40} up, we get
\begin{align}\label{A4.41}
\|t^{1/4}(\nabla^{2}u,\nabla^2 B,\nabla P,\partial_t u,\partial_t B)\|_{L_{t}^{2}(L^{2})}\leq C\|t^{1/4}(\partial_t u,\partial_t B)\|_{L_{t}^{2}(L^{2})}\leq C\|(u_{0},B_0)\|_{\dot{H}^{1/2}}. 
\end{align}
Then, substituting \eqref{A4.16}, \eqref{A4.38} and \eqref{A4.39} into \eqref{A4.37} leads to
\begin{align}\label{5.35}
&\|t^{1/2}(\partial_{t}u_{j},\partial_t B_j)\|_{L_{t}^{\infty}(L^{2})}+\|t^{1/2}(\nabla\partial_{t}u_{j},\nabla\partial_t B_j)\|_{L_{t}^{2}(L^{2})}\nonumber\\
&\quad\leq Ce^{C\|(u_{0},B_0)\|_{\dot{H}^{1/2}}^{2}}2^j\|\dot{\Delta}_{j}(u_{0},B_0)\|_{L^{2}},\quad 0\leq t<T^*.
\end{align}

Thus, for any $2\leq r\leq \infty$, we obtain from \eqref{5.888}, \eqref{A4.31}, \eqref{5.35} that
\begin{align*}
&\|t^{1/2}(\nabla  u_j, \nabla B_j)\|_{L_t^\infty(L^2)}+\|t^{1/2}(\partial_tu_j,\partial_t B_j,\nabla^2u_j, \nabla^2 B_j,\nabla  P_j)\|_{L_t^2(L^2)}\nonumber\\
&\quad\leq  C c_{j,r}2^{-j/2}\|(u_0,B_0)\|_{\dot{B}_{2,r}^{1/2}},
\end{align*}
and
\begin{align*}
&\|t^{1/2}(\nabla^2  u_j, \nabla B_j)\|_{L_t^\infty(L^2)}+\|t^{1/2}(\nabla \partial_tu_j,\nabla \partial_t B_j,\nabla^2u_j, \nabla^2 B_j,\nabla  P_j)\|_{L_t^2(L^2)}\nonumber\\
&\quad\leq  Ce^{C\|(u_{0},B_0)\|_{\dot{H}^{1/2}}^{2}} c_{j,r}2^{j/2}\|(u_0,B_0)\|_{\dot{B}_{2,r}^{1/2}},
\end{align*}
which lead to
\begin{align*}
&\|t^{1/2}\nabla\dot{\Delta}_{j}(u,B)\|_{L_{t}^{\infty}(L^{2})}+\|t^{1/2}\dot{\Delta}_{j}(\partial_tu,\partial_tB)\|_{L_{t}^{2}(L^{2})}\nonumber \\
=\,&\sum_{j^{\prime}\geq j}\Big(\|t^{1/2}\dot{\Delta}_{j}(\nabla u_{j^{\prime}},\nabla B_{j^{\prime}})\|_{L_{t}^{\infty}(L^{2})}+\|t^{1/2}\dot{\Delta}_{j}(\partial_{t} u_{j^{\prime}}, \partial_{t}B_{j^{\prime}})\|_{L_{t}^{2}(L^{2})}\Big)\nonumber \\
&+2^{-j}\sum_{j^{\prime}\leq j-1}\Big(\|t^{1/2}\dot{\Delta}_{j}(\nabla^2 u_{j^{\prime}},\nabla^2 B_{j^{\prime}})\|_{L_{t}^{\infty}(L^{2})}+\|t^{1/2}\dot{\Delta}_{j}(\nabla\partial_{t} u_{j^{\prime}}, \nabla\partial_{t}B_{j^{\prime}})\|_{L_{t}^{2}(L^{2})}\Big) \nonumber\\
\leq\,&C\sum_{j^{\prime}\geq j}\|\dot{\Delta}_{j}(u_{0},B_0)\|_{L^{2}}+C2^{-j}\sum_{j^{\prime}\leq j-1}\|\nabla\dot{\Delta}_{j}(u_{0},B_0)\|_{L^{2}},\nonumber\\
\leq\,&C\sum_{j^{\prime}\geq j}c_{j^{\prime},r}2^{-{j^{\prime}}/{2}}\|(u_0,B_0)\|_{\dot{B}_{2,r}^{1/2}}+C2^{-j}\sum_{j^{\prime}\leq j-1}c_{j^{\prime},r}2^{{j^{\prime}}/{2}}\|(u_0,B_0)\|_{\dot{B}_{2,r}^{1/2}}\nonumber\\
\leq\,&C e^{C\|(u_{0},B_0)\|_{\dot{H}^{1/2}}^{2}} c_{j,r}2^{-j/2}\|(u_0,B_0)\|_{\dot{B}_{2,r}^{1/2}}.
\end{align*}
This gives rise to \eqref{A4.29} and completes the proof of Proposition \ref{P4.2}.
\end{proof}

Finally, we establish some additionally uniform bounds of $(u,B)$ leading to the key properties for the uniqueness.

\begin{prop}\label{P4.3}
Under the assumptions in Proposition \ref{P4.3}, it holds that
\begin{align}\label{A4.53}
\|(\nabla u,\nabla B)\|_{L_{t}^{4}(L^{2})}&+\|(u,B)\|_{L^2_T(L^{\infty})}+\|t^{1/4}(\nabla u,\nabla B)\|_{L_{T}^{\infty}(L^{2})}\leq C(u_0,B_0), \\\label{A4.54}
\|t^{3/4}(\partial_{t}u,\partial_t B,\nabla^{2}u,&\nabla^2 B)\|_{L_{T}^{\infty}(L^{2})}+\|t^{3/4}(\nabla^{2}u,\nabla^2B , \nabla P)\|_{L_{T}^{2}(L^{6})}\leq\ C(u_0,B_0), \\\label{A4.55}
&\|t^{1/2}(\nabla u,\nabla B)\|_{L_{T}^{2}(L^{\infty})}\leq C(u_0,B_0),\\\label{A4.56}
\|t^{3/4}(D_{t}u,D_{t}B)\|_{L_{T}^{\infty}(L^{2})}&+\|t^{3/4}(\nabla \partial_t u,\nabla \partial_t B,\nabla D_{t}u,\nabla D_{t} B)\|_{L_{T}^{2}(L^{2})}\leq C(u_0,B_0),
\end{align}
for any $T<T^*$ and $C(u_0,B_0)=Ce^{C\|(u_{0},B_0)\|_{\dot{H}^{1/2}}^{2}} \|(u_{0},B_0)\|_{\dot{H}^{1/2}}$.   
\end{prop}

\begin{proof}
According to \eqref{A1.7}, \eqref{A4.4} and \eqref{A4.29}, we have
\begin{align}\label{A4.57}
\|(\nabla u,\nabla B)\|_{L_{t}^{4}(L^{2})}^{2}\leq &\,C\|(u,B)\|_{L_{t}^{\infty}(\dot{B}_{2,\infty}^{1/2})}\|(u,B)\|_{\widetilde{L}_{t}^{2}(\dot{B}_{2,\infty}^{3/2})}\nonumber\\
\leq &\,C\|(u,B)\|_{L_{t}^{\infty}(\dot{B}_{2,\infty}^{1/2})}\|(u,B)\|_{{L}_{t}^{2}(\dot{H}^{3/2})}\nonumber\\
\leq&\, C\|(u_{0},B_0)\|_{\dot{B}_{2,\infty}^{1/2}}\|(u_{0},B_0)\|_{\dot{H}^{1/2}}, \nonumber\\
\leq&\, C\|(u_{0},B_0)\|_{\dot{H}^{1/2}}^2, 
\end{align}
and
\begin{align}\label{A4.58}
\|t^{1/4}(\nabla u,\nabla B)\|_{L_{t}^{\infty}(L^2)}^{2}\leq &\,C\|(u,B)\|_{L_{t}^{\infty}(\dot{B}_{2,\infty}^{1/2})}\|t^{1/2}(u,B)\|_{L_{t}^{\infty}(\dot{B}_{2,\infty}^{3/2})}\nonumber\\
\leq&\,C\|(u_{0},B_0)\|_{\dot{B}_{2,\infty}^{1/2}}^{2}.
\end{align}
By virtue of Gagliardo-Nirenberg inequality $\|f\|_{L^\infty}^2\leq C\|\nabla f\|_{L^2}\|\nabla^2 f\|_{L^2}$, we deduce from \eqref{A4.6} and \eqref{A4.41}
that
\begin{align}\label{A4.46}
\|(u,B)\|_{L_{t}^{2}(L^{\infty})}^{2}\, \leq&\, C\int_{0}^{t}\|{t^{\prime}}^{-\frac{1}{4}}(\nabla u,\nabla B)\|_{L^{2}}\|{t^{\prime}}^{\frac{1}{4}}\nabla^{2} (u,B)\|_{L^{2}} {\rm d}{t^{\prime}} \nonumber\\
 \leq &\, C\|t^{-\frac{1}{4}}(\nabla u,\nabla B)\|_{L_{t}^{2}(L^{2})}\|t^{1/4}(\nabla^2 u,\nabla^2 B)\|_{L_{t}^{2}(L^{2})}\nonumber\\
\leq&\,C\|(u_{0},B_0)\|_{\dot{H}^{1/2}}.
\end{align}
The desired \eqref{A4.53} follows from  \eqref{A4.57}--\eqref{A4.46}.

Let us prove the inequality \eqref{A4.54}. Similar to the arguments used in \eqref{A4.32}, a direct calculation yields
\begin{align*}
&\frac{1}{2}\frac{\rm d}{{\rm d}t}\bigg(\int_{\mathbb{R}^{3}}\rho|\partial_{t}u_{j}|^{2}\, {\rm d}x+\int_{\mathbb{R}^{3}}|\partial_{t}B_{j}|^{2}\, {\rm d}x\bigg)+\|\nabla\partial_{t}u_{j}\|_{L^{2}}^{2}+\|\nabla\partial_{t}B_{j}\|_{L^{2}}^{2}\nonumber\\
=\,&-\int_{\mathbb{R}^{3}}(\rho_{t}\partial_t u_{j})\cdot
\partial_{t}u_{j}\, {\rm d}x-\int_{\mathbb{R}^{3}}(\rho_{t}\partial_t u_{j})\cdot
u\nabla u_{j}\, {\rm d}x
-\int_{\mathbb{R}^{3}}(\rho \partial_t u\cdot\nabla u_{j})\cdot\partial_{t}u_{j}\, {\rm d}x\nonumber\\
& +\int_{\mathbb{R}^{3}}(\partial_t B\cdot\nabla u_{j})\cdot\partial_{t}B_{j}\, {\rm d}x-\int_{\mathbb{R}^{3}}(\partial_t u\cdot\nabla B_{j})\cdot\partial_{t}B_{j}\, {\rm d}x-\int_{\mathbb{R}^{3}}(\partial_t B\cdot\nabla B_{j})\cdot\partial_{t}u_{j}\, {\rm d}x\nonumber\\
\leq\,&C\|(u,B)\|_{\dot{B}_{2,\infty}^{1/2}}^2\|(u,B)\|_{\dot{B}_{2,\infty}^{3/2}}^2\|\nabla^2(u,B)\|_{L^2}^2+C\|(\sqrt{\rho}\partial_tu,\partial_t B)\|_{L^2}^2\|(\nabla u,\nabla B)\|_{\dot{H}^{1/2}}^2.
\end{align*}
From \eqref{I-1} and Lemma \ref{L2.3}, we have
\begin{align}\label{A4.60}
&\|(\nabla^{2}u,\nabla^2B,\nabla P)\|_{L^{2}}^{2}\nonumber\\
\leq&\, C(\|\sqrt{\rho} \partial_{t}u\|_{L^{2}}^{2}+\|u\cdot\nabla u\|_{L^{2}}^{2}+\|B\cdot\nabla B\|_{L^{2}}^{2}+\|u\cdot\nabla B\|_{L^{2}}^{2}+\|B\cdot\nabla u\|_{L^{2}}^{2})\nonumber\\
\leq&\, C\|(\sqrt{\rho} \partial_{t}u,\partial_t B)\|_{L^{2}}^{2}+C\|(u,B)\|_{\dot{B}_{2,\infty}^{2}}^{2}\|\nabla^{2} (u,B)\|_{L^{2}}^{2},   
\end{align}
which, together with \eqref{A1.7} and \eqref{A4.4},  gives
\begin{align}\label{A4.61}
\|(\nabla^2u,\nabla^2B, \nabla P)\|_{L^2}^2\leq C\|\sqrt{\rho} \partial_{t}u\|_{L^2}^2+C\|\partial_t B\|_{L^2}^2.    
\end{align}
Combining \eqref{A4.60} and \eqref{A4.61} up, we derive that
\begin{align}\label{A4.62}
\frac{{\rm d}}{{\rm d}t}\|(\sqrt{\rho} &\partial_tu,\partial_t B)\|_{L^2}^2+\|\nabla\partial_t(u,B)\|_{L^2}^2\nonumber\\
\leq\,&C\Big(\|(u,B)\|_{\dot{B}_{2,\infty}^{1/2}}^2\|(u,B)\|_{\dot{B}_{2,\infty}^{3/2}}^2+\|(\nabla u,\nabla B)\|_{\dot{H}^{1/2}}^2\Big)\|(\sqrt{\rho} \partial_tu,\partial_t B)\|_{L^2}^2.
\end{align}
By multiplying \eqref{A4.62} by $t^{3/2}$, we obtain that
\begin{align*}
&\frac{{\rm d}}{{\rm d}t}\|t^{3/4}(\sqrt{\rho} \partial_{t}u,\partial_t B)\|_{L^{2}}^{2}+\|t^{3/4}(\nabla\partial_{t}u, \nabla\partial_{t}B)\|_{L^{2}}^{2} 
\\
&\quad\leq\frac{3}{2}\|t^{1/4}(\partial_{t}u,\partial_t B)\|_{L^{2}}^{2}\\
& \quad\quad  +C\Big(\|(u,B)\|_{\dot{B}_{2,\infty}^{1/2}}^2\|(u,B)\|_{\dot{B}_{2,\infty}^{3/2}}^2+\|(\nabla u,\nabla B)\|_{\dot{H}^{1/2}}^2\Big)\|(\sqrt{\rho} \partial_tu,\partial_t B)\|_{L^2}^2.
\end{align*}
Using Gronwall's inequality, we have
\begin{align*}
&\|t^{3/4}(\partial_{t}u,\partial_t B)\|_{L_{t}^{\infty}(L^{2})}^{2}+\|t^{3/4}(\nabla\partial_{t}u, \nabla\partial_{t}B)\|_{L_{t}^{2}(L^{2})}^{2}\nonumber \\
&\quad \leq \frac{3}{2}\|t^{1/4}(\partial_t u,\partial_t B)\|_{L_{t}^{2}(L^{2})}^{2} \exp\Big\{C\Big(1+ \|(u,B)\|_{L_{t}^{\infty}(\dot{B}_{2,\infty}^{1/2})}^{2}\Big) \|(u,B)\|_{L_{t}^{2}(\dot{H}^{3/2})}^{2}\Big\}.
\end{align*}
Combines this with \eqref{A4.4}, \eqref{A4.6}, \eqref{A4.39} and \eqref{A4.61}, we obtain that
\begin{align}\label{A4.65}
&\|t^{3/4}(\partial_{t}u,\nabla^{2}u,\nabla^2 B)\|_{L_{t}^{\infty}(L^{2})}^{2}+\|t^{3/4}(\nabla\partial_{t}u, \nabla\partial_{t}B)\|_{L_{t}^{2}(L^{2})}^{2}\nonumber\\
&\quad\leq  C   e^{C\|(u_{0},B_0)\|_{\dot{H}^{1/2}}^{2}} \|(u_{0},B_0)\|_{\dot{H}^{1/2}}^{2}. 
\end{align}
From \eqref{I-1} and the classical elliptic theory, one infers
\begin{align}
&\|t^{3/4}(\nabla^{2}u,\nabla^2B, \nabla P)\|_{L_{t}^{2}(L^{6})}\nonumber\\ \leq\,&\|t^{3/4}\rho(\partial_{t}u+u\cdot\nabla u)\|_{L_{t}^{2}(L^{6})}+\|t^{3/4}(B\cdot\nabla u,B\cdot\nabla B,u\cdot\nabla B)\|_{L_{t}^{2}(L^{6})} \nonumber\\
\leq\,&C\|t^{3/4}(\partial_{t}u,\partial_t B)\|_{L_{t}^{2}(L^{6})}+C\|t^{1/4}(u,B)\|_{L_{t}^{\infty}(L^{6})}\|t^{1/2}(\nabla u,\nabla B)\|_{L_{t}^{2}(L^{\infty})}\nonumber\\
\leq\,&C\|t^{3/4}\nabla(\partial_{t}u,\partial_t B)\|_{L_{t}^{2}(L^{2})}\nonumber\\
&+C\|t^{1/4}(\nabla u,\nabla B)\|_{L_{t}^{\infty}(L^{2})}\|t^{1/4}(\nabla u,\nabla B)\|_{L_{t}^{2}(L^{6})}^{1/2}\|t^{3/4}\nabla^{2} (u,B)\|_{L_{t}^{2}(L^{6})}^{1/2}\nonumber\\
\leq\,&C\|t^{3/4}\nabla(\partial_{t}u,\partial_t B)\|_{L_{t}^{2}(L^{2})}+\|(u,B)\|_{L_{t}^{\infty}(\dot{H}^{1/2})}\|t^{1/2}(u,B)\|_{L_{t}^{\infty}(\dot{H}^{3/2})}\|t^{3/4}(\nabla^2 u,\nabla^2 B)\|_{L_{t}^{2}(L^{2})}.\nonumber
\end{align}
This, combined with \eqref{A4.6}, \eqref{A4.29} with $r=2$, and \eqref{A4.65}, ensures that
\begin{align}\label{A4.67}
\|t^{3/4}(\nabla^{2}u,\nabla^2B, \nabla P)\|_{L_{t}^{2}(L^{6})}\leq C e^{C\|(u_{0},B_0)\|_{\dot{H}^{1/2}}^{2}} \|(u_{0},B_0)\|_{\dot{H}^{1/2}}.    
\end{align}
In accordance with \eqref{A4.65} and \eqref{A4.67}, \eqref{A4.54} is proved.

With the aid of \eqref{A4.39}, \eqref{A4.54} and the Gagliardo-Nirenberg-Sobolev inequality $$\|f\|_{L^{\infty}}\leq C\|f\|_{L^2}^{1/2}\|\nabla f\|_{L^6}^{1/2},$$ we get
\begin{align*}
\|t^{1/2}(\nabla u,\nabla B)\|_{L_{t}^{2}(L^{\infty})}^{2}
\leq\,& C\int_{0}^{t}\tau \|(\nabla u,\nabla B)(\tau))\|_{L^{\infty}}^{2} {\rm d}\tau\nonumber\\
\leq\,&C\int_{0}^{t} \|\tau^{\frac{1}{4}}\nabla^{2} (u,B)(\tau)\|_{L^{2}}\|\tau^{3/4}\nabla^{2} (u,B)(\tau)\|_{L^{6}} {\rm d}\tau \nonumber\\
\leq\,&C\|t^{1/4}\nabla^{2} (u,B)\|_{L_{t}^{2}(L^{2})}\|t^{3/4}\nabla^{2} (u,B)\|_{L_{t}^{2}(L^{6})}\nonumber\\
\leq\,&Ce^{C\|(u_{0},B_0)\|_{\dot{H}^{1/2}}^{2}}\|(u_{0},B_0)\|_{\dot{H}^{1/2}}^{2}. 
\end{align*}
Hence, \eqref{A4.55} holds.

Let us show the last estimate \eqref{A4.56}. Using the definition $D_t=\partial_t+u\cdot\nabla$, \eqref{A4.31} and Lemmas \ref{lemmabesov}, \ref{L2.3} and \ref{L2.5}, we obtain that
\begin{align}
&\|t^{1/2}(D_{t} u_{j}, D_{t} B_j)\|_{L_{t}^{\infty}(L^{2})}+\|t^{1/2} (\nabla D_{t} u_{j}, \nabla D_{t} B_j)\|_{L_{t}^{2}(L^{2})} \nonumber\\
\leq\,& C\|t^{1/2}(\partial_{t}u_{j},\partial_t B_j)\|_{L_{t}^{\infty}(L^{2})}+C\|t^{1/2}(u\cdot\nabla u_{j},u\cdot\nabla B_{j})\|_{L_{t}^{\infty}(L^{2})}\nonumber\\
& +C\|t^{1/2}(\nabla\partial_{t}u_{j},\nabla\partial_{t}B_j)\|_{L_{t}^{2}(L^{2})}+C\|t^{1/2}(\nabla(u\cdot\nabla)u_{j},\nabla(u\cdot\nabla)B_{j})\|_{L_{t}^{2}(L^{2})}  \nonumber\\
\leq\,& C2^{j}\|\dot{\Delta}_{j}(u_{0},B_0)\|_{L^{2}}+C\|u\|_{L_{t}^{\infty}(L^{3,\infty})}\|t^{1/2}(\nabla u_{j},\nabla B_j)\|_{L_{t}^{\infty}(L^{6,2})}\nonumber \\
& +C\|u\|_{L_{t}^{\infty}(L^{3})}\|t^{1/2}(\nabla^{2}u_{j}, \nabla^{2} B_j)\|_{L_{t}^{2}(L^{6})}+C\|\nabla u\|_{L_{t}^{2}(L^{3,\infty})}\|t^{1/2}(\nabla u_{j},\nabla B_j)\|_{L_{t}^{\infty}(L^{6,2})}\nonumber \\
\leq\, &C2^{j}\|\dot{\Delta}_{j}(u_{0},B_0)\|_{L^{2}}+C\|u\|_{L_{t}^{\infty}(\dot{B}_{2,\infty}^{1/2})}\|t^{1/2}(\nabla^{2}u_{j}, \nabla^{2} B_j)\|_{L_{t}^{\infty}(L^{2})} \nonumber\\
& +C\|u\|_{L_{t}^{\infty}(\dot{H}^{1/2})}\|t^{1/2}(\nabla^{2}u_{j}, \nabla^{2} B_j)\|_{L_{t}^{2}(L^{6})}+C\|\nabla u\|_{L_{t}^{2}(\dot{B}_{2,\infty}^{1/2})}\|t^{1/2}(\nabla^{2}u_{j}, \nabla^{2} B_j)\|_{L_{t}^{\infty}(L^{2})}\nonumber.
\end{align}
This, together with \eqref{A4.4} and \eqref{A4.6}, indicates that
\begin{align}\label{A4.44}
&\|t^{1/2}(D_{t} u_{j}, D_{t} B_j)\|_{L_{t}^{\infty}(L^{2})}+\|t^{1/2} (\nabla D_{t} u_{j}, \nabla D_{t} B_j)\|_{L_{t}^{2}(L^{2})} \nonumber \\
&\quad \leq Ce^{C\|(u_{0},B_0)\|_{\dot{H}^{1/2}}^{2}}  2^{j}\|\dot{\Delta}_{j}(u_{0},B_0)\|_{L^{2}}+C\|(u_{0},B_0)\|_{\dot{H}^{1/2}}\|t^{1/2}(\nabla^{2}u_{j}, \nabla^{2} B_j)\|_{L_{t}^{2}(L^{6})}.
\end{align}
To proceed, it follows from the standard elliptic estimates for \eqref{A2.42} that
\begin{align}\label{A4.45}
&\,\|t^{1/2}(\nabla^{2}u_{j},\nabla^2 B_j, \nabla P_{j})\|_{L_{t}^{2}(L^{6})}\nonumber\\
\leq&\, C\|t^{1/2}\rho (\partial_{t}u_{j}+u\cdot\nabla u_{j})+B\cdot\nabla B_j\|_{L_{t}^{2}(L^{6})} \nonumber\\
&\,+C\|t^{1/2} (\partial_{t}B_{j}+u\cdot\nabla B_{j})+B\cdot\nabla u_j\|_{L_{t}^{2}(L^{6})}\nonumber\\
\leq&\,C\|t^{1/2}(\partial_{t}u_{j},\partial_t B_j)\|_{L_{t}^{2}(L^{6})}+C\|(u,B)\|_{L_{t}^{2}(L^{\infty})}\|t^{1/2}(\nabla u_{j},\nabla B_j)\|_{L_{t}^{\infty}(L^{6})} \nonumber\\
\leq&\,C\|t^{1/2}(\nabla \partial_{t} u_{j},\nabla \partial_{t} B_j)\|_{L_{t}^{2}(L^{2})}+C\|(u,B)\|_{L_{t}^{2}(L^{\infty})}\|t^{1/2}(\nabla^{2}u_{j}, \nabla^{2} B_j)\|_{L_{t}^{\infty}(L^{2})}.
\end{align}
Putting \eqref{5.35} and \eqref{A4.46} into \eqref{A4.45}, we arrive at
\begin{align*}
&\|t^{1/2}(\nabla^{2}u_{j},\nabla^2 B_j, \nabla P_{j})\|_{L_{T}^{2}(L^{6})}\nonumber\\
\leq\,&C\|t^{1/2}(\nabla \partial_{t} u_{j},\nabla \partial_{t} B_j)\|_{L_{t}^{2}(L^{2})}+C\|(u_{0},B_0)\|_{\dot{H}^{1/2}}\|t^{1/2}(\nabla^{2}u_{j}, \nabla^{2} B_j)\|_{L_{t}^{\infty}(L^{2})} \nonumber\\
\leq\,& C2^{j}\|\dot{\Delta}_{j}(u_{0},B_0)\|_{L^{2}}.
\end{align*}
This, combined with \eqref{A4.44}, yields
\begin{align*}
&\|t^{1/2}(D_{t} u_{j}, D_{t} B_j)\|_{L_{t}^{\infty}(L^{2})}+\|t^{1/2} (\nabla D_{t} u_{j}, \nabla D_{t} B_j)\|_{L_{t}^{2}(L^{2})} \nonumber \\
&\quad\leq Ce^{C\|(u_{0},B_0)\|_{\dot{H}^{1/2}}^{2}}  2^{j}\|\dot{\Delta}_{j}(u_{0},B_0)\|_{L^{2}}.
\end{align*}
Using Lemmas \ref{L2.4}--\ref{L2.5}, we further get
\begin{align}
&\,\|t^{3/4}(D_{t}u,D_{t}B)\|_{L_{t}^{\infty}(L^{2})}+\|t^{3/4}(\nabla D_{t}u,\nabla D_{t}B)\|_{L_{t}^{2}(L^{2})} \nonumber\\
\leq&\,C\|t^{3/4}(\partial_{t}u, \partial_{t}B)\|_{L_{t}^{\infty}(L^{2})}+C\|t^{3/4}(u,B)\cdot(\nabla u,\nabla B)\|_{L_{t}^{\infty}(L^{2})}\nonumber\\
&\,+C\|t^{3/4}(\nabla\partial_{t}u, \nabla\partial_{t}B)\|_{L_{t}^{2}(L^{2})}+C\|t^{3/4}\big(\nabla(u\cdot\nabla)u,\nabla(u\cdot\nabla)B\big)\|_{L_{t}^{2}(L^{2})} \nonumber\\
\leq&\,C\|t^{3/4}(\partial_{t}u, \partial_{t}B)\|_{L_{t}^{\infty}(L^{2})}+C\|(u,B)\|_{L_{t}^{\infty}(L^{3,\infty})}\|t^{3/4}(\nabla u,\nabla B)\|_{L_{t}^{\infty}(L^{6,2})} \nonumber\\
&\,+C\|(u,B)\|_{L_{t}^{\infty}(L^{3})}\|t^{3/4}(\nabla^2 u,\nabla^2 B)\|_{L_{t}^{2}(L^{6})}\nonumber\\
&\,+\|(\nabla u,\nabla B)\|_{L_{t}^{2}(\dot{H}^{1/2})}\|t^{3/4}(\nabla^2 u,\nabla^2 B)\|_{L_{t}^{\infty}(L^{2})} \nonumber\\
\leq&\,Ce^{C\|(u_{0},B_0)\|_{\dot{H}^{1/2}}^{2}} \|(u_{0},B_0)\|_{\dot{H}^{1/2}}+\|(u,B)\|_{L_{t}^{\infty}(\dot{B}_{2,\infty}^{1/2})}\|t^{3/4}(\nabla^2 u,\nabla^2 B)\|_{L_{t}^{\infty}(L^{2})} \nonumber\\
&\,+\|(u,B)\|_{L_{t}^{\infty}(\dot{H}^{1/2})}\|t^{3/4}(\nabla^2 u,\nabla^2 B)\|_{L_{t}^{2}(L^{6})}\nonumber\\
&\,+\|(\nabla u,\nabla B)\|_{L_{t}^{2}(\dot{H}^{1/2})}\|t^{3/4}(\nabla^2 u,\nabla^2 B)\|_{L_{t}^{\infty}(L^{2})},\nonumber
\end{align}
which together with \eqref{A2.2} and \eqref{A4.54} implies \eqref{A4.56}. Thus, we complete the proof of Proposition \ref{P4.3}.
\end{proof}

With Propositions \ref{P4.1}--\ref{P4.3} in hand, now we state the proof of Theorem \ref{T2.2}.

\vspace{1mm}
 
\begin{proof}[Proof of Theorem \ref{T2.2}] 
Similar to the proofs in Theorems \ref{T2.3}--\ref{T2.1}, we can obtain an approximate solution sequence $\{(\rho^{\varepsilon},u^{\varepsilon},B^{\varepsilon},\nabla P^{\varepsilon})\}_{0<\varepsilon<1}$ which are global in time. Furthermore, based on the uniform bounds obtained in Propositions \ref{P4.1}--\ref{P4.3} and a standard compactness argument, one can show that as $\varepsilon\rightarrow 0$, $(\rho^{\varepsilon},u^{\varepsilon},B^{\varepsilon},\nabla P^{\varepsilon})$ converges to the limit $(\rho,u,B,\nabla P)$, which is a global solution to the Cauchy problem \eqref{I-1}--\eqref{d}. Furthermore, according to Fatou's property, we can show that $(\rho,u,B,\nabla P)$ obeys \eqref{A1.8}--\eqref{1.22}. The property $u,B\in$ $C([0,\infty),\dot{H}^{1/2})$ can be directly obtained
by using a similar process to that in \cite[Page 18]{AGZ-arXiv-2024}. In addition, as in \eqref{3.64}--\eqref{3.67}, we can prove the uniqueness based on Theorem \ref{thmunique} and the bounds \eqref{A1.8}--\eqref{1.22}. For brevity, we omit the details. Thus, the proof of Theorem \ref{T2.2} is completed.
\end{proof}

\section{Uniqueness of solutions}

This section is devoted to the stability estimates in Proposition \ref{propunique} below, which, in particular, implies Theorem \ref{thmunique} on the uniqueness 
by taking the same initial data.

\begin{prop}\label{propunique}
Let $(\rho_{1},u_{1},B_{1},\nabla P_{1})$ and $(\rho_{2},u_{2},B_{2},\nabla P_{2})$ be two solutions to the system \eqref{I-1} on $[0,T]\times\mathbb{R}^3$ supplemented with the same initial density $\rho_0$ but with possibly different initial velocity and magnetic field $(u_{0,1},B_{0,1})$ and $(u_{0,2},B_{0,2})$. Denote
\begin{equation*}
    \begin{aligned}
    f_1(t):=\|t^{1/2}\nabla u_2(t)\|_{L^{\infty}}^2, \quad f_2(t):=\|(\nabla u_2,\nabla B_2)(t)\|_{L^2}^4+\|t^{3/4}\nabla D_t u_2(t)\|_{L^2}^2.
    \end{aligned}
\end{equation*}Then there exists a constant $C$ and a small time $T_0\in(0,T]$ depending on $\|(\rho_1,\rho_2)\|_{L^{\infty}_T(L^{\infty})}$  such that the following stability estimate holds{\rm:}
\begin{align}
&\sup_{t\in[0,T]}\Big\{\|\sqrt{\rho_{1}}(u_1-u_2)(t)\|_{L^2}^2+\|(B_1-B_2)(t)\|_{L^2}^2+t^{-3/2}\|(\rho_1-\rho_2)(t)\|_{\dot{W}^{-1,3}}^2\Big\}\nonumber\\
&\quad+ \int_0^T \|(\nabla (u_1-u_2),\nabla(B_1-B_2)(t)\|_{L^2}^2 \,{\rm d}t\nonumber\\
&\quad\quad\quad\leq C\exp\bigg\{C\int_0^T f_2(t)\,{\rm d}t\, \exp\bigg\{C\int_0^Tf_1(t)\,{\rm d}t\bigg\}\bigg\}\nonumber\\
&\quad\quad\quad\quad \times \Big\{\|\sqrt{\rho_{0}}(u_{0,1}-u_{0,2})\|_{L^2}^2+\| B_{0,1}-B_{0,2}\|_{L^2}^2\Big\}.\label{stability}
\end{align}

\end{prop}


\begin{proof}

Our arguments are inspired by the recent work \cite{HSWZ-arXiv-2024-08} concerning the inhomogeneous Navier-Stokes equations.
Note that $(\delta \rho, \delta u,\delta B)$ satisfies the error system
\begin{equation}\label{deltaMHD}
 \left\{\begin{aligned}
 &\partial_t\delta\rho=-\dive ( \rho_ 1\delta u+\delta \rho u_2),\\
    &\rho_1 (\partial_t\delta u+u_1\cdot\nabla\delta u)+\nabla\delta P- \Delta\delta u=-\delta\rho D_t u_2-\rho_{1}\delta u\cdot\nabla u_{2}+B_1\cdot\nabla\delta B+\delta B\cdot\nabla B_2 ,\\
    & \partial_t \delta B+u_1\cdot\nabla \delta B-\Delta \delta B=-\delta u\cdot\nabla B_2+B_1\cdot \nabla \delta u+\delta B\cdot \nabla u_2,\\
    &\dive \delta u=\dive \delta B=0.
\end{aligned}\right.
\end{equation}
We estimate $\delta\rho$, $\delta u$ and $\delta B$ in turn.

(1) {\emph{Estimate of $\delta\rho$}}.
First, we aim to show
\begin{equation}\label{deltarho}
\begin{aligned}
&\|\delta \rho(t)\|_{\dot{W}^{-1,3}}\leq Ce^{C\|t^{1/2}u_2\|_{L^2_T(L^{\infty})}^2} t^{3/4}\|\sqrt{\rho_1}\delta u\|_{L^{\infty}_t(L^2)}^{1/2}\|\nabla\delta u\|_{L^2_t(L^2)}^{1/2},\quad t\in(0,T].
\end{aligned}
\end{equation}
To achieve \eqref{deltarho}, for any function $\psi\in \dot{W}^{1,3/2} $ and $t\in(0,T]$, one considers the backward transport equation
\begin{equation}\label{Psi}
\begin{aligned}
&\partial_s\Psi(s,x)+u_2(s,x)\cdot\nabla \Psi(s,x)=0,\quad \Psi(t,x)=\psi(x),\quad (s,x)\in (0,t]\times \mathbb{R}^3.
\end{aligned}
\end{equation}
Noticing that  $$\|u_2\|_{L^1(s,t;L^{\infty})}\leq \|t^{1/2}u_2\|_{L^2_T(L^{\infty})}|\log t/s|^{1/2},$$
 when $\Psi\in C^1_c(\mathbb{R}^3)$ the existence of $\Psi$ for all $s\in(0,t]$ can be shown by the classical theorem for linear transport equations ({\!\!\cite{BCD-Book-2011}}). For $\Psi\in \dot{W}^{1,3/2} $, one can construct the solution $\Psi$ by performing an approximation argument and using the fact that $\dot{W}^{1,3/2} $ is a Banach space. Furthermore, the solution $\Psi$ to \eqref{Psi} obeys the estimate
\begin{align}\label{es:Psi}
\|\nabla \Psi(s)\|_{L^{3/2}}\leq\, &C \exp\bigg\{\int_s^t \|\nabla u_2(\tau)\|_{L^{\infty}}\,{\rm d}t'\bigg\} \|\nabla \psi\|_{L^{3/2}} \nonumber\\
\leq \, & C \exp\Big\{C\|t^{1/2}u_2\|_{L^2_T(L^{\infty})}|\log t/s|^{1/2}\Big\}\|\nabla \psi\|_{L^{3/2}}.
\end{align}
It thus follows from $\eqref{deltaMHD}_1$ and \eqref{Psi} that
\begin{align*} 
\int_{\mathbb{R}^3}\delta\rho(t,x)\psi(x)\, {\rm d}x&\,=\int_{\mathbb{R}^3}\delta\rho(t,x)\Psi(t,x)\, {\rm d}x\\
& \,=\int_0^t \frac{{\rm d}}{{\rm d}s}\langle \delta \rho (s), \Psi(s)\rangle {\rm d}s\\
&\,=\int_0^t \langle \rho_1 \delta u(s),\nabla\Psi(s)\rangle {\rm d}s.
\end{align*}
This, combined with \eqref{es:Psi}, $\rho_1\in L^{\infty}_T(L^{\infty})$, and the fact that
\begin{equation*}
\begin{aligned}
\int_0^t e^{C\|t^{1/2}u_2\|_{L^2_T(L^{\infty})}|\log t/s|} {\rm d}s&\,=t\int_0^1 e^{C\|t^{1/2}u_2\|_{L^2_T(L^{\infty})}|\log \tau |} {\rm d}\tau\\
&\,\leq Ct e^{C\|t^{1/2}u_2\|_{L^2_T(L^{\infty})}^2}\int_0^1e^{-\frac{1}{2}\log \tau}{\rm d}\tau\leq  Ct e^{C\|t^{1/2}u_2\|_{L^2_T(L^{\infty})}^2},
\end{aligned}
\end{equation*}
gives rise to
\begin{equation*}
\begin{aligned}
\int_{\mathbb{R}^3}\delta\rho(t,x)\psi(x)\, {\rm d}x&\leq \|\nabla\Psi\|_{L^{\frac{4}{3}}(0,t;L^{3/2})} \|\rho_1\delta u\|_{L^4(0,t;L^3)}\\
&\leq C e^{C\|t^{1/2}u_2\|_{L^2_T(L^{\infty})}^2} t^{3/4}  \|\psi\|_{L^{3/2}}\|\rho_1\delta u\|_{L^{\infty}_t(L^2)}^{1/2}\|\rho_1\delta u\|_{L^2_t(L^6)}^{1/2}\\
&\leq C e^{C\|t^{1/2}u_2\|_{L^2_T(L^{\infty})}^2} t^{3/4}  \|\sqrt{\rho_1}\delta u\|_{L^{\infty}_t(L^2)}^{1/2}\|\nabla \delta u\|_{L^2_t(L^2)}^{1/2}.
\end{aligned}
\end{equation*}
Thus, \eqref{deltarho} holds.

(2){\emph{Estimate of $\delta u$}}.
Taking the inner product of $\eqref{deltaMHD}_{1}$ by $\delta u$ and integrating the result over $[0,t]$, we get
\begin{align}
&\frac{1}{2}\|\sqrt{\rho_{1}}\delta u(t)\|_{L^2}^2+ \int_0^t \|\nabla \delta u(s)\|_{L^2}^2{\rm d}s\nonumber\\
&\quad=\frac{1}{2}\|\sqrt{\rho_{0}}(u_{0,1}-u_{0,2})\|_{L^2}^2+\int_0^t\int_{\mathbb{R}^3} (B_1\cdot\nabla\delta B)\cdot \delta u \, {\rm d}x{\rm d}s+\delta I_1+\delta I_2+\delta I_3,\label{ddtdeltau}
\end{align}
with
\begin{equation}\nonumber
\begin{aligned}
\delta I_1&:=-\int_0^t\int_{\mathbb{R}^3} \rho_1(\delta u\cdot \nabla u_2)\cdot \delta u \, {\rm d}x{\rm d}s,\\
\delta I_2&:= \int_0^t\int_{\mathbb{R}^3}(\delta B\cdot\nabla B_2 )\cdot \delta u\, {\rm d}x{\rm d}s,\\
\delta I_3&:=-\int_0^t\int_{\mathbb{R}^3} \delta\rho D_s u_2\cdot \delta u \, {\rm d}x{\rm d}s.
\end{aligned}
\end{equation}
The terms $\delta I_i$ ($i=1,2,3$) are analyzed as follows. Due to H\"older's and Sobolev's inequalities and the upper bound of $\rho_1$, it is clear that
\begin{align}\label{J1}
\delta I_1&\leq \int_0^t \|\rho_1\delta u(s)\|_{L^2}^{1/2}\|\rho_1\delta u(s)\|_{L^6}^{1/2} \|\delta u(s)\|_{L^6}\|\nabla u_2(s)\|_{L^{2}}{\rm d}s \nonumber\\
&\leq \frac{1}{4}\int_0^t \|\nabla \delta u(s)\|_{L^2}^2+C\int_0^t \|\nabla u_2(s)\|_{L^2}^2\|\sqrt{\rho_1}\delta u(s)\|_{L^2}^2{\rm d}s.
\end{align}
Similar, it holds that
\begin{align}
\delta I_2&\leq \int_0^t \|\nabla B_2(s)\|_{L^2}\|\delta B(s)\|_{L^2}^{1/2}\|\delta B(s)\|_{L^6}^{1/2} \|\delta u(s)\|_{L^6} {\rm d}s\nonumber\\
&\leq \frac{1}{4}\int_0^t \|(\nabla \delta u, \nabla \delta B)(s)\|_{L^2}^2 {\rm d}s+ C\int_0^t\|\nabla B_2(s)\|_{L^2}^4\|\delta B(s)\|_{L^2}^2{\rm d}s. \label{J2}
\end{align}
To handle $\delta J_3$, we deduce from \eqref{deltarho} that
\begin{equation}\nonumber
\begin{aligned}
\delta I_3&=\int_0^t \|\delta \rho(s)\|_{\dot{W}^{-1,3}}\| D_t u_2 \cdot\delta u(s)\|_{\dot{W}^{1,3/2}}\,{\rm d}s \\
&\leq Ce^{C\|t^{1/2}u_2\|_{L^2_T(L^{\infty})}^2}\int_0^t \|\sqrt{\rho_1}\delta u\|_{L^{\infty}_s(L^2)}^{1/2}\|\nabla\delta u\|_{L^2_s(L^2)}^{1/2}\|s^{3/4}\nabla( D_s u_2\cdot \delta u )(s)\|_{L^{3/2}}\,{\rm d}s.
\end{aligned}
\end{equation}
Here, the standard embedding theorem yields
\begin{equation}\nonumber
\begin{aligned}
\|s^{3/4}\nabla( D_s u_2\cdot \delta u )\|_{L^{3/2}}&\leq \|s^{3/4}\nabla D_s u_2\|_{L^2}\|\delta u\|_{L^6}+C\|s^{3/4}D_s u_2\|_{L^6} \|\nabla\delta u\|_{L^2}\\
&\leq C\|s^{3/4}\nabla D_s u_2\|_{L^2}\|\nabla\delta u\|_{L^2}.
\end{aligned}
\end{equation}
Thus, we discover that
\begin{align}
\delta I_3&
\leq \frac{1}{4}\int_0^t\|\nabla\delta u(s)\|_{L^2}^2\,{\rm d}s\nonumber\\
&\quad+Ce^{C\|t^{1/2}u_2\|_{L^2_T(L^{\infty})}^2}\int_0^t\|s^{3/4}\nabla D_s u_2(s)\|_{L^2}^2( \|\sqrt{\rho_1}\delta u\|_{L^{\infty}_s(L^2)}^2+\|\nabla\delta u\|_{L^2_s(L^2)}^2)\,{\rm d}s.\label{J3}
  \end{align}
Substituting \eqref{J1}, \eqref{J2} and \eqref{J3} into \eqref{ddtdeltau} yields
\begin{align}\label{deltau}
&\frac{1}{2}\|\sqrt{\rho_{1}}\delta u(t)\|_{L^2}^2+ \frac{1}{2}\int_0^t \|\nabla \delta u(s)\|_{L^2}^2\,{\rm d}s\nonumber\\
&\quad\leq \frac{1}{2}\|\sqrt{\rho_{0}}(u_{0,1}-u_{0,2})\|_{L^2}^2+\int_0^t\int_{\mathbb{R}^3} (B_1\cdot\nabla\delta B)\cdot \delta u \, {\rm d}x{\rm d}s\nonumber\\
&\quad\quad+Ce^{C\|t^{1/2}u_2\|_{L^2_T(L^{\infty})}^2}\int_0^t\|s^{3/4}\nabla D_s u_2(s)\|_{L^2}^2( \|\sqrt{\rho_1}\delta u\|_{L^{\infty}_s(L^2)}^2+\|\nabla\delta u\|_{L^2_s(L^2)}^2)\,{\rm d}s\nonumber\\
&\quad\quad+\frac{1}{4}\int_0^t \|\nabla \delta B(s)\|_{L^2}^2{\rm d}s+ C\int_0^t\|\nabla B_2(s)\|_{L^2}^4\|\delta B\|_{L^{\infty}_s(L^2)}^2{\rm d}s.
\end{align}

(3){\emph{Estimate of $\delta B$}}.
We deduce from $\eqref{deltaMHD}_2$ that
\begin{equation}
\begin{aligned}
&\frac{1}{2}\|\delta B(t)\|_{L^2}^2+\int_0^t \|\nabla\delta B(s)\|_{L^2}^2\,{\rm d}s=\frac{1}{2}\| B_{0,1}-B_{0,2}\|_{L^2}^2+\delta J_1+\delta J_2+\delta J_3 \label{deltaB}
\end{aligned}
\end{equation}
with
\begin{equation*}
\begin{aligned}
\delta J_1&:=\int_0^t\int_{\mathbb{R}^3} (B_1\cdot\nabla\delta u)\cdot \delta B\, {\rm d}x{\rm d}s,\\
\delta J_2&:=\int_0^t\int_{\mathbb{R}^3} (\delta B \cdot\nabla u_2)\cdot \delta B \, {\rm d}x{\rm d}s,\\
\delta J_3&:= -\int_0^t\int_{\mathbb{R}^3} (\delta u\cdot \nabla B_2)\cdot \delta B\, {\rm d}x{\rm d}s.
\end{aligned}
\end{equation*}
We handle $\delta J_1$, $\delta J_2$ and $\delta J_3$ in turn. By integration by parts, it follows that
\begin{equation}\label{deltaB1}
\begin{aligned}
\delta J_1=-\int_0^t\int_{\mathbb{R}^3} (B_1\cdot\nabla\delta B)\cdot \delta u\, {\rm d}x{\rm d}s,
\end{aligned}
\end{equation}
which can be cancelled by the first term on the right-hand side of  \eqref{deltau}.  

For $\delta J_2$, it is observed that
\begin{align}
\delta J_2&=-\int_0^t\int_{\mathbb{R}^3} (\nabla\delta B \cdot u_2)\cdot \delta B \, {\rm d}x{\rm d}s\nonumber\\
&\leq \int_0^t \|u_2(s)\|_{L^{6}}  \|\delta B(s)\|_{L^3} \|\nabla\delta B(s)\|_{L^2}{\rm d}s\nonumber\\
&\leq C \int_0^t \|\nabla u_2(s)\|_{L^{2}}\|\nabla \delta B(s)\|_{L^2}^{3/2} \|\delta B(s)\|_{L^2}^{1/2}{\rm d}s\nonumber\\
&\leq \frac{1}{4}\int_0^t \|\nabla\delta B(s)\|_{L^2}^2\,{\rm d}s +C\int_0^t \|\nabla u_2(s)\|_{L^{2}}^4 \|\delta B(s)\|_{L^2}^2\,{\rm d}s.\label{deltaB2}
\end{align}
In addition, one also has
\begin{align}
\delta J_3&\leq \int_0^t \|\delta u(s)\|_{L^6}\|\delta B(s)\|_{L^2}^{1/2}\|\delta B(s)\|_{L^6}^{1/2} \|\nabla B_2(s)\|_{L^2}\,{\rm d}s\nonumber\\
&\leq \frac{1}{4}\int_0^t \|(\nabla\delta B,\nabla\delta u)(s)\|_{L^2}^2\,{\rm d}s +C\int_0^t \|\nabla B_2(s)\|_{L^{2}}^4 \|\delta B(s)\|_{L^2}^2\,{\rm d}s.\label{deltaB3}
\end{align}

Now we denote
\begin{equation*}
\begin{aligned}
\delta E(t):=\|\sqrt{\rho_{1}}\delta u\|_{L^{\infty}_t(L^2)}^2+ \|\delta B \|_{L^{\infty}_t(L^2)}^2+\int_0^{T_0} \|(\nabla \delta u,\nabla\delta B)(s)\|_{L^2}^2\,{\rm d}t.
\end{aligned}
\end{equation*}
Collecting \eqref{deltau} and \eqref{deltaB}--\eqref{deltaB3} up, we infer that, for any $t\in [0,T]$, 
\begin{equation*}
\begin{aligned}
\delta E(t)&\leq C(\|\sqrt{\rho_{0}}(u_{0,1}-u_{0,2})\|_{L^2}^2+\| B_{0,1}-B_{0,2}\|_{L^2}^2)\\
&\quad+Ce^{C\|t^{1/2}u_2\|_{L^2_T(L^{\infty})}^2}\int_0^t\Big(\|(\nabla u_2,\nabla B_2)(s)\|_{L^2}^4+\|s^{3/4}D_s u_2(s)\|_{L^2}^2\Big) \delta E(s)\,{\rm d}s.
\end{aligned}
\end{equation*}
Taking advantage of Gronwall's inequality, we have 
\begin{align}\label{findelta}
\delta E(t)\leq\, &  C\exp\bigg\{C\int_0^T f_2(t)\,{\rm d}t\, \exp\bigg\{C\int_0^Tf_1(t)\,{\rm d}t\bigg\}\bigg\}\nonumber\\
&\times \Big\{\|\sqrt{\rho_{0}}(u_{0,1}-u_{0,2})\|_{L^2}^2+\| B_{0,1}-B_{0,2}\|_{L^2}^2\Big\}.
\end{align}
Combining \eqref{findelta} and \eqref{deltarho} up, we get \eqref{stability} and finish the proof of Proposition \ref{propunique}.
\end{proof}

\appendix

\section{Technical lemmas}
Now, we collect some lemmas that have been used frequently in this paper.

The first lemma concerns the classical Bernstein's inequalities.
\begin{lem}{\rm(\!\!\cite{BCD-Book-2011})}
Let $1\leq p\leq q\leq \infty$ and $k\in \mathbb{N}$. Define the ball $\mathcal{B}=\{\xi\in\mathbb{R}^{3}~| ~|\xi|\leq  R\}$ and the annulus $\mathcal{C}=\{\xi\in\mathbb{R}^{3}~|~ \lambda r\leq |\xi|\leq R\}$ with two constants $r>0$ and $R>r$. For any $u\in L^p$ and $\lambda>0$, it holds
\begin{align}
{\rm{Supp}}~& \mathcal{F}(u) \subset \lambda \mathcal{B} \Rightarrow \|D^{k}u\|_{L^q}\lesssim\lambda^{k+3(\frac{1}{p}-\frac{1}{q})}\|u\|_{L^p}, \nonumber\\
{\rm{Supp}}~ \mathcal{F}&(u) \subset \lambda \mathcal{C} \Rightarrow \lambda^{k}\|u\|_{L^{p}}\lesssim\|D^{k}u\|_{L^{p}}\lesssim \lambda^{k}\|u\|_{L^{p}}.\nonumber
\end{align}
\end{lem}

The Besov spaces have the following properties.
\begin{lem}\label{lemmabesov}{\rm(\!\!\cite{BCD-Book-2011})}
The following properties hold in $\mathbb{R}^3${\rm:}
\begin{itemize}
\item{} For $s\in\mathbb{R}$, $1\leq p_{1}\leq p_{2}\leq \infty$ and $1\leq r_{1}\leq r_{2}\leq \infty$, it holds
\begin{equation}\notag
\begin{aligned}
\dot{B}^{s}_{p_{1},r_{1}}\hookrightarrow \dot{B}^{s-(\frac{1}{p_{1}}-\frac{1}{p_{2}})}_{p_{2},r_{2}}{\rm;}
\end{aligned}
\end{equation}
\item{} For $1\leq p\leq q\leq\infty$, we have the following chain of continuous embedding:
\begin{equation}\nonumber
\begin{aligned}
\dot{B}^{3/p-3/q}_{p,1}\hookrightarrow \dot{B}^{0}_{q,1}\hookrightarrow L^{q}\hookrightarrow \dot{B}^{0}_{q,\infty}{\rm;}
\end{aligned}
\end{equation}
\item{} If $p<\infty$, then $\dot{B}^{d/p}_{p,1}$ is continuously embedded in the set of continuous functions decaying to 0 at infinity{\rm;}
\item{}  For any $\varepsilon>0$, it holds that
\begin{equation}\nonumber
\begin{aligned}
H^{s+\varepsilon}\hookrightarrow \dot{B}^{s}_{2,1}\hookrightarrow \dot{H}^{s}{\rm;}
\end{aligned}
\end{equation}
\item{}
Let $\Lambda^{\sigma}$ be defined by $\Lambda^{\sigma}=(-\Delta)^{\frac{\sigma}{2}}u:=\mathcal{F}^{-}\big{(} |\xi|^{\sigma}\mathcal{F}(u) \big{)}$ for $\sigma\in \mathbb{R}$ and $u\in\dot{S}^{'}_{h}(\mathbb{R}^3)$, then $\Lambda^{\sigma}$ is an isomorphism from $\dot{B}^{s}_{p,r}$ to $\dot{B}^{s-\sigma}_{p,r}${\rm;}
\item{} Let $1\leq p_{1},p_{2},r_{1},r_{2}\leq \infty$, $s_{1}\in\mathbb{R}$ and $s_{2}\in\mathbb{R}$ satisfy
    $$
    s_{2}<3/p_2\quad\text{\text{or}}\quad s_{2}=3/p_2~\text{and}~r_{2}=1.
    $$
    The space $\dot{B}^{s_{1}}_{p_{1},r_{1}}\cap \dot{B}^{s_{2}}_{p_{2},r_{2}}$ endowed with the norm $\|\cdot \|_{\dot{B}^{s_{1}}_{p_{1},r_{1}}}+\|\cdot\|_{\dot{B}^{s_{2}}_{p_{2},r_{2}}}$ is a Banach space and has the weak compact and Fatou properties$:$ If $u_{n}$ is a uniformly bounded sequence of $\dot{B}^{s_{1}}_{p_{1},r_{1}}\cap \dot{B}^{s_{2}}_{p_{2},r_{2}}$, then an element $u$ of $\dot{B}^{s_{1}}_{p_{1},r_{1}}\cap \dot{B}^{s_{2}}_{p_{2},r_{2}}$ and a subsequence $u_{n_{k}}$ exist such that $u_{n_{k}}\rightarrow u $ in $\mathcal{S}'$ and
    \begin{equation}\nonumber
    \begin{aligned}
    \|u\|_{\dot{B}^{s_{1}}_{p_{1},r_{1}}\cap \dot{B}^{s_{2}}_{p_{2},r_{2}}}\lesssim \liminf_{n_{k}\rightarrow \infty} \|u_{n_{k}}\|_{\dot{B}^{s_{1}}_{p_{1},r_{1}}\cap \dot{B}^{s_{2}}_{p_{2},r_{2}}}.
    \end{aligned}
    \end{equation}
\end{itemize}
\end{lem}

The following interpolation inequalities in Besov spaces are also useful in our analysis.

\begin{lem}{\rm(\!\!\cite{BCD-Book-2011})}\label{L2.7}
Assume that $s_1$ and $s_2$ are real numbers satisfying $s_1<s_2$, and let $0<\theta<1$ and
$1\leq p,r\leq\infty$. Then there exists a constant $C$  depending only on $\theta$ and $s_2-s_1$, such that for any  tempered distribution $u$, it holds
\begin{align*}
\|u\|_{\dot{B}_{p,r}^{\theta s_1+(1-\theta)s_2}}\leq\|u\|_{\dot{B}_{p,r}^{s_1}}^{\theta}\|u\|_{\dot{B}_{p,r}^{s_2}}^{1-\theta},    
\end{align*}
and
\begin{align*}
\|u\|_{\dot{B}_{p,1}^{\theta s_1+(1-\theta)s_2}}\leq C\|u\|_{\dot{B}_{p,\infty}^{s_1}}^{\theta}\|u\|_{\dot{B}_{p,\infty}^{s_2}}^{1-\theta}.  
\end{align*} 
\end{lem}

There are some embedding relations with respect to Lorentz spaces and Besov spaces.

\begin{lem}{\rm(\!\!\cite[Proposition A.3]{AGZ-arXiv-2024})}\label{L2.6}

The following properties in $\mathbb{R}^3$ hold{\rm:}

\begin{itemize}
\item $H^{1} \hookrightarrow L^{6,2} ${\rm;}

\item $\dot{B}_{p,r}^{-1+3/p}\hookrightarrow L^{3,r} $, for $1\leq p<3$ and $1\leq r\leq\infty${\rm;}


\end{itemize}
    
\end{lem}

Some properties of Lorentz spaces are stated as follows:

\begin{lem}{\rm(\!\!\cite[Proposition A.1]{DW-CMMP-2023})}\label{D A.9}

$(i)$ Interpolation: For all $T>0$, $1\leq r,q\leq\infty$ and $\theta\in(0,1)$, there exists a constant $C$ such that for any  tempered distribution $u$, we have
\begin{align*}
\|u\|_{L^{p,r}_T(L^q)}&\leq C\|u\|_{L^{p_1}_T (L^q)}^{\theta} \| u\|_{L^{p_2}_T(L^q)}^{1-\theta}, 
\end{align*}
where $1<p_{1}<p<p_{2}<\infty$ are such that $1/p =(1-\theta)/p_{1}+\theta/p_{2}${\rm;}

$(ii)$ For any $k>0$, we have
$\|x^{-k}1_{\mathbb{R}_{+}}\|_{L^{1/k,\infty}}=1$.    
\end{lem}

\begin{lem} {\rm(\!\!\cite[Lemma 2.1]{AGZ-arXiv-2024})}\label{L2.3}
Let $f\in\dot{B}_{2,\infty}^{1/2}$ and
$g\in\dot{H}^1 $. It holds that
\begin{align*}
\|f g\|_{L^2}\leq C\|f\|_{\dot{B}_{2,\infty}^{1/2}}\|g\|_{\dot{H}^1}.
\end{align*}
\end{lem}

\begin{lem}{\rm(\!\!\cite[Lemma 2.2]{AGZ-arXiv-2024})}\label{L2.4}
Let $f\in\dot{B}_{2,\infty}^{1/2}\cap\dot{B}_{2,\infty}^{3/2} $. Then we have
\begin{align*}
\|f^2\|_{\dot{B}_{2,\infty}^{1/2}}\leq C\|f\|_{\dot{B}_{2,\infty}^{1/2}}\|f\|_{\dot{B}_{2,\infty}^{3/2}}.    
\end{align*}
\end{lem}

In the following, we recall some product laws in Lorentz and Besov spaces, which play a key role in controlling nonlinear terms. 

\begin{lem}{\rm(\!\!\cite{LR-Book-2002,Or-DM-1963})}\label{L2.5}
Let $\Omega$ be a domain in $\mathbb{R}^d$ with $d\geq1$, and let $f\in L^{p_1,q_1}(\Omega)$, $g\in L^{p_2,q_2}(\Omega)$ with $1\leq p,q,p_k,q_k\leq\infty$ for $k=1,2$.

$(i)$ If $1\leq q
_1\leq q_2\leq\infty$, we have the embeddings
\begin{align*}
L^{p,q_{1}}(\Omega)\hookrightarrow L^{p,q_{2}}(\Omega),\quad L^{p,p}(\Omega)=L^{p}(\Omega){\rm;}
\end{align*}

$(ii)$ If $1/p= 1/p_{1}+ 1/p_{2}$ and $1/q=1/q_{1}+1/q_{2}$, it holds that
\begin{align*}
\|fg\|_{L^{p,q}(\Omega)}\leq C\|f\|_{L^{p_{1},q_{1}}(\Omega)}\|g\|_{L^{p_{2},q_{2}}(\Omega)}{\rm;}   
\end{align*}

$(iii)$ If $1<p<\infty$, it holds that
\begin{align*}
\|fg\|_{L^{p,q}(\Omega)}\leq C\|f\|_{L^{p,q}(\Omega)}\|g\|_{L^{\infty}(\Omega)}.   
\end{align*}\end{lem}

To prove the global existence of solutions of the system \eqref{I-1}  in the $L^p$ framework, we recall two important lemmas pertaining to the maximal regularity estimates of the Stokes system and the heat equation with sources in Lebesgue spaces.

\begin{lem}{\rm(\!\!\cite[Proposition A.5]{DW-CMMP-2023})}\label{L2.8}
Let $1<p,q<\infty$ and $1\leq r\leq\infty$. Then, for any $u_{0}\in\dot{B}_{p,r}^{2-2/q}$ with $\dive  u_{0}=0$, and any $f\in L^{q,r}(0,T;L^{p})$, the following Cauchy problem of the Stokes system in $[0,T]\times\mathbb{R}^d${\rm:}
\begin{equation*}
\left\{
\begin{aligned}
&  \partial_t u-\Delta u+\nabla P=f,\\
& \dive u=0,\quad\quad\quad\quad\\
& u|_{t=0}=u_0,
\end{aligned}\right.
\end{equation*}
has a unique solution $(u,\nabla P)$ with $u\in{C}([0,T];\dot{B}_{p,r}^{2-2/q})$ and
\begin{align*}
&\|u\|_{L^{\infty}_T(\dot{B}_{p,r}^{2-2/q})}+\|(\partial_t u,\nabla^{2}u,\nabla P)\|_{L^{q,r}(0,T;L^{p})}\leq C\big(\|u_{0}\|_{\dot{B}_{p,r}^{2-2/q}}+\|f\|_{L^{q,r}(0,T;L^{p})}\big),
\end{align*}
where $C$ is a constant  independent of $T$, $u_0$ and $f$.

In addition, let $\widetilde{s}>q$ be such that
\begin{gather*}
\frac{1}{q}-\frac{1}{\widetilde{s}}\leq\frac{1}{2},\quad\frac{3}{2p}+\frac{1}{q}-\frac{1}{\widetilde{s}}> \frac{1}{2},
\end{gather*}
and define $\widetilde{m}\geq p$ by the relation
\begin{align*}
\frac{3}{2\widetilde m}+\frac{1}{\widetilde s}=\frac{3}{2p}+\frac{1}{q}-\frac{1}{2}.   
\end{align*}
Then, the following inequality holds:
\begin{align*}&\|\nabla u\|_{L^{\widetilde{s},r}(0,T;L^{\widetilde{m}})}
\leq C(\|u\|_{L^{\infty}_T(\dot{B}_{p,r}^{2-2/q})}+\|(\partial_t u,\nabla^{2}u)\|_{L^{q,r}(0,T;L^{p})}).
\end{align*}

Finally, if $2/q+3/p>2$, then for all $q<s<\infty$ and $p<m<\infty$ such that
\begin{align*}
\frac 3{2m}+\frac1s=\frac 3{2p}+\frac1q-1,    
\end{align*}
then it holds that
\begin{align*}&\|u\|_{L^{s,r}(0,T;L^{m})}\leq C\big(\|u\|_{L^{\infty}_T(\dot{B}_{p,r}^{2-2/q})}+\|(\partial_t u,\nabla^{2}u)\|_{L^{q,r}(0,T;L^{p})}\big).
\end{align*}
\end{lem}

\vspace{3mm}

\begin{lem}{\rm({\!\!\!\cite[Proposition 2.1]{DMT-JEE-2021})}} \label{L2.9}
Let $1< p, q< \infty$ and $1\leq r\leq \infty$. Then, for any $u_0\in \dot{B} _{p, r}^{2- 2/ q}$
and $f\in L^{q, r}(0,T;L^{p})$, the following Cauchy problem of the heat equation in $[0,T]\times\mathbb{R}^d${\rm:}
\begin{equation*}
\left\{
\begin{aligned}
&  \partial_t u -\Delta u=f,\\
& \dive u=0,\\
& u|_{t=0}=u_0,
\end{aligned}\right.
\end{equation*}
has a unique solution $u$ satisfying $u\in C([0,T];\dot{B}^{2-2/q}_{p,r})$ and
\begin{align*}
&\|u\|_{L^{\infty}(\dot{B}_{p,r}^{2-2/q})}+\|(\partial_t u,\nabla^{2}u)\|_{L^{q,r}(L^{p})}
\leq C\big(\|u_{0}\|_{\dot{B}_{p,r}^{2-2/q}}+\|f\|_{L^{q,r}(L^{p})}\big),
\end{align*}
for a constant $C$ independent of $T$, $u_0$ and $f$.

Furthermore, if $2/q + 3/p>2$, then for all $q < s < \infty$ and $p<m<\infty$ such that
\begin{align*}
\frac 3{2m}+\frac1s=\frac 3{2p}+\frac1q-1,    
\end{align*}
it holds that
\begin{align*}
&\|u\|_{L^{s,r}(L^{m})}
\leq C\big(\|u\|_{L^{\infty}(\dot{B}_{p,r}^{2-2/q})}+\|(\partial_t u,\nabla^{2}u)\|_{L^{q,r}(L^{p})}\big).
\end{align*}   
\end{lem}

\bigskip 
{\bf Acknowledgements:}
The authors are indebted to Prof. Guilong Gui
and Dr. Tiantian Hao for their valuable
suggestions on this paper. Li and  Ni  are supported by NSFC (Grant Nos. 12331007, 12071212).  
And Li is also supported by the ``333 Project" of Jiangsu Province.
Shou is supported by NSFC (Grant No. 12301275).


\bibliographystyle{plain}

\begin{thebibliography}{aaa}

\bibitem{Ah-RMI-2007} H. Abidi,
 Équation de Navier-Stokes avec densité et viscosité variables dans l'espace critique,
{\it Rev. Mat. Iberoam.} {\bf23} (2) (2007) 537--586.

\bibitem{AG-ARMA-2021} H. Abidi, G. Gui, 
 Global well-posedness for the 2-D inhomogeneous incompressible Navier-Stokes system with large initial data in critical spaces,
{\it Arch. Ration. Mech. Anal.}  {\bf242} (3) (2021) 1533--1570.

\bibitem{AGZ-arXiv-2024} H. Abidi, G. Gui, P. Zhang, 
 Global refined Fujita-Kato solution of 3-D inhomogeneous incompressible Navier-Stokes equations with large density,
 arXiv: 2410.09386.

\bibitem{AH-AMBP-2007} H. Abidi, T. Hmidi, 
 Résultats d'existence dans des espaces critiques pour le système de la MHD inhomogène,
{\it Ann. Math. Blaise Pascal}  {\bf 14} (1) (2007) 103--148.

\bibitem{AP-AIF-2007} H. Abidi, M. Paicu, 
 Existence globale pour un fluide inhomogène,
{\it Ann. Inst. Fourier (Grenoble)}   {\bf 57} (3) (2007) 883--917.

\bibitem{AP-PRSESA-2008} H. Abidi, M. Paicu, 
Global existence for the magnetohydrodynamic system in critical spaces,
{\it Proc. Roy. Soc. Edinburgh Sect. A}   {\bf 138} (3) (2008) 447--476.

\bibitem{AF-Pa-2003} R. A. Adams, J. J. F. Fournier,
{\it Sobolev Spaces}, Second Edition,  Elsevier/Academic Press, Amsterdam, 2003.

\bibitem{BCD-Book-2011} H. Bahouri, J. Chemin, R. Danchin,
{\it Fourier Analysis and Nonlinear Partial Differential Equations,}
Grundlehren Math. Wiss, 343,
Springer, Heidelberg, 2011.

\bibitem {biskamp1} D. Biskamp, \emph{Nonlinear Magnetohydrodynamics}, Cambridge University Press, Cambridge, 1993.

\bibitem{CW-Adv-2011}  C. Cao, J. Wu, 
 Global regularity for the 2D MHD equations with mixed partial dissipation and magnetic diffusion,
{\it Adv. Math.}   {\bf 226} (2) (2011) 1803--1822.

\bibitem {ChenTan1} Q. Chen, Z. Tan, Y. Wang, Strong solutions to the incompressible magnetohydrodynamic equations, \emph{Math. Methods Appl. Sci.} {\bf 34} (1) (2011) 94–107.

\bibitem {chemin1} J.-Y. Chemin, D. S. McCormick, J. C. Robinson, J. L. Rodrigo, Local existence for the non-resistive MHD equations in Besov spaces, \emph{Adv. Math.} {\bf 286} (2016) 1--31.

\bibitem{CLX-DCDS-2016} F. Chen, Y. Li, H. Xu,
 Global solution to the 3D nonhomogeneous incompressible MHD equations with some large initial data,
{\it Discrete Contin. Dyn. Syst.} {\bf 36} (6) (2016) 2945--2967.

\bibitem{CGZ-KRM-2019} F. Chen, B. Guo, X. Zhai, 
Global solution to the 3-D inhomogeneous incompressible MHD system with discontinuous density,
{\it Kinet. Relat. Models}  {\bf 12} (1) (2019) 37--58.


\bibitem {chenmiaozhang1}  Q. Chen, C. Miao, Z. Zhang, The Beale–Kato–Majda criterion for the 3D magnetohydrodynamics
equations, \emph{Comm. Math. Phys}. {\bf 275} (2007) 861–872.





\bibitem {craig1} W. Craig, X. Huang, Y. Wang, Global wellposedness for the 3D inhomogeneous incompressible Navier-Stokes equations, \emph{J. Math. Fluid Mech.} {\bf 15} (2013) 747--758.

\bibitem{Dr-PRSES-2003} R. Danchin
Density-dependent incompressible viscous fluids in critical spaces,
{\it Proc. Roy. Soc. Edinburgh Sect. A}  {\bf 133} (6) (2003) 1311--1334.

\bibitem{Dr-2004-adv} R. Danchin,
Local and global well-posedness results for flows of inhomogeneous viscous fluids,
{\it Adv. Differential Equations}  {\bf 9} (2004) 353--386.


\bibitem{DM-CPAM-2012} R. Danchin, P.-B. Mucha, 
A Lagrangian approach for the incompressible Navier-Stokes equations with variable density,
{\it Comm. Pure Appl. Math.}   {\bf 65} (10) (2012) 1458--1480.

\bibitem{DMT-JEE-2021} R. Danchin, P.-B. Mucha, P. Tolksdorf, 
Lorentz spaces in action on pressureless systems arising from models of collective behavior,
{\it J. Evol. Equ.}  {\bf 21} (2021) 3103--3127.

\bibitem{DT-CPDE-2021}  R. Danchin, J. Tan, 
On the well-posedness of the Hall-magnetohydrodynamics system in critical spaces,
{\it Comm. Partial Differential Equations} {\bf 46} (1) (2021) 31--65.

\bibitem{DW-CMMP-2023} R. Danchin, S. Wang,
Global unique solutions for the inhomogeneous Navier-Stokes equations with only bounded density, in critical regularity spaces,
{\it Comm. Math. Phys.}  {\bf 399} (2023) 1647--1688.


\bibitem{davidson1} P. A. Davidson, \emph{An Introduction to Magnetohydrodynamics}, Cambridge University Press, Cambridge, England, 2001.

\bibitem {Desjardins1} B. Desjardins, C. Le Bris, Remarks on a nonhomogeneous model of magnetohydrodynamics, \emph{Differential Integral Equations} {\bf 11} (3) (1998) 377--394.


\bibitem{DL-IM-1998} R.-J. Diperna, P.-L. Lions,  Ordinary differential equations,
transport theory and Sobolev spaces,
{\it Invent. Math.}  {\bf 98} (1989) 511--547.

\bibitem{DL-ARMA-1972} G. Duvaut, J.-L. Lions, 
Inéquations en thermoélasticité et magnétohydrodynamique,
{\it Arch. Rational Mech. Anal.} {\bf 46} (1972) 241--279.




\bibitem {fujita1} H. Fujita, T. Kato, On the Navier-Stokes initial value problem I, \emph{Arch. Rational Mech. Anal.} {\bf 16} (1964) 269--315.

\bibitem {Gerbeau1} J. F. Gerbeau, C. Le Bris, Existence of solution for a density-dependent magnetohydrodynamic equation, \emph{Adv. Differential Equations} {\bf 2} (1997) 427--452.

\bibitem{Ggl-JFA-2014} G. Gui,
Global well-posedness of the two-dimensional incompressible magnetohydrodynamics system with variable density and electrical conductivity,
{\it J. Funct. Anal.}   {\bf 267} (5) (2014) 1488--1539.

\bibitem {Grafakos1} L. Grafakos, \emph{Classical Fourier analysis}, Springer, New York, third edition, 2014.

\bibitem{HSWZ-arXiv-2024-08} T. Hao, F. Shao, D. Wei, Z. Zhang,
 Global well posedness of inhomogeneous Navier-Stokes equations with
bounded density, 
arXiv: 2406.19907.

\bibitem {hexin1}  C. He, Z. Xin, Partial regularity of suitable weak solutions to the incompressible magnetohydrodynamic equations, \emph{J. Funct. Anal.} {\bf 227} (2005) 113--152.


\bibitem {heli1} C. He, J. Li and B. L\"u, Global well-posedness and exponential stability of 3D Navier-Stokes
equations with density-dependent viscosity and vacuum in unbounded domains, \emph{Arch. Ration. Mech. Anal.} {\bf 239} (2021) 1809--1835.



\bibitem {he1} L.-B. He, L. Xu, P. Yu, On global dynamics of three dimensional Magnetohydrodynamics: Nonlinear stability
of Alfvén waves, \emph{Ann. PDE} {\bf 4} 2018, Paper No. 5.


\bibitem{HPZ-ARMA-2013} J. Huang, M. Paicu, P. Zhang,  Global well-posedness to incompressible inhomogeneous fluid system with bounded density and non-Lipschitz 
velocity,
{\it Arch. Ration. Mech. Anal.}  {\bf 209} (2013) 631--682.

\bibitem{HQ-MMAS-2023} T. Huang, C. Qian, 
Global well-posedness of 3D incompressible inhomogeneous magnetohydrodynamic equations,
{\it Math. Methods Appl. Sci.}  {\bf 46} (2) (2023) 2906--2940.



\bibitem {HuangWang1} X. Huang, Y. Wang, Global strong solution to the 2D nonhomogeneous incompressible MHD system, {\it J. Differential Equations}   {\bf 254} (2013) 511--527.



\bibitem {Kazhihov} A. V. Kažihov, Solvability of the initial-boundary value problem for the equations of the motion of an inhomogeneous viscous incompressible fluid, \emph{Dokl. Akad. Nauk SSSR}, {\bf 216} (1974) 1008--1010.


\bibitem{LS-ZMSLOMIS-1975} O.-A. Ladyženskaja, V.-A. Solonnikov, Unique solvability of an initial and boundary value problem for viscous incompressible inhomogeneous fluids,
{\it J. Sov. Math.}   {\bf 9} (5) (1978) 697-749.


\bibitem {laudau1} L. D. Laudau, E. M. Lifshitz, Electrodynamics of Continuous Media, second ed., Pergamon, New
York, 1984.

\bibitem{LR-Book-2002} P.-G. Lemarié-Rieusset, 
{\it Recent developments in the Navier-Stokes problem},
Chapman $\&$ Hall/CRC Res. Notes Math, 431
Chapman $\&$ Hall/CRC, Boca Raton, FL, 2002.

\bibitem {LiWang1} X. Li, D. Wang, Global strong solution to the three-dimensional density-dependent incompressible magnetohydrodynamic flows, \emph{J. Differential Equations}  {\bf 251} (2011) 1580--1615.

\bibitem{Lions-book-1996} P.-L. Lions,
{\it Mathematical Topics in Fluid Mechanics}, Vol. 1.
Incompressible models,  
The Clarendon Press, Oxford University Press, New York, 1996.


\bibitem{LZ-CPAM-2014} F. Lin, P. Zhang, 
Global small solutions to an MHD-type system: The three-dimensional case,
{\it Comm. Pure Appl. Math.}  {\bf 67} (4) (2014) 531--580.

\bibitem{LXZ-JDE-2015} F. Lin, L. Xu, P. Zhang, 
Global small solutions of 2-D incompressible MHD system,
{\it J. Differential Equations}   {\bf 259} (10) (2015) 5440--5485.


\bibitem {LuXuZhong1} B. L\"u, Z. Xu and X. Zhong, Global existence and large time asymptotic behavior of strong solutions to the Cauchy problem of 2D density-dependent
magnetohydrodynamic equations with vacuum, \emph{J. Math. Pures Appl. (9)} {\bf 108} (2017) 41--62.



\bibitem{Or-DM-1963} R. O'Neil,
Convolution operators and $L^{(p,q)}$ spaces,
{\it Duke Math. J.} {\bf 30} (1963) 129--142.

\bibitem{RWXZ-JFA-2014} X. Ren, J. Wu, Z. Xiang, Z. Zhang, 
Global existence and decay of smooth solution for the 2-D MHD equations without magnetic diffusion,
{\it J. Funct. Anal.}   {\bf 267} (2) (2014) 503--541.

\bibitem {Schmidt1} P. Schmidt, On a magnetohydrodynamic problem of Euler type, \emph{J. Differential Equations} {\bf 74} (1988) 318--335.

\bibitem {Secchi1} P. Secchi, On the equations of ideal incompressible magnetohydrodynamics, \emph{Rend. Semin. Mat. Univ. Padova} {\bf 90} (1993)
103--119.

\bibitem{sermange1} M. Sermange, R. Temam, Some mathematical questions related to the MHD equations, \emph{Comm.
Pure Appl. Math.} {\bf 36} (1983) 635--664.




\bibitem {simon} J. Simon, Nonhomogeneous viscous incompressible fluids: Existence of velocity, density, and pressure, \emph{SIAM J. Math. Anal.} {\bf 21} (1990) 1093--1117.

\bibitem {Tan1} J. Tan, L. Zhang, The inhomogeneous incompressible Hall-MHD system with only bounded density, \emph{Sci. China Math.} (2024), Published online.


\bibitem {wu1} J. Wu, Bounds and new approaches for the 3D MHD equations, \emph{J. Nonlinear Sci.} {\bf 12} (2002) 395–413.

\bibitem {wu2} J. Wu, Regularity results for weak solutions of the 3D MHD equations, \emph{Discrete. Contin. Dynam.
Syst.} {\bf 10} (2004) 543--556.


\bibitem{Xh-JDE-2022}  H. Xu, 
 Maximal $L^1$ Regularity for solutions to inhomogeneous incompressible  Navier-Stokes Equations,
{\it J. Differential Equations} {\bf 335} (2022) 1--42.

\bibitem{ZY-JDE-2017} X. Zhai, Z. Yin, 
Global well-posedness for the 3D incompressible inhomogeneous Navier-Stokes equations and MHD equations,
{\it  J. Differential Equations}   {\bf 262} (3) (2017) 1359--1412.

\bibitem{Zp-Adv-2020} P. Zhang,
Global Fujita-Kato solution of 3-D inhomogeneous incompressible Navier-Stokes system,
{\it  Adv. Math.} {\bf 363} (2020), Paper No. 107007.

\bibitem {Zhou0} Y. Zhou, Remarks on regularities for the 3D MHD equations, \emph{Discrete. Contin. Dynam. Syst.} {\bf 12} (2005)   881--886.


\bibitem {Zhou1} Y. Zhou, J. Fan, Local well-posedness for the ideal incompressible density dependent magnetohydrodynamic equations, \emph{Commun. Pure Appl. Anal.} {\bf 9} (2010) 813--818.



\end{thebibliography}

\end{document}